\documentclass[11]{article}
\usepackage[utf8]{inputenc} \usepackage[T1]{fontenc}    \usepackage[english]{babel}

\usepackage{amsfonts}
\usepackage{nicefrac}       \usepackage{microtype}      \usepackage{lmodern}
\usepackage{amssymb,amsmath,amsthm}

\usepackage{stmaryrd}
\SetSymbolFont{stmry}{bold}{U}{stmry}{m}{n}

\usepackage{bbm,bm}
\usepackage{latexsym}
\usepackage{xcolor}

\usepackage{enumerate}
\usepackage{verbatim}
\usepackage{booktabs}       

\usepackage{url}            \usepackage[colorlinks=true]{hyperref}
\usepackage[numbers,sort&compress,square,comma]{natbib}
\usepackage[capitalise]{cleveref}

\usepackage{parskip}
\usepackage{geometry}
\usepackage[short]{optidef}
\usepackage{algorithmic,algorithm}

\usepackage{thmtools}

\declaretheorem{theorem}
\declaretheorem{corollary}
\declaretheorem{lemma}
\declaretheorem{claim}
\declaretheorem{conjecture}
\declaretheorem{proposition}

\declaretheorem{fact}

\declaretheoremstyle[qed=$\square$]{definitionwithend}
\declaretheorem{definition}

\declaretheorem[style=definitionwithend]{example}
\declaretheorem{remark}

\crefname{assumption}{Assumption}{Assumptions}
\crefname{setting}{Setting}{Setting}
\crefname{conjecture}{Conjecture}{Conjectures}
\crefname{fact}{Fact}{Facts}
\crefname{claim}{Claim}{Claims}

\newcommand{\abs}[1]{\ensuremath{\left\lvert #1 \right\rvert}}

\newcommand{\by}{\times}
 
\newcommand{\norm}[1]{\ensuremath{\left\lVert #1 \right\rVert}}
\newcommand{\ip}[1]{\ensuremath{\left\langle #1 \right\rangle}}
\newcommand{\grad}{\ensuremath{\nabla}}

\newcommand{\set}[1]{\left\{#1\right\}}

\def\R{{\mathbb{R}}}
\def\S{{\mathbb{S}}}

\def\cD{{\cal D}}

\DeclareMathOperator*{\argmin}{arg\,min}

\DeclareMathOperator{\tr}{tr}

\newcommand{\framedheader}[3]{
  \framebox[\textwidth]{
    \vbox{
      \vspace{2mm}
      \hbox to \textwidth {\hspace{1em}\today \hfill #1\hspace{1em}}
      \vspace{4mm}
      \hbox to \textwidth {\hfill \Large{#2} \hfill}
      \vspace{2mm}
    }
  }
  \vspace*{4mm}
}

\DeclareMathOperator{\rev}{rev}
 \usepackage{letltxmacro}
\LetLtxMacro\orgvdots\vdots
\LetLtxMacro\orgddots\ddots

\makeatletter
\DeclareRobustCommand\vdots{\mathpalette\@vdots{}}
\newcommand*{\@vdots}[2]{\sbox0{$#1\cdotp\cdotp\cdotp\m@th$}\sbox2{$#1.\m@th$}\vbox{\dimen@=\wd0 \advance\dimen@ -3\ht2 \kern.5\dimen@
\dimen@=\wd2 \advance\dimen@ -\ht2 \dimen2=\wd0 \advance\dimen2 -\dimen@
    \vbox to \dimen2{\offinterlineskip
      \copy2 \vfill\copy2 \vfill\copy2 }}}
\DeclareRobustCommand\ddots{\mathinner{\mathpalette\@ddots{}\mkern\thinmuskip
  }}
\newcommand*{\@ddots}[2]{\sbox0{$#1\cdotp\cdotp\cdotp\m@th$}\sbox2{$#1.\m@th$}\vbox{\dimen@=\wd0 \advance\dimen@ -3\ht2 \kern.5\dimen@
\dimen@=\wd2 \advance\dimen@ -\ht2 \dimen2=\wd0 \advance\dimen2 -\dimen@
    \vbox to \dimen2{\offinterlineskip
      \hbox{$#1\mathpunct{.}\m@th$}\vfill
      \hbox{$#1\mathpunct{\kern\wd2}\mathpunct{.}\m@th$}\vfill
      \hbox{$#1\mathpunct{\kern\wd2}\mathpunct{\kern\wd2}\mathpunct{.}\m@th$}}}}

\DeclareRobustCommand\bddots{\mathinner{\mathpalette\@bddots{}\mkern\thinmuskip
  }}
\newcommand*{\@bddots}[2]{\sbox0{$#1\cdotp\cdotp\cdotp\m@th$}\sbox2{$#1.\m@th$}\vbox{\dimen@=\wd0 \advance\dimen@ -3\ht2 \kern.5\dimen@
\dimen@=\wd2 \advance\dimen@ -\ht2 \dimen2=\wd0 \advance\dimen2 -\dimen@
    \vbox to \dimen2{\offinterlineskip
      \hbox{$#1\mathpunct{\kern\wd2}\mathpunct{\kern\wd2}\mathpunct{.}\m@th$}\vfill
      \hbox{$#1\mathpunct{\kern\wd2}\mathpunct{.}\m@th$}\vfill
      \hbox{$#1\mathpunct{.}\m@th$}}}}
\makeatother \usepackage{soul}

\usepackage[makeroom]{cancel}
\usepackage{subcaption}

\usepackage{tikz}
\usetikzlibrary{shapes}

\newcommand{\emptySchedule}{[\hspace{0.5em}]}
\newcommand{\fcomposable}{{f-composable}}
\newcommand{\scomposable}{{s-composable}}
\newcommand{\gcomposable}{{g-composable}}
\newcommand{\vr}{{1+\sqrt{2}}}

\DeclareMathOperator{\len}{len}

\begin{document}

\title{Composing Optimized Stepsize Schedules for Gradient Descent}
\author{
    Benjamin Grimmer\footnote{Johns Hopkins University, Department of Applied Mathematics and Statistics, \texttt{grimmer@jhu.edu}}
    \and
    Kevin Shu\footnote{California Institute of Technology, Computational and Mathematical Sciences, \texttt{kshu@caltech.edu}}
    \and
    Alex L.\ Wang\footnote{Purdue University, Daniels School of Business, \texttt{wang5984@purdue.edu}}
}
	\date{\today}
	\maketitle

\begin{abstract}
Recent works by \citet{altschuler2023accelerationPartII} and \citet{grimmer2024accelerated} have shown that it is possible to accelerate the convergence of gradient descent on smooth convex functions, even without momentum, just by picking special stepsizes.
    In this paper, we provide a general theory for composing stepsize schedules capturing all recent advances in this area and more. We propose three notions of ``composable'' stepsize schedules with elementary associated composition operations for combining them. From these operations, in addition to recovering recent works, we construct three highly optimized sequences of stepsize schedules. We first construct optimized stepsize schedules of every length generalizing the exponentially spaced silver stepsizes of~\cite{altschuler2023accelerationPartII}. We then construct highly optimized stepsizes schedules for minimizing final objective gap or gradient norm, improving on prior rates by constants and, more importantly, matching or beating the numerically computed minimax optimal schedules of~\cite{gupta2023branch}. We conjecture these schedules are in fact minimax (information theoretic) optimal. 
    Several novel tertiary results follow from our theory including recovery of the recent dynamic gradient norm minimizing short stepsizes of~\cite{rotaru2024exact} and extending them to objective gap minimization.
\end{abstract}

\section{Introduction}

We consider minimizing an $L$-smooth convex function $f : \R^d \rightarrow \R$. A classic algorithm for approximating the minimum value of $f$ is \emph{gradient descent}, which proceeds given oracle access to the gradient of $f$ and an initialization $x_0 \in \R^d$ by iteratively setting 
\[
x_{i+1} = x_i - \frac{h_i}{L}\nabla f(x_i),
\]
for a schedule of (normalized) stepsizes $h_i$, and for $i = 0, \dots, n-1$.

Rich new theory for gradient descent has been developed in recent years, largely enabled by the performance estimation problem (PEP) framework established in~\cite{drori2012PerformanceOF,taylor2017interpolation,taylor2017smooth,Teboulle2022}.
Classic theory has primarily considered settings where $h_i\in (0,2)$, which we refer to as the ``short stepsize'' regime. For such stepsizes, the objective value $f(x_i)$ is guaranteed to monotonically decrease.
If $f$ attains its minimum at some $x_\star \in \R^d$ with bounded distance from the initialization, $\|x_0-x_\star\|\leq D$, 
the best possible convergence rate for the value of the function at its final iterate is that $f(x_n) - f(x_\star) \le \frac{1}{4}\frac{LD^2/2}{n} + o\left(1/n\right)$.
This rate (and the optimal coefficient of $1/4$) is attained by the dynamic stepsizes of Teboulle and Vaisbourd~\cite{Teboulle2022} and is conjectured\footnote{Since this paper was initially submitted, this conjecture was proven in \cite{kim2024proofexactconvergencerate}. To maintain the historical record, we will continue to refer to this statement as a conjecture in the body of the paper.} to be attained by the constant stepsize schedule considered by~\citet{drori2012PerformanceOF}, differing only in the $o\left(1/n\right)$ term.

Strong indications that convergence rates strictly faster than $O(1/n)$ could be attained by allowing stepsizes larger than two were given by~\citet{gupta2023branch}. Therein, a branch-and-bound technique was applied to produce numerically minimax optimal stepsize schedules of length up to $n=25$. By minimax optimal for a fixed length $n$, we mean the stepsize schedule $h\in\R^n$ solves
\begin{equation} \label{eq:minimax}
	\min_{h \in \mathbb{R}^n} \max_{(f,x_0)\in \mathcal{F}_{L,D}} f(x_n) - \inf f,
\end{equation}
where the set $\mathcal{F}_{L,D}$ contains all considered problem instances $(f,x_0)$ defined by an $L$-smooth convex $f$ and initialization $x_0$ at most distance $D$ from a minimizer of $f$. At face value, this problem may appear intractable given the inner maximization is over the space of functions. However, leveraging the performance estimation techniques of~\cite{drori2012PerformanceOF,taylor2017interpolation}, this inner problem reduces to semidefinite programming and the overall minimax problem reduces to a nonconvex QCQP. From numerical branch-and-bound global solves of this QCQP for $n=1,2,\dots, 25$ (requiring use of the \texttt{MIT Supercloud}) and additional local solves up to $n=50$, \citet{gupta2023branch} conjectured an $O(1/n^{1.178})$ rate may be possible.

Stepsize schedules with provable big-O improvements were then developed concurrently by \citet{altschuler2023accelerationPartII} and~\citet{grimmer2023accelerated}. Specifically, \citet{altschuler2023accelerationPartII} showed that it is possible to achieve an even faster rate of $f(x_n) - f(x_\star) \leq\frac{LD^2/2}{n^{\log_2(1+\sqrt{2})}}$ whenever $n$ is one less than a power of two. This rate was later replicated using the stepsizes of~\citet{grimmer2023accelerated} with an improved constant in~\cite{grimmer2024accelerated}. Both of these works analyzed carefully crafted fractal stepsize schedules with lengths exactly one less than a power of two.

\paragraph{Our Contributions.}
We provide a principled approach to \emph{composing} stepsize schedules for gradient descent. These techniques and ideas motivate and unify recent literature on stepsize schedules for gradient descent.

In \cref{sec:def}, we identify three general families of structured stepsize schedules which can be composed with special composition operations.
Specifically, we define the \fcomposable{}, \gcomposable{}, and \scomposable{} families of stepsize schedules.
These are schedules that, respectively, provide well-behaved convergence guarantees on function value decrease, gradient norm reduction, and both simultaneously.
The empty schedule (in which no steps are taken) is an example of a schedule that is \fcomposable{}, \gcomposable{}, and \scomposable{}.
Along with these families, we also introduce composition operations: the f-join $\triangleright$, the g-join $\triangleleft$, and the s-join $\Join$.
These binary operations take two composable stepsize schedules $a$ and $b$ and produce a new composable stepsize schedule $[a,\mu,b]$, where $\mu \in \R$ is an additional stepsize depending on $a$, $b$, and the particular type of composition being applied.
Any schedule of stepsizes built from these operations
is immediately endowed with tight\footnote{Here, ``tight'' means that our theory produces a convergence rate proof and matching problem instances exactly attaining the proven convergence rates.} convergence rate theory. 
We say that such a schedule is \emph{basic} if it is built via these operations beginning only with the empty schedule.

In Section~\ref{sec:basic}, we show that basic schedules suffice to recover and analyze (i) the recent fractal stepsize schedules of~\cite{altschuler2023accelerationPartII} and~\cite{grimmer2024accelerated}, (ii) all 25 numerically computed minimax optimal schedules of \citet{gupta2023branch}, hence providing these schedules their first formal proofs, and (iii) a host of other schedules including the dynamic stepsizes recently proposed by~\cite{rotaru2024exact}. These results simplify and unify existing literature.

In Section~\ref{sec:schedules}, we show that the optimal \emph{basic} schedules of any given length can be computed easily, specifically via dynamic programming.
The Optimal Basic \scomposable{} Schedules (denoted OBS-S) provides a family generalizing the silver stepsizes of~\cite{altschuler2023accelerationPartII} defined for all $n$. The easily computed Optimal Basic \fcomposable{} Schedules (OBS-F) has worst case convergence rate $\Theta(1/n^{\log_2(1+\sqrt{2})})$ and improves on prior works by a constant factor.
We conjecture that the OBS-F schedules are minimax optimal in the sense of \eqref{eq:minimax}
(see Conjecture~\ref{conj:strong-f-minimax-descripiton} and Figure~\ref{fig:stepsizes}). Similarly, we prove best-known rates for minimizing gradient norm $\|\nabla f(x_n)\|$ via the Optimal Basic \gcomposable{} Schedules (OBS-G) and conjecture they are the minimax optimal schedules for minimizing the final gradient norm. 

As we will see, the three composition operations that we introduce are designed to (i) maintain fundamental symmetries between minimizing the final objective gap and the final gradient norm (i.e., between \fcomposable{} and \gcomposable{} schedules), and (ii) balance worst-case convergence guarantees on quadratic and Huber functions.
Although these two properties do not hold for arbitrary stepsize schedules, they have been observed for many \emph{optimized} first-order methods such as the Optimized Gradient Methods (OGM and OGM-G)~\cite{Kim2016optimal,Kim2021gradient,Kim2024-Hduality} and the conjectured optimal constant stepsize schedules for gradient descent~\cite{rotaru2024exact,drori2012PerformanceOF,grimmer2024strengthenedconjectureminimaxoptimal}.
The symmetry between final objective gap and final gradient norm guarantees is referred to as H-duality and has been proved for an important subset of momentum methods~\cite{Kim2024-Hduality} that do not cover the OBS-F and OBS-G schedules described in this paper.
We will prove that these symmetries and worst-case behaviors extend to OBS-F and OBS-G.

\paragraph{Fixed step gradient descent vs. other first-order methods.}
This paper is concerned with \emph{fixed step gradient descent methods}. Such methods are entirely specified by the number of steps $n$, and the stepsize schedule $h \in \R^n$. Importantly, this stepsize schedule is \emph{fixed} and may not depend on the function $f$ to be minimized.

Many important first-order methods for minimizing convex functions fall outside the class of \emph{fixed step} gradient descent methods.
For example, methods using momentum \cite{Nesterov1983,Kim2016optimal} are not fixed step gradient descent methods due to the inclusion of momentum terms.
These methods seem to be able to obtain convergence rates which are strictly faster than those of gradient descent with fixed step sizes.
Similarly, ``adaptive'' or ``dynamic'' gradient descent methods~\cite{malitsky2020adaptive}, where stepsizes may be chosen in response to the first-order information falls outside the class of fixed step gradient descent methods.

Despite the apparent simplicity of fixed step gradient descent methods, recent work has shown that there is a rich but poorly understood design space for choosing these stepsize schedules. A primary motivation of this work is that a clearer understanding of this setting may lead to improved design principles for first-order algorithms even in other settings.

\paragraph{A Note on Organization.}
Many of the results in this paper come in \fcomposable{}, \gcomposable{}, and \scomposable{} versions.
Thus, for the sake of readability, we present only the proofs for the \fcomposable{} versions of such statements in the main body and defer the analogous proofs for \gcomposable{} and \scomposable{} schedules to Appendix \ref{ap:deferred}.

In the final preparation of this manuscript, the authors became aware of the concurrent work of Zhang and Jiang~\cite{zhang2024acceleratedgradientdescentconcatenation}.  Therein, similar techniques are developed in slightly different terms of so-called ``primitive'', ``dominant'', and ``g-bounded'' schedules. Using a similar dynamical programming technique, they provide an equivalent alternative construction of the schedules we construct in Section~\ref{sec:schedules}. We highlight the main differences between these works in \cref{rem:compare_def,rem:compare_rates}.

\begin{figure}
	\includegraphics[height=2.7cm]{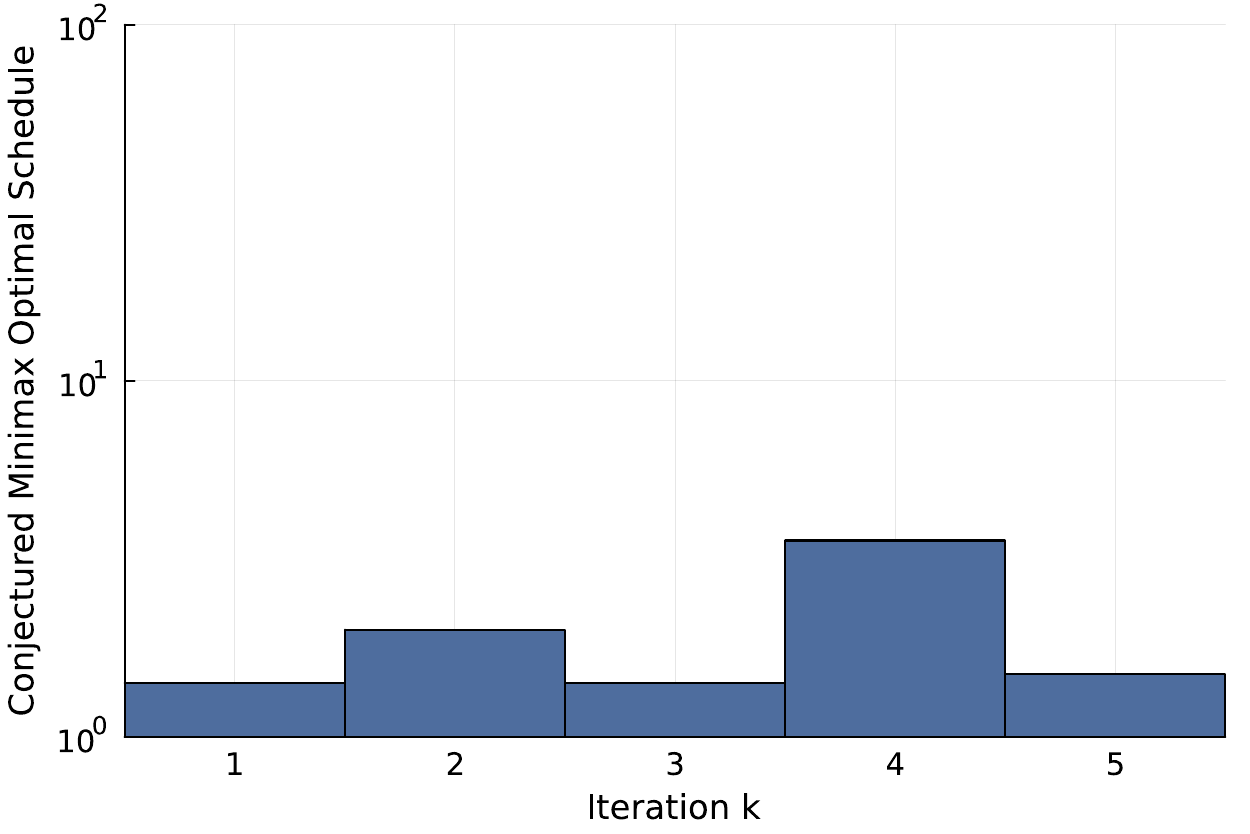}\includegraphics[height=2.7cm]{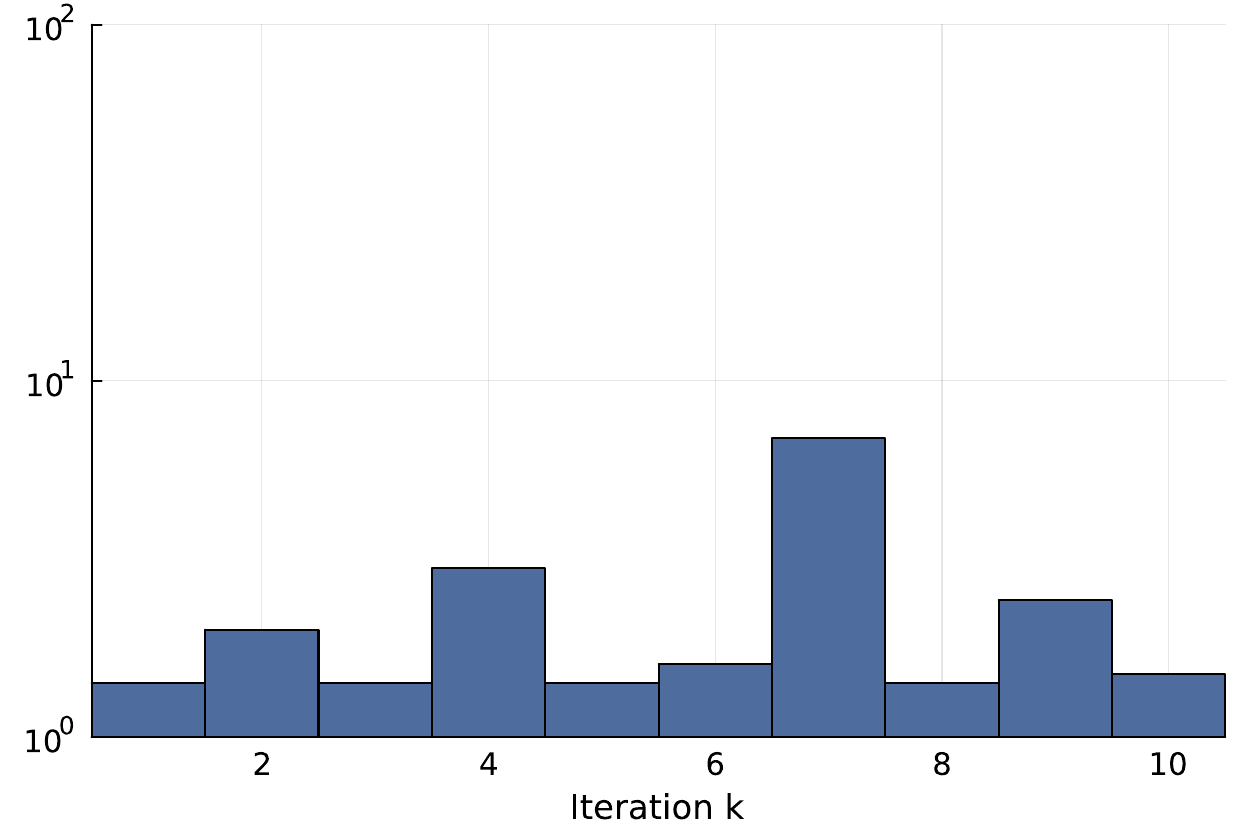}\includegraphics[height=2.7cm]{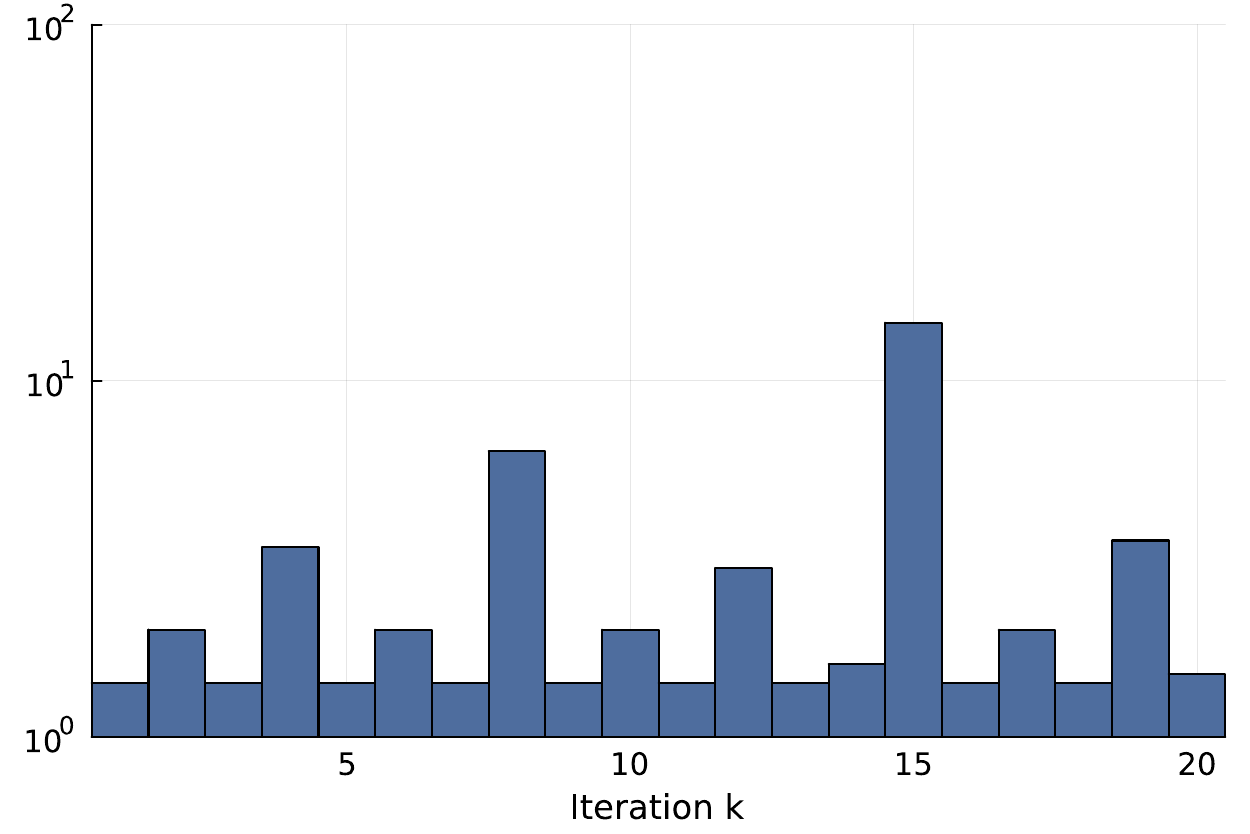}\includegraphics[height=2.7cm]{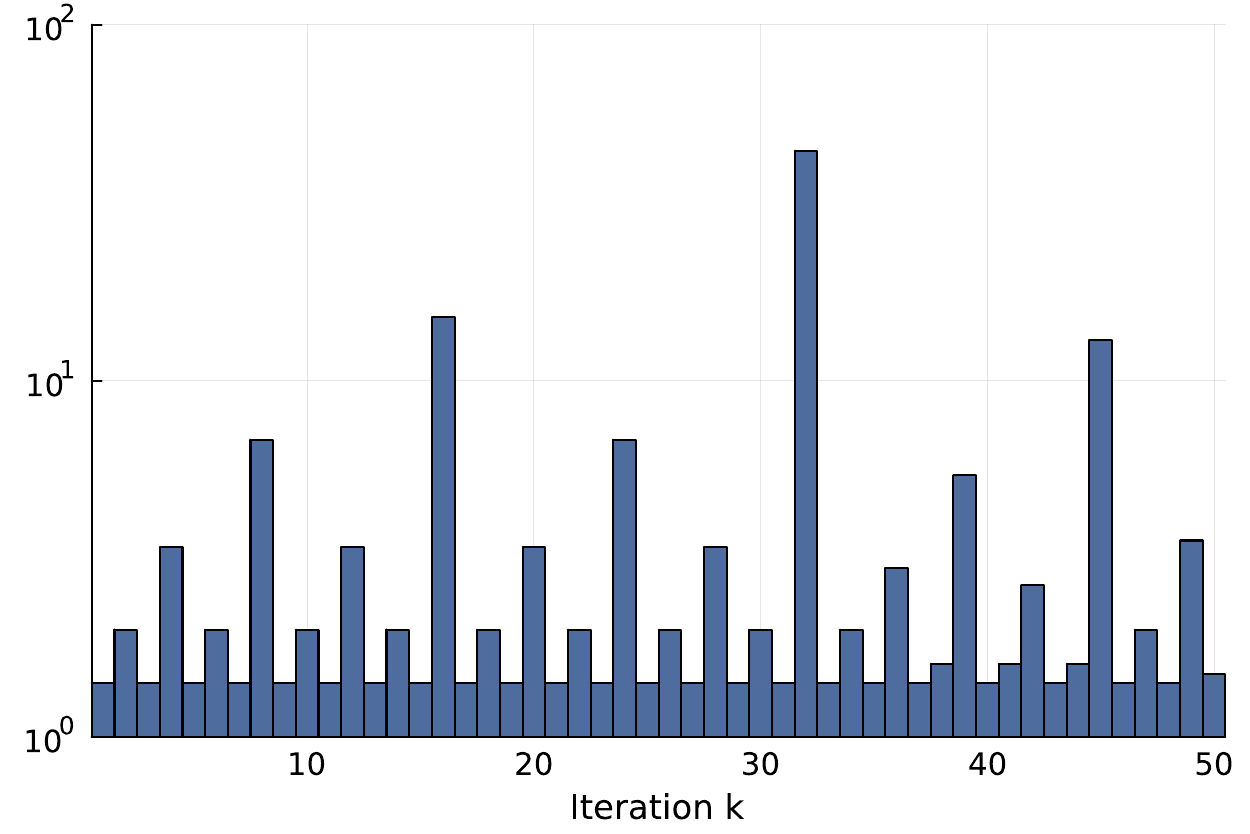}\\
	\includegraphics[height=2.7cm]{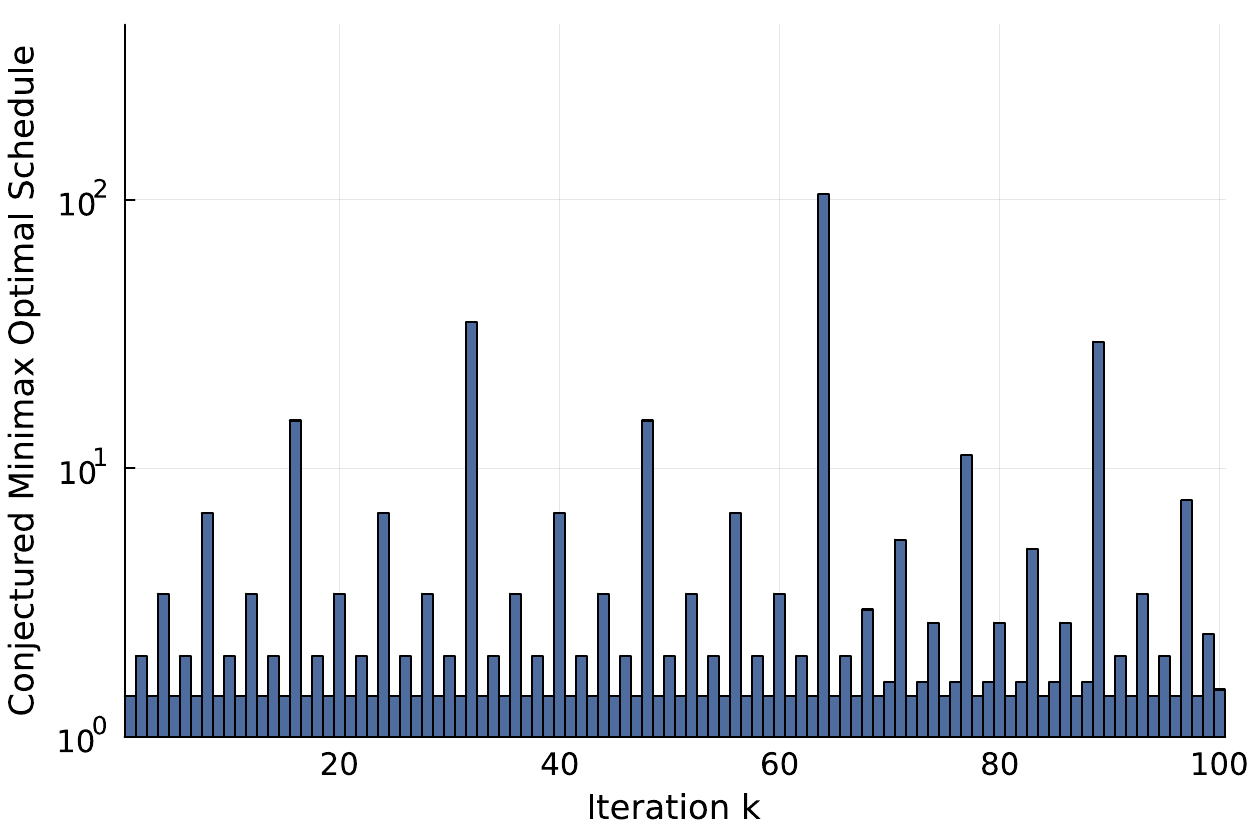}\includegraphics[height=2.7cm]{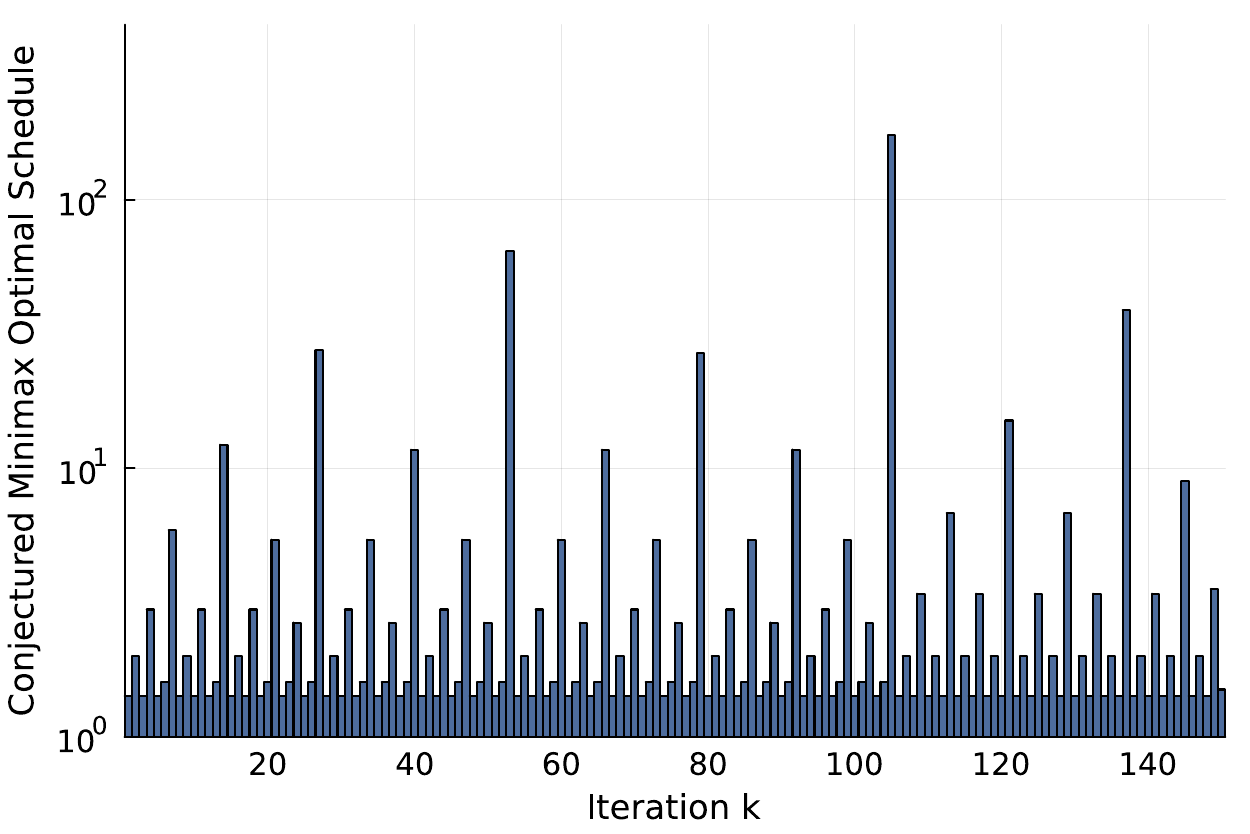}\includegraphics[height=2.7cm]{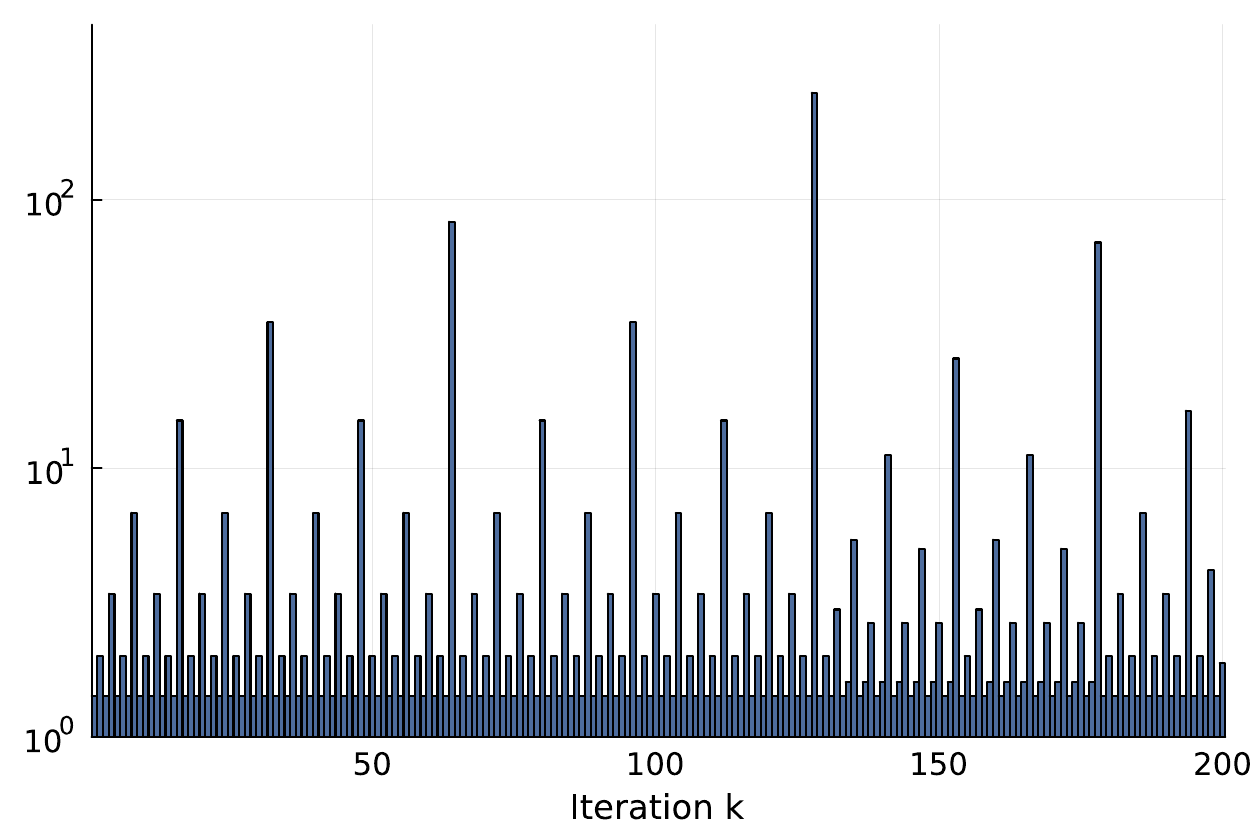}\includegraphics[height=2.7cm]{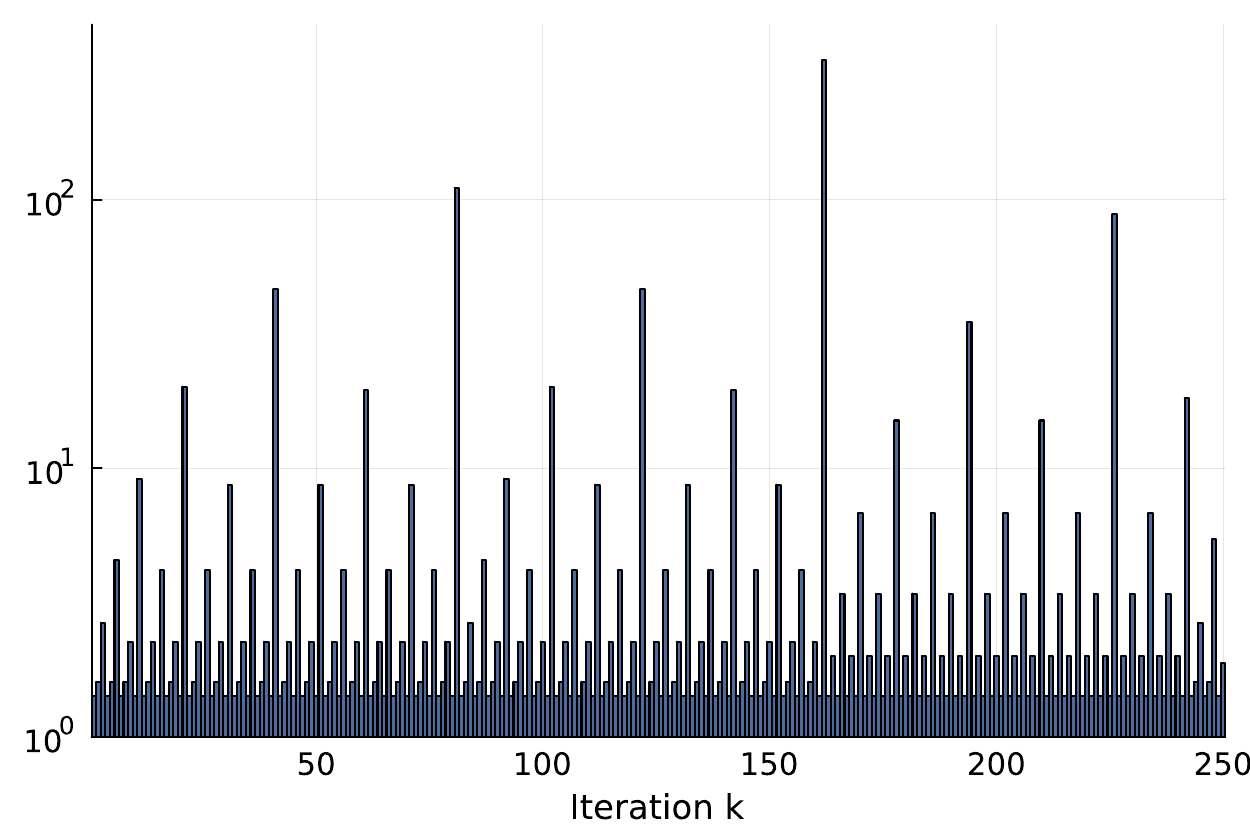}
	\caption{Conjectured Minimax Optimal Stepsize Schedules $h_k$ for sample values of $n=5\dots 250$, formally defined in Section~\ref{sec:obs-f}. Efficient code producing such optimized, OBS-F, schedules of any length $n$ is available at~\url{https://github.com/bgrimmer/OptimizedBasicSchedules}.}\label{fig:stepsizes}
\end{figure}
 \section{Composable Stepsize Schedules and Composition Operations} \label{sec:def}

In the remainder of the paper, we will assume that $L=1$, i.e., that the convex functions of interest are $1$-smooth. This simplifies notation and may be made without loss of generality.

We begin by introducing the \fcomposable{}, \gcomposable{}, and \scomposable{} stepsize schedules. The key motivating principle behind the design of these classes is to balance worst-case performance (i.e., ``hedge'' in the terminology of~\cite{altschuler2023accelerationPartI}) between the two extremes of functions with minimal and maximal curvature (limited by $f$ being $1$-smooth and convex), or more specifically, the one-dimensional quadratic function $q(x)$ and Huber functions $H_{\delta}(x)$ parameterized by $\delta$:
\begin{equation}
    \label{eq:quad_huber}
    q(x)= \frac{1}{2}x^2 \qquad\text{and} \qquad H_{\delta}(x) = \begin{cases}
	\frac{1}{2}x^2&\text{if }\abs{x}\leq \delta\\
	\delta\abs{x} - \frac{\delta^2}{2} & \text{else}
\end{cases},
\end{equation}

This principle has been well-established by prior works considering gradient decent with a constant stepsize schedule. Conjectures of~\cite{drori2012PerformanceOF,taylor2017interpolation} state that the optimal constant stepsize schedule is precisely the one balancing the final iterate's performance between these two functions (for an appropriately chosen $\delta$). In the case of minimizing the final iterate's gradient's norm, this conjecture was recently proven by~\cite{rotaru2024exact}. Modest progress on the conjecture for minimizing final objective gap was given by~\cite{grimmer2024strengthenedconjectureminimaxoptimal}.

\subsection{Definitions of composable Stepsize Schedules}
We now define our three classes of composable stepsize schedules. 
We begin with \fcomposable{} schedules.

Below, and in the remainder of the paper, the phrase ``gradient descent with stepsize schedule $h$'' refers to the algorithm that given $(f, x_0)$, produces $x_1, \dots, x_n$ by setting $x_i = x_{i-1} - h_{i-1} \nabla f(x_{i-1})$.
\begin{definition}
    Let $n\geq 0$ and let $h\in\R^n_{++}$ be indexed by $[0,n-1]$. We say that $h$ is \emph{\fcomposable{} with rate $\eta$} if gradient descent with stepsize schedule $h$ satisfies the following inequality for all $1$-smooth convex functions $f$, all minimizers $x_\star$ of $f$, and all $x_0 \in \R^d$:
	\begin{equation}
		\label{eq:f_composable_inequality}
		f(x_n)-f(x_\star) \leq \eta \frac{\norm{x_0 - x_\star}^2}{2}
	\end{equation}
	and moreover,
		$\eta = \frac{1}{1+2\sum_{i=0}^{n-1} h_i}  = \prod_{i=0}^{n-1} (h_i - 1)^2$.\qedhere
\end{definition}
The $1$-smooth convex function $f:\R^d\to\R$ above may be taken in \emph{any} ambient dimension $d$.

The second condition in the definition of an \fcomposable{} schedule is precisely that the inequality in the first condition is tight and witnessed by both quadratic and Huber functions. The following lemma verifies this.
\begin{lemma}
	\label{lem:f_huber_tight}
    Suppose $h\in\R^n_{++}$ is \fcomposable{} with rate $\eta$. If $x_0 = 1$ and $f(x)$ is either $q(x)$ or $H_{\eta}(x)$ (defined in \cref{eq:quad_huber}), then gradient descent with stepsize schedule $h$ satisfies \eqref{eq:f_composable_inequality} at equality.
\end{lemma}
\begin{proof}
	First, suppose $f(x)= \frac{1}{2}x^2$ and $x_0=1$. Then,
	$x_n = \prod_{i=0}^{n-1}(1-h_i)$
	and
$$f(x_n)-f(x_\star) = \frac{1}{2}\prod_{i=0}^{n-1}(1-h_i)^2 = \frac{\eta}{2} = \frac{\eta}{2}\norm{x_0-x_\star}^2.$$

	Next, suppose $f(x)$ is the Huber function described in the lemma statement. Note that:
$$1 - \eta \sum_{i=0}^{n-1}h_i = 1 - \eta\left(\frac{1}{2\eta}- \frac{1}{2}\right) = \frac{1+\eta}{2} \geq \eta.$$
Here the last inequality uses the fact that $\eta \leq 1$.
	Thus, we have that $x_n = \frac{1+\eta}{2}$ and
	\begin{equation*}
		f(x_n)-f(x_\star) = \eta\left(\frac{1+\eta}{2}\right)- \frac{\eta^2}{2} = \frac{\eta}{2} = \frac{\eta}{2}\norm{x_0-x_\star}^2.\qedhere
	\end{equation*}
\end{proof}

Next, we define \gcomposable{} schedules similarly. Again, the second condition in the below definition states that the first condition is tight and witnessed both by quadratic and Huber functions (proof in Appendix \ref{ap:deferred}).
\begin{definition}
	Let $n\geq 0$ and let $h\in\R^n_{++}$ be indexed by $[0,n-1]$. We say that $h$ is \emph{\gcomposable{} with rate $\eta$} if gradient descent with stepsize schedule $h$, on any $1$-smooth convex function $f$ with minimizer $x_\star$, satisfies 
	\begin{equation}
		\label{eq:g_composable_ineq}
		\frac{1}{2}\norm{\nabla f(x_n)}^2 \leq \eta (f(x_0)-f(x_\star))
	\end{equation}
	and moreover,
		$\eta = \frac{1}{1+2\sum_{i=0}^{n-1} h_i}  = \prod_{i=0}^{n-1} (h_i - 1)^2$.\qedhere
\end{definition}

\begin{lemma}
	\label{lem:g_huber_tight}
    Suppose $h\in\R^n_{++}$ is \gcomposable{} with rate $\eta$. If $x_0 = 1$ and $f(x)$ is either $q(x)$ or $H_{2\eta/(1+\eta)}(x)$ (defined in \cref{eq:quad_huber}), then gradient descent with stepsize schedule $h$ satisfies \eqref{eq:g_composable_ineq} at equality.
\end{lemma}

Finally, by the same principles, we define composable schedules that maintain an inequality intermediate between the inequalities for \fcomposable{} and \gcomposable{} schedules, which we call \scomposable{} schedules. Such schedules possess a nice self-duality, hence the denoting letter ``s''. Again, the second condition below states that the first condition is tight and witnessed both by quadratic and Huber functions (proof in Appendix \ref{ap:deferred}).
\begin{definition}
	Let $n\geq 0$ and let $h\in\R^n_{++}$ be indexed by $[0,n-1]$. We say that $h$ is \emph{\scomposable{} with rate $\eta$} if gradient descent with stepsize schedule $h$, on any $1$-smooth convex function $f$, satisfies 
	\begin{equation}
		\label{eq:self_composable_ineq}
		\frac{1-\eta}{2}\norm{\nabla f(x_n)}^2 +\frac{\eta^2}{2}\norm{x_n-x_\star}^2 + (\eta-\eta^2)(f(x_n)-f(x_\star)) \leq \frac{\eta^2}{2}\norm{x_0 - x_\star}^2
	\end{equation}
	and moreover,
		$\eta = \frac{1}{1+\sum_{i=0}^{n-1}h_i} = \prod_{i=0}^{n-1}(h_i - 1)$.\qedhere
\end{definition}
\begin{lemma}
    \label{lem:s_composable_huber_tight}
    Suppose $h\in\R^n_{++}$ is \scomposable{} with rate $\eta$. If $x_0 = 1$ and $f(x)$ is either $q(x)$ or $H_{\delta}(x)$ (defined in \cref{eq:quad_huber}) for some $\delta\leq \eta$, then gradient descent with stepsize schedule $h$ satisfies \eqref{eq:self_composable_ineq} at equality.
\end{lemma}
\begin{remark}
    Note that the definition of the rate $\eta$ for an \scomposable{} schedule differs from that of the \fcomposable{} and \gcomposable{} schedules, notably in that factors of 2 and squares are dropped.
\end{remark}

To further explain \eqref{eq:self_composable_ineq}, note that the left-hand-side expression in \eqref{eq:self_composable_ineq} is a convex combination of the three performance criteria: $f(x_n)-f(x_\star)$, $\frac{1}{2}\norm{x_n - x_\star}^2$, and $\frac{1}{2}\norm{\nabla f(x_n)}^2$.
Additionally, this inequality is a natural precondition for guaranteeing bounds on both the final suboptimality and gradient norm simultaneously:
\begin{proposition}
    \label{prop:fg_rates_of_scomp}
If $h\in\R^n_{++}$ is \scomposable{} with rate $\eta$, then for any $1$-smooth convex $f$, gradient descent with stepsize schedule $h$ satisfies
\begin{align*}
    f(x_n) - f(x_\star) &\leq \frac{1}{1+2\sum_{i=0}^{n-1}h_i}\left(\frac{1}{2}\norm{x_0 - x_\star}^2 - \frac{1}{2}\norm{x_n - x_\star - \frac{\grad f(x_n)}{\eta}}^2\right)\\
    \frac{1}{2}\norm{\nabla f(x_n)}^2 &\leq\frac{1}{1+2\sum_{i=0}^{n-1}h_i} \left(f(x_0) - f(x_\star)- \frac{1}{2}\norm{g_0 - \eta\sum_{i=0}^{n-1}h_i\grad f(x_i) - \eta \grad f(x_n)}^2\right).
\end{align*}
The bound on the first line is tight and is attained when $f(x)$ is either $q(x)$ or $H_{\eta/(2-\eta)}(x)$ (defined in \eqref{eq:quad_huber}). In the latter case, the bracketed term in the first line simplifies to $\frac{1}{2}\norm{x_0-x_\star}^2$.
Similarly, the bound on the second line is tight and is attained when $f(x)$ is either $q(x)$ or $H_{\eta}(x)$. In the latter case, the bracketed term in the second line simplifies to $f(x_0)-f_\star$.
\end{proposition}
See Appendix \ref{subap:fgratesofscomp} for the proof of \cref{prop:fg_rates_of_scomp}.

\begin{remark}
\label{rem:compare_def}
Concurrent work by \citet{zhang2024acceleratedgradientdescentconcatenation} proposes ``dominant'', ``primitive'', and ``$g$-dominated'' stepsize schedules that roughly parallel our \fcomposable{}, \scomposable{}, and \gcomposable{} definitions, \emph{but are distinct definitions}.
We find our definitions to be slightly more natural as they stem from balancing (standard) performance measures on Huber and quadratic functions.
For example, a stepsize schedule $h\in\R^n_{++}$ is \fcomposable{} if the worst-case objective gap is attained by both the quadratic and Huber functions.
On the other hand, this stepsize schedule is dominant in the language of~\citet{zhang2024acceleratedgradientdescentconcatenation} if there exists $u\in\R^{n+1}_{+}$ so that $\sum_i u_i = 1+2\sum_i h_i$ and
$$\frac{1}{2}\left(\norm{x_0 - x_\star}^2 - \norm{x_0 - x_\star - \sum_i u_i g_i}^2\right) - \frac{1}{2}\sum_{i}u_iQ_{\star,i} - (1^\intercal u)(f_n - f_\star) \geq 0.$$
See \eqref{eq:qdef} for definition of $Q_{\star,i}$.
They show that this condition implies that gradient descent with stepsize $h$ satisfies $f_n - f_\star \leq \frac{1}{1+2\sum_i h_i}\cdot\frac{1}{2}\norm{x_0 - x_\star}^2$ (roughly, dominant schedules are \fcomposable{}) but do not prove a reverse implication (roughly,  \fcomposable{} schedules are dominant).\qedhere
\end{remark}

\subsubsection{Simple Examples of Composable Schedules}
As concrete examples, we first consider the empty schedule and constant stepsize schedules. The following two sections will provide many more, nontrivial examples leveraging composition operations to build from these simple examples. In particular, the empty schedule is an invaluable building block.
\begin{example}
	\label{ex:h0_composable}
	Define $h = \emptySchedule$ to be the empty vector.
	By convention, we evaluate empty sums as $0$ and empty products as $1$. Then it is easy to verify that $h$ is \fcomposable{} with rate $1$, \gcomposable{} with rate $1$, and \scomposable{} with rate $1$. For example, we can verify that the empty schedule $h = \emptySchedule$ is \fcomposable{} because 
$$f(x_0) - f(x_\star) \leq \frac{\norm{x_0 - x_\star}^2}{2}$$
on any $1$-smooth convex function $f$.\qedhere
\end{example}
\begin{example}
	\label{ex:constant_composable}
	For any fixed $n$, consider the constant schedule of stepsizes $h_0=\dots=h_{n-1} = \bar h$ where $\bar h$ is the unique positive solution to the equation
	$$ \frac{1}{1+2\bar hn} = (\bar h-1)^{2n}. $$
	Recently, \cite{rotaru2024exact} proved that for any $n$ this schedule is optimal among all constant schedules of length $n$ for reducing gradient norm given a bound on the initial suboptimality.
    Its worst-case convergence rate is attained by both quadratic and Huber problem instances.
    As a result, this schedule is \gcomposable{} with rate $\eta = \frac{1}{1+2n \bar h}$.
    Similarly, this $\bar h$ is also the conjectured optimal constant stepsize for minimizing the final objective gap~\cite{drori2012PerformanceOF,grimmer2024strengthenedconjectureminimaxoptimal}. If true, this pattern is also \fcomposable{} with rate $\eta = \frac{1}{1+2n\bar h}$.

    Numerical evidence suggests that for any $n$, the constant stepsize schedule $h_0=\dots=h_{n-1}= \bar h$, where $\bar h$ is the unique positive solution of
	$$ \frac{1}{1+\bar hn} = (\bar h-1)^{n},$$
    is \scomposable{}.
    This construction gives $h=[\sqrt{2}]$ and $h=[3/2,3/2]$ when $n=1$ and $n=2$ respectively. One can directly verify that these schedules are \scomposable{} with rates $\sqrt{2}-1$ and $\frac{1}{4}$ respectively. We leave the general statement for $n\geq 3$ as an open question.
    \qedhere
\end{example}

\subsection{Composing Schedules and Inductive Composition Theorems}

As our naming indicates, composable schedules can be composed together to yield larger, more interesting composable schedules. This section introduces three operations that can be used to produce new composable schedules. Given these operations and their associated guarantees, one can readily recover existing theory (Section~\ref{sec:basic}) and derive new (potentially minimax optimal) stepsize schedules and theory (Section~\ref{sec:schedules}). We will first showcase a number of examples of these composable schedules, so we defer proofs of these inductive composition theorems to Section~\ref{sec:proofs} (and Appendix \ref{ap:deferred}), where alternative performance estimation style definitions of each type of composability and the underlying proof machinery leveraging the general recursive gluing technique of~\cite{altschuler2023accelerationPartI} are developed.

\begin{definition}
	Suppose $a\in\R^{n_a-1}_{++}$ is
	\scomposable{} with rate $\alpha$
	and $b\in\R^{n_b-1}_{++}$ is \fcomposable{} with rate $\beta$.
	Define the f-join of these schedules as $a\triangleright b \coloneqq [a,\mu,b]$, where
	\begin{equation*}
		\mu \coloneqq 1 + \frac{\sqrt{\alpha^2 + 8 \alpha\beta}-\alpha}{4\alpha\beta}.
    \end{equation*}
    We will overload the $\triangleright$ symbol: Given two nonnegative scalars $\alpha$ and $\beta$, define $\alpha\triangleright \beta \coloneqq \frac{2\alpha\beta}{\alpha + 4\beta + \sqrt{\alpha^2+8\alpha\beta}}$.\qedhere
\end{definition}
We verify that the joined schedule is \scomposable{} at the joined rate in the following theorem.
\begin{theorem}
	\label{thm:f_recurrence}
	For any \scomposable{} $a\in\R^{n_a-1}_{++}$ with rate $\alpha$
	and \fcomposable{} $b\in\R^{n_b-1}_{++}$ with rate $\beta$, $a\triangleright b$ is \fcomposable{} with rate $\alpha\triangleright \beta$.
\end{theorem}

We define the g-join (denoted $b\triangleleft a$) and the s-join (denoted $a\Join b$) operations similarly below and present their inductive theorems (again with proofs deferred to Section~\ref{sec:proofs}). Our notation choice emphasizes the duality, proven in Section~\ref{subsec:Hdual}, between the f-join $a\triangleright b$ and the g-join $b\triangleleft a$, as well as the self-duality of the s-join.
\begin{definition}
	Suppose $b\in\R^{n_b-1}_{++}$ is \gcomposable{} with rate $\beta$ and $a\in\R^{n_a-1}_{++}$ is
	\scomposable{} with rate $\alpha$.
	Define the g-join of these schedules as $b\triangleleft a \coloneqq [b,\mu,a]$, where
	\begin{align*}
		\mu \coloneqq 1 + \frac{\sqrt{\alpha^2 + 8 \alpha\beta}-\alpha}{4\alpha\beta}.
    \end{align*}
    We will overload the $\triangleleft$ symbol: Given two nonnegative scalars $\alpha$ and $\beta$, define
	$\beta\triangleleft \alpha \coloneqq \frac{2\alpha\beta}{\alpha + 4\beta + \sqrt{\alpha^2+8\alpha\beta}}$.\qedhere
\end{definition}
\begin{theorem}
	\label{thm:g_recurrence}
	For any \gcomposable{} $b\in\R^{n_b-1}_{++}$ with rate $\beta$
	and \scomposable{} $a\in\R^{n_a-1}_{++}$ with rate $\alpha$, $b\triangleleft a$ is \gcomposable{} with rate $\beta\triangleleft \alpha$.
\end{theorem}

\begin{definition}
	Suppose $a\in\R^{n_a-1}_{++}$ is \scomposable{} with rate $\alpha$
	and $b\in\R^{n_b-1}_{++}$ is \scomposable{} with rate $\beta$.
	Define the s-join of these schedules as $a\Join b \coloneqq [a,\mu,b]$, where
	\begin{align*}
		\mu \coloneqq 1+\frac{\sqrt{\alpha^2 + 6\alpha\beta + \beta^2}-\left(\alpha + \beta\right)}{2\alpha\beta}.
    \end{align*}
    We will overload the $\Join$ symbol: Given two nonnegative scalars $\alpha$ and $\beta$, define
	$\alpha\Join \beta \coloneqq \frac{2\alpha\beta}{\alpha + \beta + \sqrt{\alpha^2 + 6\alpha\beta + \beta^2}}$.\qedhere
\end{definition}
\begin{theorem}
	\label{thm:selfcomp_recurrence}
	For any \scomposable{} $a\in\R^{n_a-1}_{++}$ with rate $\alpha$
	and \scomposable{} $b\in\R^{n_b-1}_{++}$ with rate $\beta$, $a\Join b$ is \scomposable{} with rate $\alpha\Join \beta$.
\end{theorem}

Lastly, we note the following basic observations regarding the different join operations.
\begin{lemma}
    \label{lem:composition_basic_facts}
    Suppose $\alpha,\beta>0$. Then,
    \begin{itemize}
        \item The s-join is homogeneous. That is, if $r>0$, then
        \begin{equation*}
            (r\alpha)\Join(r\beta) = r (\alpha\Join \beta).
        \end{equation*}
        The f-join and g-join are similarly homogeneous.
        \item The s-join is increasing in both arguments. That is, if $\alpha'>\alpha$ and $\beta'>\beta$, then
        \begin{equation*}
            \alpha'\Join \beta > \alpha \Join \beta \qquad\text{and}\qquad \alpha\Join \beta' > \alpha \Join \beta.
        \end{equation*}
        The f-join and g-join are similarly increasing in both arguments.
        \item The s-join is commutative on rates, that is $\alpha\Join\beta = \beta\Join \alpha$.
        \item The identity $\alpha \Join \alpha = \frac{\alpha}{1+\sqrt{2}}$ holds.
        \item It holds that $\alpha \Join 1 < \alpha$. Similarly, $1\triangleright \alpha = \alpha\triangleleft 1 < \alpha$.
    \end{itemize}
\end{lemma}

For a given stepsize schedule $h \in \R^n_{++}$, we define $\len(h) = n$.
\begin{lemma}\label{lem:fcomp_simple}
    In this case, we have that if $a, b$ are \scomposable{} schedules, then
    \[
        \len(a \Join b) = \len(a) + \len(b) + 1,
    \]
    and similarly if $a$ is \scomposable{} and $b$ is \fcomposable{}, then
    $\len(a \triangleright b) = \len(a) + \len(b) + 1$ and if $a$ is \gcomposable{} and $b$ is \scomposable{}, then
    $\len(a \triangleleft b) = \len(a) + \len(b) + 1$.
\end{lemma}
We note that these composition operations do not follow the associative law, i.e., in general,
\[
    (a \Join b) \Join c \neq a \Join (b \Join c).
\]
Indeed, for an \scomposable{} schedule $a$, it is not even the case that $a \Join (a \Join a) = (a \Join a) \Join a$. Similar failures for associativity also occur for the associated rates.
\emph{We will be careful to include all necessary parentheses}.

\subsection{A Summary of Notation and Conventions}
So far, we have introduced an array of new concepts and notation. This subsection is a summary of the notation and basic definitions which have already been established so far, and is meant to be a concise ``user's guide'' to the notation in this paper.

Throughout the paper, we will use $h \in \R^n_{++}$ to denote a stepsize schedule, with the convention that these schedules are indexed starting from 0, i.e., the entries of $h$ are $h_0, \dots, h_{n-1}$.
In the case that $n = 0$, we have that $h = \emptySchedule$, which we think of as an ``algorithm'' which, when run on an instance $(f, x_0)$ returns $x_0$.
At times, we will consider ``composing'' two stepsize schedules, and will use the symbols $a \in \R^{m_1}, b \in \R^{m_2}$ to denote these stepsize schedules. 

Often, we will treat the stepsize schedule $h$, the initial iterate $x_0 \in \R^d$, and the 1-smooth convex function $f : \R^d \rightarrow \R$ with minimizer $x_\star \in \R^d$ as being implicit and define $x_0, \dots, x_n$ by letting $x_i = x_{i-1} + h_{i-1} \nabla f(x_{i-1})$; $g_0, \dots g_n$ by $g_i  = \nabla f(x_i)$, and finally, $f_0, \dots, f_{n}$ by $f_i = f(x_i)$. We will further define $f_\star$ to be $f(x_\star)$, the minimum value of $f$.

There are different notions of convergence for a stepsize schedule, which take the form of different inequalities relating the initial conditions $f_0-f_\star$ and $\norm{x_0-x_\star}^2$ and the final conditions $f_n-f_\star, \norm{x_n-x_\star}^2$ and $\norm{g_n}^2$.
There are three such notions of convergence which we will be interested in with for a given schedule $h$:
\begin{itemize}
    \item The f-convergence rate
    \[
        f_n-f_\star \leq \eta \frac{\norm{x_0 - x_\star}^2}{2}\qquad \text{for all }(f,x_0),
    \]
    \item The g-convergence rate
    \[
        \frac{1}{2}\norm{g_n}^2 \leq \eta (f_0-f_\star)\qquad \text{for all }(f,x_0),
    \]
    \item The s-convergence rate
    \[
        \frac{1-\eta}{2}\norm{g_n}^2 +\frac{\eta^2}{2}\norm{x_n-x_\star}^2 + (\eta-\eta^2)(f_n-f_\star) \leq \frac{\eta^2}{2}\norm{x_0 - x_\star}^2\qquad \text{for all }(f,x_0).
    \]
\end{itemize}
Once one of these notions of convergence is fixed, the ``convergence rate'' of the schedule $h$ is the smallest value of $\eta$ for which the associated inequality holds.
If the notion of convergence is implicit, we will use the symbol $\eta$ to denote the convergence rate of the schedule $h$, or if there are two schedules $a,b$ for which we must consider the convergence rates, we will use the symbols $\alpha$ and $\beta$ to denote these convergence rates respectively.

We then define the ``composable'' schedules for each of these notions of convergence to be those schedules with bounded convergence rates and for which the associated rate is met at equality when the function $f$ is either the Huber function or quadratic functions defined earlier.
In operational terms, this condition translates to the specific equations summarized below.
The empty schedule $\emptySchedule$ is noteworthy as it satisfies all three of these conditions with rate 1.

\begin{center}
    \label{tbl:summary}
\begin{tabular}{c|cc}
    Schedule type & Convergence guarantee & Equation\\
    \hline
    \fcomposable{} & $f_n-f_\star \leq \eta \frac{\norm{x_0 - x_\star}^2}{2}$ &
    $\eta = \frac{1}{1+2\sum_{i=0}^{n-1} h_i}  = \prod_{i=0}^{n-1} (h_i - 1)^2$\\
    \gcomposable{} & $\frac{1}{2}\norm{g_n}^2 \leq \eta (f_0-f_\star)$  &
    $\eta = \frac{1}{1+2\sum_{i=0}^{n-1} h_i}  = \prod_{i=0}^{n-1} (h_i - 1)^2$\\
    \scomposable{} &
    $\begin{array} {lcl}
    \frac{1-\eta}{2}\norm{g_n}^2 +\frac{\eta^2}{2}\norm{x_n-x_\star}^2\\ \quad + (\eta-\eta^2)(f_n-f_\star) \leq \frac{\eta^2}{2}\norm{x_0 - x_\star}^2
    \end{array}$
    & 
		$\eta = \frac{1}{1+\sum_{i=0}^{n-1}h_i} = \prod_{i=0}^{n-1}(h_i - 1)$
\end{tabular}
\end{center}

We will also need notation for composing schedules.
Our definition of composition will always involve a middle stepsize, which we will denote by $\mu \in \R$, and we will compose schedules $a$ and $b$ by setting $h = [a, \mu, b]$, where the brackets denote concatenation.
Depending on the types of schedules being composed, we will require different definitions of $\mu$, and the resulting rates will depend on the rates $\alpha$ and $\beta$ of the input stepsizes.
We overload the notation for composition to also be applicable to rates.
This information is summarized in the next table.

\begin{center}
    \label{tbl:rate_summary}
\begin{tabular}{c|ccc}
    Schedules & Composition & $\mu$ & New Rate\\
    \hline
    $a$ (s-comp.), $b$ (f-comp.)
        & $a \triangleright b$ (f-comp.)
        & $1 + \frac{\sqrt{\alpha^2 + 8 \alpha\beta}-\alpha}{4\alpha\beta}$
        & $\alpha\triangleright\beta\coloneqq \frac{2\alpha\beta}{\alpha + 4\beta + \sqrt{\alpha^2+8\alpha\beta}}$
        \\
    $a$ (s-comp.), $b$ (g-comp.) 
        & $b \triangleleft a$ (g-comp.)
        & $1 + \frac{\sqrt{\alpha^2 + 8 \alpha\beta}-\alpha}{4\alpha\beta}$
        & $\beta\triangleleft\alpha\coloneqq \frac{2\alpha\beta}{\alpha + 4\beta + \sqrt{\alpha^2+8\alpha\beta}}$
        \\
    $a$ (s-comp.), $b$ (s-comp.)
        & $a \Join b$ (s-comp.)
        & $1+\frac{\sqrt{\alpha^2 + 6\alpha\beta + \beta^2}-\left(\alpha + \beta\right)}{2\alpha\beta}$
        & $\alpha\Join\beta\coloneqq\frac{2\alpha\beta}{\alpha + \beta + \sqrt{\alpha^2 + 6\alpha\beta + \beta^2}}$
\end{tabular}
\end{center}

\section{Basic Stepsize Schedules} \label{sec:basic}

Given the f-join, g-join, and s-join operations and any initial set of \fcomposable{}, \gcomposable{}, and \scomposable{} schedules, one can generate new composable schedules using \cref{thm:f_recurrence,thm:g_recurrence,thm:selfcomp_recurrence}.
Here, we consider the simplest initialization, starting with only the empty schedule $h=\emptySchedule$, which is \fcomposable{}, \gcomposable{}, and \scomposable{}, all with rate one. Suprisingly, this suffices for the purposes of recovering existing theory (this section) and proposing new, potentially minimax optimal, schedules (\cref{sec:schedules}).

\begin{definition}\label{def:basic}
	We refer to the empty schedule $h=\emptySchedule$ as a \emph{basic} \fcomposable{}, \gcomposable{}, and \scomposable{} schedule. Inductively, we refer to any schedule that can be created by composing two basic schedules via any operation $\triangleright, \triangleleft, \Join$ as a basic schedule.\qedhere
\end{definition}

\begin{remark}
    The constant stepsize schedules of \cref{ex:constant_composable} are examples of stepsize schedules which are \emph{not} basic. Thus, the \emph{basic} f-, g-, and s-composable schedules are a \emph{proper subset} of the general f-, g-, and s-composable schedules.
\end{remark}

The following table lists all length $n=1,2,3$ basic \scomposable{} schedules. We omit exact formulas for stepsizes and rates for $n=3$ involving more than one nested squareroot, instead simply presenting numerical estimates. \\
\begin{center}
	\begin{tabular}{c|c|c}
	Basic \scomposable{} Schedule & Basic Construction via $\Join$ & \scomposable{} Rate \\
	\hline
	$[\sqrt{2}]$ & $\emptySchedule \Join \emptySchedule$ & $\sqrt{2}-1$ \\
	\hline
	$\left[1+\frac{\sqrt{4\sqrt{2} -2}-\sqrt{2}}{2(\sqrt{2} - 1)}, \sqrt{2}\right]$ & $\emptySchedule \Join (\emptySchedule \Join \emptySchedule)$ & $\frac{2(\sqrt{2}-1)}{\sqrt{4\sqrt{2} - 2} +\sqrt{2}}$ \\
	$\left[\sqrt{2}, 1+\frac{\sqrt{4\sqrt{2} -2}-\sqrt{2}}{2(\sqrt{2} - 1)}\right]$ & $(\emptySchedule \Join \emptySchedule) \Join \emptySchedule$ & $\frac{2(\sqrt{2}-1)}{\sqrt{4\sqrt{2} - 2} +\sqrt{2}}$\\
	\hline
	$[\sqrt{2}, 2, \sqrt{2}]$ & $(\emptySchedule \Join \emptySchedule) \Join (\emptySchedule \Join \emptySchedule)$ & $\frac{1}{3+2\sqrt{2}}$\\
	$\approx [1.7023, 1.6012, 1.4142]$ & $\emptySchedule \Join (\emptySchedule \Join (\emptySchedule \Join \emptySchedule))$ & $\approx 0.17489$ \\
	$\approx [1.7023, 1.4142, 1.6012]$ & $\emptySchedule \Join ((\emptySchedule \Join \emptySchedule) \Join \emptySchedule)$ & $\approx 0.17489$\\
	$\approx [1.6012, 1.4142, 1.7023]$ & $(\emptySchedule \Join (\emptySchedule \Join \emptySchedule) \Join \emptySchedule $ & $\approx 0.17489$ \\
	$\approx [1.4142, 1.6012, 1.7023]$ & $((\emptySchedule \Join \emptySchedule) \Join \emptySchedule) \Join \emptySchedule$ & $\approx 0.17489$ \\
\end{tabular}
\end{center}

From the five possible basic \scomposable{} schedules with $n=3$, we can see they do not all possess the same convergence rate. $[\sqrt{2}, 2, \sqrt{2}] = (\emptySchedule \Join \emptySchedule) \Join (\emptySchedule \Join \emptySchedule)$, corresponding to the second silver stepsize schedule of~\cite{altschuler2023accelerationPartII}, is the optimal basic \scomposable{} schedule of length three. More sophisticated examples of  \scomposable{} schedules are discussed in Section~\ref{subsubsec:silver} where we show repeated composition with $\Join$ can produce every silver stepsize schedule. In Section~\ref{sec:schedules}, we will return to the idea of constructing optimal basic schedules.

Similarly, one can construct \fcomposable{} and \gcomposable{} schedules inductively. For example, all basic \fcomposable{} schedules of length $n=1,2,3$ are given below. In this case, the optimal basic schedule of length two is $[\sqrt{3}, 3/2]$ and length three is $[\sqrt{2}, 1+\sqrt{2}, 3/2]$, corresponding to the minimax optimal schedules numerically found by~\cite{gupta2023branch}. As discussed in the following section, the basic \gcomposable{} schedules are exactly the reverse of each basic \fcomposable{} schedule and possess the same rate, hence we omit a table of basic \gcomposable{} schedules.
\begin{center}
    \label{tbl:basicFcomposableExamples}
	\begin{tabular}{c|c|c}
		Basic \fcomposable{} Schedule & Basic Construction via $\Join$ and $\triangleright$ & \fcomposable{} Rate \\
		\hline
		$[3/2]$ & $\emptySchedule \triangleright \emptySchedule$ & $1/4$ \\
	\hline
		$[\sqrt{3}, 3/2]$ & $\emptySchedule \triangleright (\emptySchedule \triangleright \emptySchedule)$ & $\frac{1}{\sqrt{3} +4}$ \\
		$[\sqrt{2}, \frac{1}{4}\left(3+\sqrt{9+8\sqrt{2}}\right)]$ & $(\emptySchedule \Join \emptySchedule) \triangleright \emptySchedule$ & $\frac{2}{\sqrt{9+8\sqrt{2}} +4\sqrt{2} + 5}$\\
	\hline
		$[\sqrt{2}, 1+\sqrt{2}, 3/2]$ & $(\emptySchedule \Join \emptySchedule) \triangleright (\emptySchedule \triangleright \emptySchedule)$ & $\frac{1}{6+4\sqrt{2}}$\\
		$\approx [1.8199, 1.7321, 1.5]$ & $\emptySchedule \triangleright (\emptySchedule \triangleright (\emptySchedule \triangleright \emptySchedule))$ & $\approx 0.09006$ \\
		$\approx [1.8218, 1.4142, 1.8768]$ & $\emptySchedule \triangleright ((\emptySchedule \Join \emptySchedule) \triangleright \emptySchedule)$ & $\approx 0.08908$\\
		$\approx [1.6012, 1.4142, 2.1888]$ & $(\emptySchedule \Join (\emptySchedule \Join \emptySchedule) \triangleright \emptySchedule $ & $\approx 0.08765$ \\
		$\approx [1.4142, 1.6012, 2.1888]$ & $((\emptySchedule \Join \emptySchedule) \Join \emptySchedule) \triangleright \emptySchedule$ & $\approx 0.08765$ \\
	\end{tabular}
\end{center}
More sophisticated examples of \fcomposable{} and \gcomposable{} schedules are given in Sections~\ref{subsubsec:bnb-optimal-schedules}--\ref{subsubsec:dynamic-short-schedules} recovering the numerically minimax optimal stepsizes previously computed for all $n\leq 25$, previous long step schedules of the authors, and novel short stepsize schedules.

\subsection{H-Duality Theory for All Basic Schedules}\label{subsec:Hdual}
By construction, basic schedules exhibit a simple duality theory following the symmetries identified by~\citet{Kim2024-Hduality}. Therein, it was shown that if a fixed step first-order method (potentially including momentum) possesses a special type of inductive proof of an objective gap convergence rate like~\eqref{eq:f_composable_inequality}, then the reversed fixed-step first-order method would also possess a specific type of inductive proof showing a gradient norm convergence rate like~\eqref{eq:g_composable_ineq}. Their proof of this structure, called H-duality, was highly tailored to the particular type of inductive proof used to analyze Nesterov's momentum method~\cite{Nesterov1983} and OGM~\cite{Kim2016optimal}. The convergence proofs generated by our composition operations lack this tailored form. Regardless, the following propositions establish an H-duality theory for every basic schedule between $h=[h_0\dots h_{n-1}]$ and its reverse $\rev(h)=[h_{n-1}\dots h_0]$.
\begin{proposition}\label{thm:self-duality}
	A stepsize schedule $h$ is basic \scomposable{} with rate $\eta$ if and only if $\rev(h)$ is basic \scomposable{} with rate $\eta$.
\end{proposition}
\begin{proof}
    By symmetry, it suffices to prove the forward direction. Suppose $h$ has length $n$. We prove this statement by induction. If $n=0$, then $h = \emptySchedule$, and $h$ and $\rev(h)$ are both \scomposable{} schedules with the same rate $\eta = 1$, as was shown in \cref{ex:h0_composable}.
	Next, suppose $n\geq 1$. By definition, $h$ can be written as $a \Join b$ for some basic \scomposable{} schedules $a$ and $b$ with rates $\alpha$ and $\beta$. By induction, $\rev(a)$ and $\rev(b)$ are basic \scomposable{} schedules with rates $\alpha$ and $\beta$ and $\len(a), \len(b)<\len(h)$. Then,
	$\rev(b)\Join\rev(a)= \rev(h)$ is basic \scomposable{} with rate $\beta\Join\alpha = \eta$.\qedhere
\end{proof}

\begin{proposition}\label{thm:HDual}
	\label{lem:basic_f_g}
	A stepsize schedule $h$ is basic and \fcomposable{} with rate $\eta$ if and only if $\rev(h)$ is basic and \gcomposable{} schedule with rate $\eta$.
	Analogously, $h$ is basic and \gcomposable{} with rate $\eta$ if and only if $\rev(h)$ is \fcomposable{} with rate $\eta$.
\end{proposition}
\begin{proof}
	We prove the forward direction. Suppose $h$ has length $n$. We prove this statement by induction. If $n=0$, then $h$ is \fcomposable{} with rate $\eta =1$ and $\rev(h)$ is \gcomposable{} with rate $\eta = 1$.
	Next, suppose $n\geq 1$. By definition, $h$ can be written as $a \triangleright b$ for some basic \scomposable{} schedule $a$ with rate $\alpha$ and some basic \fcomposable{} schedule $b$ with rate $\beta$. By 
    the previous proposition, $\rev(a)$ is basic \scomposable{} with rate $\alpha$. By induction, $\rev(b)$ is basic \gcomposable{} schedule with rate $\beta$. Then,
	$\rev(b)\triangleleft\rev(a)= \rev(h)$ is basic \gcomposable{} with rate $\beta\triangleleft \alpha = \eta$.
    
    The proof of the backward direction is analogous.\qedhere
\end{proof}
Lastly, we observe that \cref{prop:fg_rates_of_scomp} shows that \scomposable{} schedules have the same \fcomposable{} and \gcomposable{} rates, which resembles a form of H-duality in which the sequence is not reversed.

\subsection{Recovery of Prior State-of-the-Art Works as Basic Schedules}

The performance estimation framework pioneered by~\cite{drori2012PerformanceOF,taylor2017interpolation} enabled a wave of recent works identifying stepsize schedules for gradient descent with superior performance to classic ``textbook'' approaches. Below we discuss four such advances which each developed specialized state-of-the-art stepsizes schedules via their own ad hoc approach. Our composition machinery recovers all of these prior works as basic schedules arising from simple combinations of the operations $\Join,\triangleright,\triangleleft$.

\subsubsection{Silver Stepsizes of~\cite{altschuler2023accelerationPartII}}\label{subsubsec:silver}
The first stepsize schedule achieving an accelerated objective gap convergence rate of $O(1/n^{\log_2(1+\sqrt{2})})$ was the silver stepsize schedule of \citet{altschuler2023accelerationPartII}. For any $k\geq 0$, we denote the silver stepsize schedule by $\pi^{(k)}\in \mathbb{R}^n_{++}$ which takes length exactly $n=2^{k}-1$. This sequence of stepsize schedules can be defined recursively having $\pi^{(0)}=\emptySchedule$ and thereafter
\begin{equation}\label{eq:silver-construction}
	\pi^{(k+1)} = [\pi^{(k)},  1+(1+\sqrt{2})^{k-1} ,\pi^{(k)}] \ .
\end{equation}
An objective gap convergence rate was proven in~\cite[Theorem 1.1]{altschuler2023accelerationPartII} showing for any $1$-smooth convex function, gradient descent with $h=\pi^{(k)}$ has
\begin{equation}
	f(x_n) - f(x_\star) \leq \frac{\|x_0-x_\star\|^2}{1+\sqrt{4(1+\sqrt{2})^{2k}-3}} \ . \label{eq:silver-rate}
\end{equation}
Despite its fast objective gap convergence, this stepsize schedule is not \fcomposable{}: considering its performance on the simple $1$-smooth quadratic function $\frac{1}{2}x^2$, it attains a much faster rate of
$$ f(x_n) - f(x_\star) =  \frac{\|x_0-x_\star\|^2}{2(1+\sqrt{2})^{2k}}, $$
whereas on an appropriately chosen Huber function, convergence occurs at a slightly faster rate of
$$ f(x_n) - f(x_\star) = \frac{\|x_0-x_\star\|^2}{4(1+\sqrt{2})^{k}-2} \ . $$
As a result, these two extremal cases are not equally balanced (nor do they attain the previously proven convergence rate of~\eqref{eq:silver-rate}).

Instead of hedging objective gap performance between these extremal functions, we find that the silver schedule perfectly balances performance on quadratic and Huber functions in terms of the \scomposable{} inequality~\eqref{eq:self_composable_ineq}.
\begin{lemma}\label{lem:silver_steps_basic}
    The silver stepsize schedule $\pi^{(k)}$ is basic \scomposable{} with rate $(1+\sqrt{2})^{-k}$. 
\end{lemma}
\begin{proof}
    We show that $\pi^{(k)}$ is basic with this rate by induction, first noting that $\pi^{(0)} = \emptySchedule$ is basic with rate 1 by definition. We then claim that $\pi^{(k)} \Join \pi^{(k)} = \pi^{(k+1)}$, which follows since
    \[
        \pi^{(k)} \Join \pi^{(k)} = [\pi^{(k)}, \mu_k, \pi^{(k)}],
    \]
    where
    \[
        \mu_k = 1+\frac{\sqrt{8 \eta_k^2}-2\eta_k}{2\eta_k^2} = 
        1+\frac{\sqrt{2}-1}{\eta_k},
    \]
    where $\eta_k$ is the rate of $\pi^{(k)}$, which by induction is $(1+\sqrt{2})^{-k}$.
    This implies that $\mu_k = 1+\frac{\sqrt{2}-1}{\eta_k} = 1+(1+\sqrt{2})^{k-1}$, which agrees with the definition of the silver stepsize schedule.

    The rate then follows from Lemma~\ref{lem:composition_basic_facts}. \qedhere
\end{proof}

By~\cref{prop:fg_rates_of_scomp}, our theory yields an improved objective gap guarantee for the silver stepsize schedule of $f(x_n) - f(x_\star) \leq  \frac{L\|x_0-x_\star\|^2}{4(1+\sqrt{2})^{k}-2}$, which was conjectured by~\cite{luner2024averagingextrapolationgradientdescent} and recently proven directly by~\cite{wang2024relaxedproximalpointalgorithm}. A matching gradient norm rate similarly follows from~\cref{prop:fg_rates_of_scomp}. Note these improved convergence rates are tight as they are attained by the Huber instance above. The first few examples of this construction are below.
\begin{center}
	\begin{tabular}{c|c|c|c}
		$n$ & Basic Construction of ``Silver Stepsizes''~\cite{altschuler2023accelerationPartII} & Schedule & New $s$-Rate\\
		\hline
		1 & $\emptySchedule\Join\emptySchedule$ & $[\sqrt{2}]$ & $\frac{1}{1+\sqrt{2}}$ \\
		3 & $(\emptySchedule\Join\emptySchedule)\Join(\emptySchedule\Join\emptySchedule)$ & $[\sqrt{2},2,\sqrt{2}]$ & $\frac{1}{(1+\sqrt{2})^2}$ \\
		7 & $((\emptySchedule\Join\emptySchedule)\Join(\emptySchedule\Join\emptySchedule))\Join((\emptySchedule\Join\emptySchedule)\Join(\emptySchedule\Join\emptySchedule))$ & $[\sqrt{2},2,\sqrt{2}, 2+\sqrt{2}, \sqrt{2},2,\sqrt{2}]$ & $\frac{1}{(1+\sqrt{2})^3}$
	\end{tabular}
\end{center}

\subsubsection{Numerically Minimax Optimal Stepsizes for $n=1,2,\dots,25$ of~\cite{gupta2023branch}} \label{subsubsec:bnb-optimal-schedules}

As discussed in the introduction, \citet{gupta2023branch} leveraged substantial computational resources to globally identify minimax optimal stepsize schedules (up to numeric error) for $n\leq 25$. Here, minimax optimal is in the sense of solving~\eqref{eq:minimax}. These floating-point schedules (as well as locally optimized schedules up to $n=50$) are available online\footnote{See  \url{https://github.com/Shuvomoy/BnB-PEP-code/blob/main/Misc/stpszs.jl}}.

We find that every globally numerically optimized schedule reported by~\cite{gupta2023branch} is basic (up to numerical errors): For every $n=1,\dots,25$, a simple computer search over possible constructions found a basic \fcomposable{} schedule with relative, infinity norm difference from the numerically minimax optimal schedule of at most $10^{-3}$. This error is smaller, on the order of $10^{-6}$, for schedules with smaller $n$ where the branch-and-bound solves of~\cite{gupta2023branch} were of higher accuracy. Basic constructions for the first ten schedules matching the numerical minimax optimal solves are given below.
\begin{center}
	\begin{tabular}{c|c|c}
		 & Basic Construction of Previously Numerically Identified & New Objective Gap\\
		$n$ & Minimax Optimal Schedules~\cite{gupta2023branch} &  Convergence Rates\\
		\hline
		1  & $\emptySchedule \triangleright \emptySchedule$ &  $0.25$ \\
		2  & $(\emptySchedule \Join \emptySchedule) \triangleright \emptySchedule$ & $0.13189$  \\
		3  & $(\emptySchedule \Join \emptySchedule) \triangleright (\emptySchedule \triangleright \emptySchedule)$ &  $0.08579$ \\
		4  & $((\emptySchedule \Join \emptySchedule) \Join \emptySchedule) \triangleright (\emptySchedule \triangleright \emptySchedule)$ & $0.06234$  \\
		5  & $((\emptySchedule \Join \emptySchedule) \Join (\emptySchedule \Join \emptySchedule)) \triangleright (\emptySchedule \triangleright \emptySchedule)$ & $0.04814$  \\
		6  & $(\emptySchedule \Join (\emptySchedule \Join \emptySchedule)) \triangleright (((\emptySchedule \Join \emptySchedule) \Join \emptySchedule) \triangleright \emptySchedule)$ & $0.04020$  \\
		7  & $((\emptySchedule \Join (\emptySchedule \Join \emptySchedule)) \Join (\emptySchedule \Join \emptySchedule)) \triangleright ((\emptySchedule \Join \emptySchedule) \triangleright \emptySchedule)$ & $0.03266$  \\
		8  & $(((\emptySchedule \Join \emptySchedule) \Join (\emptySchedule \Join \emptySchedule)) \Join (\emptySchedule \Join(\emptySchedule \Join \emptySchedule))) \triangleright (\emptySchedule \triangleright \emptySchedule)$ & $0.02811$  \\
		9  & $((\emptySchedule \Join (\emptySchedule \Join \emptySchedule)) \Join (\emptySchedule \Join \emptySchedule))\triangleright ((\emptySchedule \Join \emptySchedule) \triangleright ((\emptySchedule \Join \emptySchedule) \triangleright \emptySchedule))$ &  $0.02456$ \\
		10 & $((\emptySchedule \Join(\emptySchedule \Join \emptySchedule)) \Join (\emptySchedule \Join \emptySchedule) \Join (\emptySchedule \Join \emptySchedule)) \triangleright ((\emptySchedule \Join \emptySchedule) \triangleright (\emptySchedule \triangleright \emptySchedule))$ & $0.02124$
	\end{tabular}
\end{center}
Since these basic schedules are built entirely from our operations $\Join$ and $\triangleright$, our theory provides these schedules their first formal convergence rate guarantees, following immediately from our inductive composition Theorems~\ref{thm:f_recurrence} and~\ref{thm:selfcomp_recurrence}. This is in contrast to~\cite{gupta2023branch} where only floating point dual certificates were available, approximating a proof of convergence.

This strong relation between every numerically identified minimax optimal schedule and basic schedules motivates the following natural conjecture.
\begin{conjecture}\label{conj:strong-f-minimax-descripiton}
    For each $n$, every minimax optimal (fixed step) stepsize schedule, solving~\eqref{eq:minimax}, is basic and \fcomposable{}.
\end{conjecture}
If true, this would have two key consequences: (i) the basic optimized stepsize schedules we propose in \cref{sec:schedules} are optimal for fixed step gradient descent and (ii) the minimax optimal rate for any fixed step gradient descent method is $\Theta(1/n^{\log_2{1+\sqrt{2}}})$ with nearly tight bounds on the exact suppressed coefficient.
Thus, this conjecture would imply a strict separation between momentum based methods (which converge at a rate of $O(\frac{1}{n^2})$) and fixed step gradient descent methods. We also note that methods that choose step sizes depending on earlier gradient queries (or indeed randomly as in \cite{altschuler2023accelerationPartII}) may be able to outperform fixed step methods even if the conjecture holds.
\subsubsection{Objective Gap and Gradient Norm Minimizing Schedules of~\cite{grimmer2024accelerated}} \label{subsubsec:left-right-heavy-schedules}
Similar to the above construction of silver stepsize schedules via recursive application of $\Join$, the objective gap minimizing stepsizes of~\cite{grimmer2024accelerated} can be similarly recovered, therein called the ``right-side heavy stepsizes''. Namely, these are recovered recursively by setting $h^{\mathtt{right},0} = \emptySchedule$ and recursively setting
$$ h^{\mathtt{right},k+1} = \pi^{(k)} \triangleright h^{\mathtt{right},k},$$
where the quantity $\pi^{(k)}$ on the left is the $k$th silver stepsize schedule. Similarly, the gradient norm minimizing stepsizes, called ``left-side heavy stepsizes'' therein are given recursively by $h^{\mathtt{left},0} = \emptySchedule$ and
$$ h^{\mathtt{left},k+1} =  h^{\mathtt{left},k} \triangleleft \pi^{(k)}\ .$$
That these schedule are \fcomposable{} and \gcomposable{} (with their associated rates) follow immediately from repeated applications of Theorems~\ref{thm:f_recurrence}, \ref{thm:g_recurrence} and~\ref{thm:selfcomp_recurrence}. Unlike the previous examples, the resulting rates are not new. Rather they exactly match those proven directly in~\cite[Theorems 1 and 2]{grimmer2024accelerated}. The first few rounds of these constructions are given below, again providing numerical estimates of formulas with more than one nested squareroot.
\begin{center}
	\begin{tabular}{c|c|c|c}
		$n$ & Basic Construction of ``Right-Side Heavy'' Steps~\cite{grimmer2024accelerated} & Schedule & f-Rate\\
		\hline
		1 & $\emptySchedule\triangleright\emptySchedule$ & $ [3/2]$ & $\frac{1}{4}$ \\
		3 & $(\emptySchedule\Join\emptySchedule)\triangleright(\emptySchedule\triangleright\emptySchedule)$ & $ [\sqrt{2}, 1+\sqrt{2}, 3/2]$ & $\frac{1}{6+4\sqrt{2}}$ \\
		7 & $((\emptySchedule\Join\emptySchedule)\Join(\emptySchedule\Join\emptySchedule))\triangleright((\emptySchedule\Join\emptySchedule)\triangleright(\emptySchedule\triangleright\emptySchedule))$ & $ [\sqrt{2},2,\sqrt{2}, 4.602, \sqrt{2},1+\sqrt{2},3/2]$ & $ 0.03277$\\ & & & \\
		\hline
		n & Basic Construction of ``Left-Side Heavy'' Steps~\cite{grimmer2024accelerated} & Schedule & g-Rate\\
		\hline
		1 & $\emptySchedule\triangleleft\emptySchedule $ & $ [3/2]$ & $\frac{1}{4}$ \\
		3 & $(\emptySchedule\triangleleft\emptySchedule)\triangleleft(\emptySchedule\Join\emptySchedule) $ & $ [3/2,1+\sqrt{2},\sqrt{2}]$ & $\frac{1}{6+4\sqrt{2}}$ \\
		7 & $((\emptySchedule\triangleleft\emptySchedule)\triangleleft(\emptySchedule\Join\emptySchedule))\triangleleft((\emptySchedule\Join\emptySchedule)\Join(\emptySchedule\Join\emptySchedule))$ & $ [3/2,1+\sqrt{2},\sqrt{2}, 4.602, \sqrt{2},2,\sqrt{2}]$ & $ 0.03277$
	\end{tabular}
\end{center}

\subsubsection{Dynamic Short Stepsizes of~\cite{rotaru2024exact}} \label{subsubsec:dynamic-short-schedules}
Our composition theory can also meaningfully contribute to the design of stepsize schedules restricted to have $h_i\in (0,2)$, ensuring a decrease in objective value at every iteration. Recall that we say that $h$ is in the short step size regime if $h_i \in (0,2)$ for each $i$, as this ensures that each step of gradient descent using $h$ decreases both the objective value and the gradient norm of the function.

    We will consider a general construction: for a \gcomposable{} step size schedule $h$ with rate $\eta$, we define $h^{\triangleleft \len(h)} = h$, and for $n > \len(h)$, we define $h^{\triangleleft n} = h^{\triangleleft (n-1)} \triangleleft \emptySchedule$. Explicitly, we have that
    \[
        h^{\triangleleft n} = (\dots(h  \triangleleft \emptySchedule)\dots) \triangleleft \emptySchedule = [h, \mu_{\len(h)+1}(h), \dots, \mu_n(h)],
    \]
    where we perform the composition operation $n - \len(h)$ times, i.e., such that $\len(h^{\triangleleft n}) = n$. We let $\eta_n(h)$ be the rate of $h^{\triangleleft n}$.

    We may consider an analogous operation for \fcomposable{} schedules, $h^{\triangleright n} = (\underbrace{\emptySchedule \triangleright (\emptySchedule \triangleright \dots  (\emptySchedule}_{n-\len(h) \text{ times}} \triangleright h)\dots)$, for which the rates and observations below also apply.

    \begin{lemma}
        If $h$ is a \gcomposable{} stepsize schedule, then for each $n > \len(h)$,
        \[
            \eta_n(h) = \frac{2-\mu_n(h)}{2}.
        \]
        We also have that if $n > \len(h)$, then 
        \[
            \mu_{n+1}(h) = \frac{3-2\mu_n(h)+\sqrt{9-4\mu_{n}(h)}}{2(2-\mu_n(h))}.
        \]
    \end{lemma}
    \begin{proof}
        We will let $\mu_n = \mu_n(h)$ and $\eta_n = \eta_n(h)$ to ease notation.
        For $n > \len(h)$, recall the definition of $\mu_{n}$ from the composition $h^{\triangleleft {n-1}} \triangleleft \emptySchedule$,
        \[
            \mu_{n} = 1 + \frac{\sqrt{1 + 8\eta_{n-1}}- 1}{4\eta_{n-1}}.
        \]
        The first equation then follows because
        \begin{align*}
            \eta_{n} &= \frac{2\eta_{n-1}}{4\eta_{n-1} +1 +\sqrt{1 + 8\eta_{n-1}}}\\
            &= \frac{2\eta_{n-1}\left(4\eta_{n-1} -\sqrt{1 + 8\eta_{n-1}}+1\right)}{(4\eta_{n-1}+1)^2 -\left(1 + 8\eta_{n-1}\right)}\\
&= 1-\frac{1 + \frac{\sqrt{1 + 8\eta_{n-1}}-1}{4\eta_{n-1}}}{2}\\
            &= \frac{2-\mu_n}{2}.
        \end{align*}

        For $n > \len(h)$, we may then substitute $\eta_n = \frac{2-\mu_n}{2}$ into the definition of $\mu_{n+1}$ to obtain
        \begin{equation*}
            \mu_{n+1} = 1 + \frac{\sqrt{1 + 4(2-\mu_n)}-1}{2(2-\mu_n)}
            = \frac{3-2\mu_n+\sqrt{9-4\mu_n}}{2(2-\mu_n)}. \qedhere
        \end{equation*}
    \end{proof}
    In particular, for each $n$, $\eta_n(h) > 0$, so we have that $\mu_n(h) = 2(1-\eta_n) < 2$ for $n > \len(h)$. Therefore, if $h$ is in the short stepsize regime, then $h^{\triangleleft n}$ is also in the short stepsize regime. 
    This construction applied to $\emptySchedule$ recovers the stepsize schedule considered in \citet[Corollary 2.19]{rotaru2024exact} and exactly recovers its g-convergence rate. 
\begin{corollary}
    The schedule $\emptySchedule^{\triangleleft n}$ is the same as that defined in \citet[Corollary 2.19]{rotaru2024exact}.
\end{corollary}
\begin{proof}
    The recurrence relation above agrees with that defined in \citet[Corollary 2.19]{rotaru2024exact}, and also, as a base case, both yield the length 1 stepsize schedule $[\frac{3}{2}]$.\qedhere
\end{proof}
As a consequence of the H duality built into our theory, it is immediate that the reversed dynamic short stepsize schedule $\rev(\emptySchedule^{\triangleleft n})$ achieves a matching objective gap rate.

On the other hand, we can obtain a strictly better stepsize schedule in the short stepsize regime by using the ``seed'' $\sigma = \emptySchedule \triangleleft (\emptySchedule \Join \emptySchedule)$ instead of $\emptySchedule$.
\begin{lemma}
    Every entry of $\sigma^{\triangleleft n}$ is in $(0,2)$, and the rate of $\sigma^{\triangleleft n}$ is strictly smaller than that of $(\emptySchedule)^{\triangleleft n}$ for each $n \ge 2$.
\end{lemma}
\begin{proof}
    A direct computation shows
    \[\sigma = \emptySchedule \triangleleft (\emptySchedule\Join \emptySchedule) = \left[\frac{1}{4}\left(3+\sqrt{9+8\sqrt{2}}\right), \sqrt{2}\right] \approx\left[1.876, 1.414\right]. \]
    In particular, since $\sigma$ has entries in $(0,2)$, every entry of $\sigma^{\triangleleft n}$ is at most 2.

    The rate of $\sigma$ is $\frac{2}{\sqrt{9+8\sqrt{2}} +4\sqrt{2} + 5} = 0.131892$, which is less than the rate of $(\emptySchedule)^{\triangleleft 2}$ being $0.174458$. Since for each $n > 2$, $\eta_n(h) = \eta_{n-1}(h) \triangleleft 1$, which is strictly monotonically increasing in $\eta_{n-1}(h)$, we have that for every $n > 2$, $\eta_n(\sigma)<\eta_n(\emptySchedule)$.\qedhere
\end{proof}
 \section{Optimized Basic Stepsize Schedules (OBS)}\label{sec:schedules}

Beyond recovering prior stepsize schedules, one can leverage the structure of our composition operations to optimize the construction of new schedules.
Computing the optimal basic \scomposable{}, \fcomposable{}, and \gcomposable{} schedules can be done efficiently via dynamic programming: a simple Julia implementation producing optimized stepsize schedules for any $n$ is available at \url{https://github.com/bgrimmer/OptimizedBasicSchedules}.

In the following three subsections, we address each of these optimal constructions and provide nearly matching upper and lower bounds on the worst-case performance of these optimized schedules.

\subsection{The Optimized Basic \scomposable{} Schedule (OBS-S)} \label{sec:obs-s}
We define an Optimized Basic \scomposable{} Schedule (OBS-S) to be a basic \scomposable{} schedule of a given length which has the minimum rate out of all basic \scomposable{} schedules of that length.

The following definitions and results are stated for a schedule of length $n-1$ with $n\geq 1$ (as opposed to a schedule of length $n$ with $n\geq 0$).
This will be natural for the following results where composition plays a key role. Indeed, recall that a schedule of length $n-1$ composed with a schedule of length $m-1$ results in a schedule of length $n+m-1$.

Below, we give a recursive construction of an Optimized Basic \scomposable{} Schedule (OBS-S), $h^\mathtt{OBS\mbox{-}S}(n-1) \in\mathbb{R}^{n-1}_{++}$, and its rate $\eta^\mathtt{OBS\mbox{-}S}(n-1)$.
Begin by setting $h^{\mathtt{OBS\mbox{-}S}}(0) = \emptySchedule$ and $\eta^{\mathtt{OBS\mbox{-}S}}(0) = 1$.
For each $n > 1$, we will let $h^{\mathtt{OBS\mbox{-}S}}(n-1) = h^{\mathtt{OBS\mbox{-}S}}(m-1) \Join h^{\mathtt{OBS\mbox{-}S}}(n-m-1)$, where we choose $1\le m\le n-1$ to minimize the associated rate $\eta^\mathtt{OBS\mbox{-}S}(n-1) = \eta^\mathtt{OBS\mbox{-}S}(m-1) \Join \eta^{\mathtt{OBS\mbox{-}S}}(n-m-1)$. Practically, we can find the value of $m$ minimizing this rate using dynamic programming.

Formally, we define
\begin{equation}
    \eta^\mathtt{OBS\mbox{-}S}(n-1) \coloneqq \min_{1\leq m\leq n-1}\eta^\mathtt{OBS\mbox{-}S}(m-1) \Join \eta^\mathtt{OBS\mbox{-}S}(n-m-1)\label{eq:OBS-S-construction-eta},
\end{equation}
and set $h^\mathtt{OBS\mbox{-}S}(n-1)$ to be a corresponding stepsize schedule 
\begin{equation}
    h^\mathtt{OBS\mbox{-}S}(n-1) \coloneqq h^\mathtt{OBS\mbox{-}S}(m-1) \Join h^\mathtt{OBS\mbox{-}S}(n- m-1), \label{eq:OBS-S-construction-h}
\end{equation}
where $m$ is any $\argmin$ of \eqref{eq:OBS-S-construction-eta}.

It is clear by induction that $h^{\mathtt{OBS\mbox{-}S}}(n-1)$ will then be basic. It will also have the minimum possible rate out of any basic schedule because $\alpha \Join \beta$ is an increasing function in $\alpha$ and $\beta$ (\cref{lem:composition_basic_facts}).
We also note that $\eta^\mathtt{OBS\mbox{-}S}(n) < \eta^\mathtt{OBS\mbox{-}S}(n-1)$, since $h^{\mathtt{OBS\mbox{-}S}}(n-1) \Join \emptySchedule$ has rate $\eta^\mathtt{OBS\mbox{-}S}(n-1) \Join 1 <  \eta^\mathtt{OBS\mbox{-}S}(n-1)$ by \cref{lem:composition_basic_facts}.

As we will see, when $n = 2^k$ for some $k$, then $h^{\mathtt{OBS\mbox{-}S}}(n-1)$ is exactly the silver stepsizes of~\cite{altschuler2023accelerationPartII}.
Our construction~\eqref{eq:OBS-S-construction-h} can then be viewed as a generalization of the silver stepsizes of arbitrary length.
Motivated by the construction of the silver stepsize schedules, the next lemma considers the s-convergence rate achieved by composing a schedule with itself.
\begin{lemma}
    \label{lem:easy_construction}
    Fix some $m \geq 1$ and $k \ge 0$. For all $n \ge 2^km$,
    \[
        \eta^{\mathtt{OBS\mbox{-}S}}(n-1) \le 
        \frac{1}{(\vr)^{k}}\eta^{\mathtt{OBS\mbox{-}S}}(m-1).
    \]
\end{lemma}
\begin{proof}
    It is clear in light of \cref{lem:composition_basic_facts} that 
    \[
        \eta^{\mathtt{OBS\mbox{-}S}}(2m-1) \le \eta^{\mathtt{OBS\mbox{-}S}}(m-1) \Join \eta^{\mathtt{OBS\mbox{-}S}}(m-1) = \frac{1}{\vr}\eta^{\mathtt{OBS\mbox{-}S}}(m-1).
    \]
    We then see that by repeatedly applying this fact,
    \begin{align*}
        \eta^{\mathtt{OBS\mbox{-}S}}(n-1) &\le \eta^{\mathtt{OBS\mbox{-}S}}(2^km-1)\\
                                        &\le \frac{1}{\vr}\eta^{\mathtt{OBS\mbox{-}S}}(2^{k-1}m-1)\\
                                        &\leq\dots\\
                                        &\leq\frac{1}{(\vr)^k}\eta^{\mathtt{OBS\mbox{-}S}}(m-1).\qedhere
    \end{align*}
\end{proof}
This lemma allows us to control the asymptotic rate of growth of $\eta^{\mathtt{OBS\mbox{-}S}}(n-1)$, so long as we have computed this value explicitly for enough small values of $n$ to use as ``seeds'' for the above construction.

The next lemma uses this construction to provide nearly tight uniform upper and lower bounds on the rates of all OBS-S schedules. It is stated in terms of a constant
\begin{equation}
	R^\mathtt{OBS\mbox{-}S}_{k} = \max_{n \in[2^k,2^{k+1})}\eta^{\mathtt{OBS\mbox{-}S}}(n-1)n^{\log_2(1+\sqrt{2})},\ \label{eq:def_R_OBS-S} 
\end{equation}
defined for any $k\geq 0$. For reference, $R^\mathtt{OBS\mbox{-}S}_{18}\approx 1.00723$.

\begin{theorem}\label{thm:OBS-S-rate}
    For all $n \ge 1$, the \scomposable{} rate of $h^\mathtt{OBS\mbox{-}S}(n-1)$ is lower bounded by
    \[
        \eta^\mathtt{OBS\mbox{-}S}(n-1)\geq \frac{1}{n^{\log_2(1+\sqrt{2})}}.
    \]
    If $n$ is a power of two, this inequality holds with equality and is attained by the silver stepsize schedule.

    Moreover, if $k \ge 0$ is an integer so that $n \ge 2^{k+1}$, then we also have an upper bound 
	$$\eta^\mathtt{OBS\mbox{-}S}(n-1) \le \frac{R^\mathtt{OBS\mbox{-}S}_{k}}{n^{\log_2(1+\sqrt{2})}} \left(1+\frac{1}{2^k}\right)^{\log_2(1+\sqrt{2})}. $$
\end{theorem}
\begin{proof}
	We prove the lower bound for all $n\geq 1$ inductively. For $n=1$, the base case follows as the empty schedule $h=\emptySchedule$ is the only basic schedule of length $n-1=0$ and achieves rate $1$.
    Now suppose $n>1$ and the lower bound holds for all schedules of length $m-1=0,\dots, n-2$. Letting $r=n^{\log_2(1+\sqrt{2})}$, we may prove the desired lower bound as follows:
	\begin{align}
		\eta^\mathtt{OBS\mbox{-}S}(n-1) & = \min_{m=1,\dots n-1} \eta^\mathtt{OBS\mbox{-}S}(m-1) \Join \eta^\mathtt{OBS\mbox{-}S}(n-m-1)\nonumber\\
		& = \frac{1}{r}\min_{m=1,\dots n-1} \left(r \eta^\mathtt{OBS\mbox{-}S}(m-1) \right)\Join \left(r \eta^\mathtt{OBS\mbox{-}S}(n-m-1)\right)\nonumber\\
		& \geq \frac{1}{r}\ \min_{m=1,\dots n-1} \frac{n^{\log_2(1+\sqrt{2})}}{m^{\log_2(1+\sqrt{2})}} \Join \frac{n^{\log_2(1+\sqrt{2})}}{(n-m)^{\log_2(1+\sqrt{2})}}\nonumber\\
		& \geq \frac{1}{r}\ \min_{0< \lambda < 1} \lambda^{-\log_2(1+\sqrt{2})} \Join (1-\lambda)^{-\log_2(1+\sqrt{2})}\label{eq:lower_bound_lambda}\\
		& = \frac{1}{n^{\log_2(1+\sqrt{2})}}\nonumber.
	\end{align} 
    The first inequality applies our inductive assumption; the second inequality reformulates and relaxes the minimization with $\lambda=\frac{m}{n}$, and the final equality notes this univariate function of $\lambda$ is minimized at $\lambda=\frac{1}{2}$ with value $1$.  

        To see that the silver stepsize schedule $\pi^{(k)}$ meets this inequality with equality when $n = 2^k$, we use \cref{lem:silver_steps_basic} to see that they are indeed basic with rate 
        \[
            \frac{1}{(1+\sqrt{2})^k} = \frac{1}{n^{\log_2(1+\sqrt{2})}}
        \]
        and length $\len(\pi^{(k)})=n-1$.
        This in particular implies that $\eta^\mathtt{OBS\mbox{-}S}(2^k-1) = \frac{1}{(1+\sqrt{2})^k}$ for any integer $k \ge 0$, since this rate is obtained by the silver stepsize schedule and we have just shown that no basic stepsize schedule can achieve a better rate.

Next we prove the upper bound using \cref{lem:easy_construction}.
        For this, let $k' = \lfloor\log_2(n)\rfloor$, and consider $m = \lfloor \frac{n}{2^{k'-k}} \rfloor$.
        Then $m \in [2^{k}, 2^{k+1})$, and $2^{k'-k} (m + 1) \ge n \ge 2^{k'-k} m$, where $k' - k >0$.
        We may then apply \cref{lem:easy_construction} and the bound that $R^\mathtt{OBS\mbox{-}S}_k \ge \eta^\mathtt{OBS\mbox{-}S}(m-1)m^{\log_2(1+\sqrt{2})}$ to see that 
        \begin{align*}
            \eta^\mathtt{OBS\mbox{-}S}(n-1) & \leq \frac{1}{(1+\sqrt{2})^{k'-k}}\eta^\mathtt{OBS\mbox{-}S}(m-1)\\
            & \leq \frac{R^\mathtt{OBS\mbox{-}S}_k}{m^{\log_2(1+\sqrt{2})} (1+\sqrt{2})^{k'-k}} \\
            & \leq \frac{R^\mathtt{OBS\mbox{-}S}_k}{(2^{k'-k}m)^{\log_2(1+\sqrt{2})}}\\
            & = \frac{R^\mathtt{OBS\mbox{-}S}_k}{n^{\log_2(1+\sqrt{2})}}\left(\frac{n}{2^{k'-k}m}\right)^{\log_2(\vr)}\\
            & \leq \frac{R^\mathtt{OBS\mbox{-}S}_k}{n^{\log_2(1+\sqrt{2})}} \left(1+\frac{1}{2^k}\right)^{\log_2(1+\sqrt{2})}.
        \end{align*}
        This last inequality uses the fact that $\frac{m+1}{m} =  1+\frac{1}{m} \le 1+\frac{1}{2^k}$.\qedhere
\end{proof}

We now return to the more standard setting where the stepsize schedule $h^\mathtt{OBS\mbox{-}S}_i(n)$ has length $n$ for $n\geq0$.
Applying~\cref{prop:fg_rates_of_scomp}, this theorem establishes a uniform bound for every $n\geq 0$ and the corresponding OBS-S on final objective gap and gradient norm for $1$-smooth convex functions.
For this, we note that 
\[
    1+2\sum_{i=0}^{n-1} h^\mathtt{OBS\mbox{-}S}_i(n) = 
    2\left(1+\sum_{i=0}^{n-1} h^\mathtt{OBS\mbox{-}S}_i(n)\right)-1 = 
    \frac{2}{\eta^\mathtt{OBS\mbox{-}S}(n)}-1,
\]
where this second equation uses the definition of \scomposable{} schedules.

Therefore, \cref{prop:fg_rates_of_scomp} can be rewritten to state that
\begin{gather*}
    f(x_n) - f(x_\star) \leq \left(\frac{1}{\frac{2}{\eta^\mathtt{OBS\mbox{-}S}(n)}-1}\right)\frac{\|x_0-x_\star\|^2}{2} \qquad\text{and} \\
\frac{1}{2}\|\nabla f(x_n)\|^2 \leq \left(\frac{1}{\frac{2}{\eta^\mathtt{OBS\mbox{-}S}(n)}-1}\right)\ (f(x_0)-f(x_\star)).
\end{gather*}
Substituting our bounds for $\eta^\mathtt{OBS\mbox{-}S}(n)$ produces that for any $n\geq 2^{18}-1$, any optimal basic \scomposable{} schedule simultaneously has
\begin{gather*}
    f(x_n) - f(x_\star) \leq \left(\frac{1.00724}{2(n+1)^{\log_2(1+\sqrt{2})} - 1.00724}\right)\frac{\|x_0-x_\star\|^2}{2} \qquad\text{and} \\
    \frac{1}{2}\|\nabla f(x_n)\|^2 \leq \left(\frac{1.00724}{2(n+1)^{\log_2(1+\sqrt{2})} - 1.00724}\right)\ (f(x_0)-f(x_\star)).
\end{gather*}

Moreover, since $h^\mathtt{OBS\mbox{-}S}(n)$ attains the minimum rate over all basic schedules, the lower bound of Theorem~\ref{thm:OBS-S-rate} establishes that no basic \scomposable{} schedule of length $n$ can achieve an \scomposable{} rate smaller than $1/(n+1)^{\log_2(1+\sqrt{2})}$. The fact that the lower bound in \eqref{eq:lower_bound_lambda} minimizes with $\lambda=1/2$ above highlights the importance and necessity of symmetry for optimal basic \scomposable{} schedules.

\subsection{The Optimized Basic \fcomposable{} Schedule (OBS-F)} \label{sec:obs-f}
Mirroring the construction of OBS-S schedules above, we define an Optimized Basic \fcomposable{} Schedule (OBS-F) to be any basic \fcomposable{} schedule attaining the minimum rate among all basic \fcomposable{} schedules of a given length. 
We will again find a particular OBS-F $h^\mathtt{OBS\mbox{-}F}(n-1)\in\mathbb{R}^{n-1}_{++}$ with rate $\eta^\mathtt{OBS\mbox{-}F}(n-1)$.

We will again note that $\eta^\mathtt{OBS\mbox{-}F}(n) < \eta^\mathtt{OBS\mbox{-}F}(n-1)$, since $\emptySchedule \triangleright h^{\mathtt{OBS\mbox{-}F}}(n-1)$ is an \fcomposable{} basic schedule of length $n$ with rate $(1\triangleright\eta^\mathtt{OBS\mbox{-}F}(n-1) ) <  \eta^\mathtt{OBS\mbox{-}F}(n-1)$ by \cref{lem:composition_basic_facts}.
Noting that the rate $\alpha \triangleright \beta$ is increasing in both arguments, these schedules are readily computable using dynamic programming. Define $h^\mathtt{OBS\mbox{-}F}(0)=\emptySchedule$ and $\eta^\mathtt{OBS\mbox{-}F}(0)=1$. Then, define
\begin{equation}
    \eta^\mathtt{OBS\mbox{-}F}(n-1) \coloneqq \min_{1\leq m\leq n-1}\eta^\mathtt{OBS\mbox{-}S}(m-1) \triangleright \eta^\mathtt{OBS\mbox{-}F}(n-m-1), \label{eq:OBS-F-eta}
\end{equation}
and set $h^\mathtt{OBS\mbox{-}F}(n)$ to be a corresponding stepsize schedule
\begin{equation}
	h^\mathtt{OBS\mbox{-}F}(n-1) \coloneqq h^\mathtt{OBS\mbox{-}S}(m-1) \triangleright h^\mathtt{OBS\mbox{-}F}(n- m - 1), \label{eq:OBS-F-h}
\end{equation}
where $m$ is any $\argmin$ of \eqref{eq:OBS-F-eta}.

    As above, we will give a simple recursive construction for basic \fcomposable{} stepsize schedules.
\begin{lemma}
    \label{lem:easy_f_construction}
    Fix some $m \ge 1$ and $k \ge 1$. For all $n \ge 2^km$,
    \[
        \eta^{\mathtt{OBS\mbox{-}F}}(n-1) \le 
        \frac{1}{(\vr)^{k}}\eta^{\mathtt{OBS\mbox{-}F}}(m-1).
    \]
\end{lemma}
\begin{proof}
    We will first show this in the case $k = 1$, where we need to establish
    \[
        \eta^{\mathtt{OBS\mbox{-}F}}(2m-1) \le 
        \frac{1}{\vr}\eta^{\mathtt{OBS\mbox{-}F}}(m-1).
    \]
    We will do this by strong induction on $m$.
    Note for $m=1$, 
    \[
        \eta^\mathtt{OBS\mbox{-}F}(1) = \frac{1}{4} \leq \frac{1}{1+\sqrt{2}}=\frac{1}{1+\sqrt{2}}\eta^\mathtt{OBS\mbox{-}F}(0).
    \]
    Now, let $m\geq 2$ and set $m'\in[1, m-1]$ so that $\eta^\mathtt{OBS\mbox{-}F}(m - 1) = \eta^\mathtt{OBS\mbox{-}S}(m' - 1) \triangleright\eta^\mathtt{OBS\mbox{-}F}(m - m' - 1)$. Then,
    \begin{align*}
        \eta^\mathtt{OBS\mbox{-}F}(2m-1) &\leq \eta^\mathtt{OBS\mbox{-}S}(2m' - 1) \triangleright \eta^\mathtt{OBS\mbox{-}F}(2(m-m') - 1)\\
        &\leq \left(\frac{1}{1+\sqrt{2}} \eta^\mathtt{OBS\mbox{-}S}(m' - 1)\right) \triangleright \left( \frac{1}{1+\sqrt{2}} \eta^\mathtt{OBS\mbox{-}F}(m-m' - 1)\right) \\
        &= \frac{1}{1+\sqrt{2}} \left(\eta^\mathtt{OBS\mbox{-}S}(m' - 1) \triangleright \eta^\mathtt{OBS\mbox{-}F}(m-m' - 1) \right)\\
        &= \frac{\eta^\mathtt{OBS\mbox{-}F}(m - 1)}{1+\sqrt{2}} \ . 
    \end{align*}
    In the second inequality, we use \cref{lem:easy_construction} and the inductive hypothesis, and the first equality is by \cref{lem:composition_basic_facts}.

    For $k > 1$, we simply apply the case $k=1$ repeatedly, as we have that
    \begin{align*}
        \eta^\mathtt{OBS\mbox{-}F}(2^km-1) &\leq \frac{1}{\vr}\eta^\mathtt{OBS\mbox{-}F}(2^{k-1}m - 1)\\
                                           &\leq\dots\\
                                           &\leq \frac{1}{(\vr)^k}\eta^\mathtt{OBS\mbox{-}F}(m - 1). \qedhere\\
    \end{align*}
\end{proof}

Under Conjecture~\ref{conj:strong-f-minimax-descripiton}, any stepsize schedule computed via this dynamic program is minimax optimal in the sense of~\eqref{eq:minimax}. As we did with OBS-S schedules, we will establish Theorem~\ref{thm:OBS-F-rate} to provide nearly tight uniform upper and lower bounds on the convergence rate of these conjectured minimax optimal schedules. Our upper bounds will be stated in terms of the constant
\begin{equation}
    R^\mathtt{OBS\mbox{-}F}_k = \max_{n \in[2^k,2^{k+1})}\eta^{\mathtt{OBS\mbox{-}F}}(n-1)n^{\log_2(1+\sqrt{2})} \ , \label{eq:def_R_OBS-F}
\end{equation}
defined for any $k\geq 0$. For reference, $R^\mathtt{OBS\mbox{-}F}_{18}\approx 0.42311$.
Our lower bounds will be stated in terms of a constant $c_\mathtt{low}$ defined implicitly by the below formula.
\begin{equation}
    c_\mathtt{low} = \min_{0\leq \lambda \leq 1} \left(\lambda^{-\log_2(1+\sqrt{2})}\right) \triangleright \left(c_\mathtt{low} (1-\lambda)^{-\log_2(1+\sqrt{2})}\right). \label{eq:def_c_low}
\end{equation}
For reference, $c_\mathtt{low} \approx 0.4208$.

\begin{theorem}\label{thm:OBS-F-rate}
    For all $n\geq 1$, the \fcomposable{} rate of $h^\mathtt{OBS\mbox{-}F}(n-1)$ is lower bounded by
 \begin{equation*}
    \eta^\mathtt{OBS\mbox{-}F}(n-1)\geq \frac{c_\mathtt{low}}{n^{\log_2(1+\sqrt{2})}}.
 \end{equation*}
 If $k\geq 0$ is an integer so that $n\geq 2^{k+1}$, then we also have an upper bound
	$$ \eta^\mathtt{OBS\mbox{-}F}(n-1) \leq \frac{R^\mathtt{OBS\mbox{-}F}_k}{n^{\log_2(1+\sqrt{2})}} \left(1 + \frac{1}{2^k}\right)^{\log_2(1+\sqrt{2})} \ . $$
\end{theorem}
\begin{proof}	
	First we prove the lower bound for all $n\geq 1$ inductively. For $n=1$, the base case follows as the empty schedule $h=\emptySchedule$ is the only basic schedule of length $n-1=0$ and achieves rate $1\geq c_\mathtt{low}$.
	Supposing the lower bound holds for all schedules of length $m-1=0,\dots, n-2$, the lower bound at $n-1$ follows as
	\begin{align*}
		\eta^\mathtt{OBS\mbox{-}F}(n-1) & = \min_{m=1,\dots n-1} \eta^\mathtt{OBS\mbox{-}S}(m-1) \triangleright \eta^\mathtt{OBS\mbox{-}F}(n-m-1)\\
		&\geq \min_{m=1,\dots n-1} \left(\frac{1}{m^{\log_2(1+\sqrt{2})}}\right) \triangleright \left(\frac{c_\mathtt{low}}{(n-m)^{\log_2(1+\sqrt{2})}}\right)\\
		& = \frac{1}{n^{\log_2(1+\sqrt{2})}}\ \min_{m=1,\dots n-1} \frac{n^{\log_2(1+\sqrt{2})}}{m^{\log_2(1+\sqrt{2})}} \triangleright \frac{c_\mathtt{low} n^{\log_2(1+\sqrt{2})}}{(n-m)^{\log_2(1+\sqrt{2})}}\\
		& \geq \frac{1}{n^{\log_2(1+\sqrt{2})}}\ \min_{0\leq \lambda \leq 1} \left(\lambda^{-\log_2(1+\sqrt{2})}\right) \triangleright \left(c_\mathtt{low}(1-\lambda)^{-\log_2(1+\sqrt{2})}\right)\\
		& = \frac{c_\mathtt{low}}{n^{\log_2(1+\sqrt{2})}}
	\end{align*} 
	where the first inequality applies Theorem~\ref{thm:OBS-S-rate} and our inductive assumption, the second equality is by Lemma~\ref{lem:composition_basic_facts}, the second inequality reformulates with $\lambda=\frac{m}{n}$ and relaxes the minimization, and the final equality uses the property defining our selection of $c_\mathtt{low}$.

    We prove the upper bound using \cref{lem:easy_f_construction} (similarly to the case of \scomposable{} schedules).
        For this, let $k' = \lfloor\log_2(n)\rfloor$, and consider $m = \lfloor \frac{n}{2^{k'-k}} \rfloor$.
        Then $m \in [2^{k}, 2^{k+1})$, and $2^{k'-k} (m + 1) \ge n \ge 2^{k'-k} m$, where $k' - k >0$.
        We may then apply the lemma and the bound that $R^\mathtt{OBS\mbox{-}F}_k \ge \eta^\mathtt{OBS\mbox{-}F}(m-1)m^{\log_2(1+\sqrt{2})}$ to see that 
        \begin{align*}
            \eta^\mathtt{OBS\mbox{-}F}(n-1) & \leq \frac{1}{(1+\sqrt{2})^{k'-k}}\eta^\mathtt{OBS\mbox{-}F}(m-1)\\
            & \leq \frac{R^\mathtt{OBS\mbox{-}S}_k}{m^{\log_2(1+\sqrt{2})} (1+\sqrt{2})^{k'-k}} \\
            & = \frac{R^\mathtt{OBS\mbox{-}S}_k}{(2^{k'-k}m)^{\log_2(1+\sqrt{2})}}\\
            & =\frac{R^\mathtt{OBS\mbox{-}S}_k}{n^{\log_2(1+\sqrt{2})}}\left(\frac{n}{2^{k'-k}m}\right)^{\log_2(\vr)}\\
            & \leq \frac{R^\mathtt{OBS\mbox{-}S}_k}{n^{\log_2(1+\sqrt{2})}} \left(1+\frac{1}{2^k}\right)^{\log_2(1+\sqrt{2})}.
        \end{align*}
        This last inequality uses the fact that $\frac{m+1}{m} =  1+\frac{1}{m} \le 1+\frac{1}{2^k}$.\qedhere

\end{proof}

Again, this statement can be interpreted to give performance guarantees for OBS-F schedule of length $n$ where $n\geq 2^{18} - 1$: On any $1$-smooth convex $f$ and initialization $x_0$, gradient descent with stepsize schedule $h^\mathtt{OBS\mbox{-}F}(n)$ satisfies
$$ f(x_n) - f(x_\star) \leq \left(\frac{0.42312}{(n+1)^{\log_2(1+\sqrt{2})}}\right)\frac{\norm{x_0 -x_\star}^2}{2}\ . $$
On the other hand, for any basic \fcomposable{} schedule  $h$ of length $n$, there exists a $1$-smooth convex $f$ and $x_0$ such that gradient descent with stepsize schedule $h$ satisfies 
$$ f(x_n) - f(x_\star) \geq \left(\frac{0.4208}{(n+1)^{\log_2(1+\sqrt{2})}}\right)\frac{\norm{x_0 - x_\star}^2}{2} \ . $$

\begin{remark}
\label{rem:compare_rates}
\citet{zhang2024acceleratedgradientdescentconcatenation} analyze an identical construction and give an asymptotic upper bound on $\eta^\mathtt{OBS\mbox{-}F}(n)$. However, they use a suboptimal proof strategy that is unable to get the sharp asymptotic constant in \cref{thm:OBS-F-rate}.\qedhere
\end{remark}

\subsection{The Optimized Basic \gcomposable{} Schedule (OBS-G)} \label{sec:obs-g}

Our H-Duality theory from \cref{thm:HDual} establishes an exact correspondence between basic \fcomposable{} and basic \gcomposable{} schedules. As a result, the basic \gcomposable{} schedule with minimum rate is $h^\mathtt{OBS\mbox{-}G}(n) = \rev(h^\mathtt{OBS\mbox{-}F}(n))$. Then convergence guarantees for gradient norm convergence follow immediately from Theorem~\ref{thm:OBS-F-rate}.
\begin{corollary}\label{cor:OBS-G-rate}
    Let $n\geq 1$. The OBS-G schedule $h^\mathtt{OBS\mbox{-}G}(n)$ has \gcomposable{} rate lower bounded by
\begin{equation*}
    \eta^\mathtt{OBS\mbox{-}G}(n) \geq \frac{c_\mathtt{low}}{(n+1)^{\log_2(1+\sqrt{2})}}.
\end{equation*}
If $n\geq 2^{k+1}-1$ for some integer $k\geq 0$, then,
    $$ \eta^\mathtt{OBS\mbox{-}G}(n) \leq \frac{R^\mathtt{OBS\mbox{-}G}_k}{(n+1)^{\log_2(1+\sqrt{2})}} \left(1 + \frac{1}{2^k}\right)^{\log_2(1+\sqrt{2})} $$
    where $R^\mathtt{OBS\mbox{-}G}_k=R^\mathtt{OBS\mbox{-}F}_k$. Hence, for any $1$-smooth convex $f$ and $n\geq 2^{18}-1$, gradient descent with stepsize schedule $h^\mathtt{OBS\mbox{-}G}(n)$ has
	$$ \frac{1}{2}\|\nabla f(x_n)\|^2 \leq \left(\frac{0.42312  }{(n+1)^{\log_2(1+\sqrt{2})}}\right)  (f(x_0)-f(x_\star)) \ . $$
	Moreover, for any basic \gcomposable{} schedule $h$ of length $n$, there exists a $1$-smooth convex $f$ and $x_0$ such that
	$$ \frac{1}{2}\|\nabla f(x_n)\|^2 \geq \left(\frac{0.4208}{(n+1)^{\log_2(1+\sqrt{2})}}\right)  (f(x_0)-f(x_\star))\ . $$
\end{corollary}

Complementary to our Conjecture~\ref{conj:strong-f-minimax-descripiton}, we expect that the minimax optimal stepsizes for minimizing the final gradient norm are basic \gcomposable{} schedules.
\begin{conjecture}\label{conj:strong-g-minimax-descripiton}
For each $n$, every minimax optimal stepsize schedule solving 
$$ \min_{h \in \mathbb{R}^n} \max_{(f,x_0)\in \mathfrak{F}_{1,\delta}} \frac{1}{2}\|\nabla f(x_n)\|^2, $$
is basic and \gcomposable{}, where $\mathfrak{F}_{1,\delta}$ is the set of all problem instances $(f,x_0)$ defined by a $1$-smooth convex $f$ and initialization $x_0$ having suboptimality $f(x_0)-f(x_\star)$ at most $\delta$.
\end{conjecture}
 \section{Proofs for Inductive Composition Theorems}\label{sec:proofs}

Following the PEP framework~\cite{drori2012PerformanceOF,taylor2017interpolation}, our goal will be to certify convergence rates of GD using the following inequality, specifically applying it to the points on the gradient descent trajectory and then taking a conic combination of the resulting inequalities. This inequality is a standard fact~\cite[Theorem 2.1.5]{nesterov-textbook} for smooth convex functions.
\begin{fact}
	\label{fact}
	Suppose $f:\R^d\to\R$ is a $1$-smooth convex function. Suppose $x,y\in\R^d$, then
	\begin{equation*}
		2\left[f(x) - f(y)\right] - 2\ip{\grad f(y), x - y } - \norm{\grad f(x) - \grad f(y)}^2 \geq 0.
	\end{equation*}
\end{fact}

\subsection{Useful Equivalent Definitions of Composablity}

The following lemma captures the background information we will use about the PEP framework. It states that any inequality that is linear in $f_\star,f_0,\dots,f_n$ and quadratic in $x_0-x_\star,g_0,\dots,g_n$ and tight over the set of $1$-smooth convex functions has a certificate.
This result is a minor extension of~\cite[Theorem 6]{taylor2017interpolation}.
\begin{lemma}
	\label{lem:pep_exists}
	Fix $h\in\R^{n}_{++}$.
	Let $P$ be an expression that is linear in the formal variables
	\begin{align*}
		f &= \begin{pmatrix}
			f_\star & f_0 & f_1 & \dots & f_n
		\end{pmatrix}^\intercal,\quad\text{and}\\
		G &= \begin{pmatrix}
			(x_0 - x_\star) & g_0 & g_1 & \dots & g_n
		\end{pmatrix}^\intercal\begin{pmatrix}
			(x_0 - x_\star) & g_0 & g_1 & \dots & g_n
		\end{pmatrix}.
	\end{align*}
    For $i\in\set{\star,0,\dots,n}$ and $j\in\set{\star,0,\dots,n}$, define the formal expressions
    \begin{equation}
    \label{eq:qdef}
        Q_{i,j}\coloneqq 2f_i - 2f_j - 2\ip{g_j, x_i - x_j} - \norm{g_i - g_j}^2.
    \end{equation}
 
	Suppose $P$ is nonnegative for all $1$-smooth convex functions $f$ and that there exists some $1$-smooth convex function so that $P = 0$.
	Then, there exist $\lambda_{i,j}\geq 0$ and $S$, a PSD quadratic form in $x_0 - x_\star, g_0, \dots, g_n$ such that
	\begin{equation*}
		P = \sum_{i,j}\lambda_{i,j}Q_{i,j} + S.
	\end{equation*} 
\end{lemma}
\begin{proof}
Let $\cD_\textup{PEP}$ denote the set of $(f,G)\in\R^{n+2}\times\R^{(n+2) \times (n+2)}$ (where $G$ is symmetric) for which there exists a $1$-smooth convex function $\tilde f:\R^d\to\R$ realizing $f$ and $G$ (for an arbitrary $d$).
By \cite[Corollary 1]{taylor2017smooth},
\begin{equation*}
    \cD_\textup{PEP} = \set{(f,G)\in\R^{n+2}\times\S^{n+2}:\,
    \begin{array}{l}
         Q_{i,j}\geq 0,\, \forall i,j\in\set{\star,0,\dots,n}\\
         G\succeq 0 
    \end{array}}.
\end{equation*}

	By assumption,
	\begin{align*}
		0 &= \min_{(f, G) \in\cD_\textup{PEP}}P(f,G)\\
		&= \max_{\lambda\in\R^{(n+2)\by(n+2)}, S\in\S^{n+2}}\set{0:\, \begin{array}{l}
				\sum_{i,j}\lambda_{i,j}Q_{i,j}  + S = P\\
				\lambda\geq 0\\
				S\succeq 0
			\end{array}
		}.
	\end{align*}
	Here, the second line follows by the fact that $\cD_{\textup{PEP}}$ is strictly feasible~\cite[Theorem 6]{taylor2017interpolation} so that strong duality holds and the program on the second line has a maximizer (and hence a feasible solution).\qedhere
\end{proof}

Below, we give equivalent conditions for \fcomposable{}, \gcomposable{}, and \scomposable{} schedules. In contrast to the original definitions, which measure performance against optimality, the definitions below measure performance only against history that has already been seen and are the key inequalities we will need to perform our inductive proofs.

The following lemma will be useful.
\begin{lemma}
	\label{lem:positive_Qstari}
	Suppose $h\in\R^{n}_{++}$ and let $H\coloneqq \sum_{i=0}^{n-1}h_i$. Consider gradient descent with stepsize $h$ from $x_0 = 1$ on
	\begin{equation*}
		f(x) = \begin{cases}
			\frac{1}{2}x^2&\text{if }\abs{x}\leq \eta\\
			\eta\abs{x} - \frac{\eta^2}{2} & \text{else}
		\end{cases}.
	\end{equation*}
	If $\eta \in(0, \frac{1}{1+H})$, then $Q_{i,\star}>0$ for all $i\in[0,n]$.
	If $\eta = \frac{1}{1+H}$, then $Q_{i,\star}>0$ for all $i\in[0,n-1]$.
\end{lemma}
\begin{proof}
	Suppose $\eta \in(0,\frac{1}{1+H}]$. Note that
		$1 - \eta H \geq \eta$ so that $x_i > \eta$ for all $i\in[0,n-1]$ and that  $x_n > \eta$ if $\eta < \frac{1}{1+H}$.
	
	Now, suppose $x_i> \eta$. Then
	\begin{equation*}
		Q_{i,\star} = 2(f_i - f_\star) - \norm{g_i}^2 = 2\left(\eta x_i - \frac{\eta^2}{2}\right) - \eta^2= 2\eta (x_i-\eta) >0.\qedhere
	\end{equation*}
\end{proof}

\begin{proposition}
	\label{prop:composable_eq}
	Let $h\in\R^n_{++}$ and let $\eta>0$.
	Suppose
	\begin{equation*}
		\eta = \frac{1}{1+2\sum_{i=0}^{n-1} h_i}  = \prod_{i=0}^{n-1} (h_i - 1)^2.
	\end{equation*}
	Then, $h$ is \fcomposable{} with rate $\eta$ if and only if there exists a vector $v\in\R^{n+1}_+$ indexed by $[0,n]$ such that $\sum_{i=0}^n v_i = \frac{1}{\eta}$ and, for any $1$-smooth convex $f$ and $x_0$, it holds that
	\begin{equation}
		\label{eq:composable_eq}
		\sum_{i=0}^n v_i(2(f_i - f_n) + \norm{g_i}^2 + 2\ip{g_i, x_0 - x_i}) - \norm{\sum_{i=0}^{n}v_i g_i}^2 \ge 0.
	\end{equation}
\end{proposition}

\begin{proof}

	First, suppose a vector $v$ with the stated properties exists and let $f$ be a $1$-smooth convex function with minimizer $x_\star$. Then,
	\begin{align*}
		0&\leq \sum_{i=0}^n v_i(2(f_i - f_n) + \norm{g_i}^2 + 2\ip{g_i, x_0 - x_i}) - \norm{\sum_{i=0}^{n}v_i g_i}^2 \\
		&\qquad + \sum_{i=0}^n v_i Q_{\star, i}\\
		&= \sum_{i=0}^n v_i(2(f_\star - f_n) + 2\ip{g_i, x_0 - x_\star}) - \norm{\sum_{i=0}^{n}v_i g_i}^2 \\
		&= \frac{2}{\eta}(f_\star - f_n) + 2\ip{\sum_{i=0}^n v_i g_i, x_0 - x_\star} - \norm{\sum_{i=0}^n v_ig_i}^2\\
		&= \frac{2}{\eta}(f_\star - f_n) + \norm{x_0 - x_\star}^2 - \norm{x_0 - \sum_{i=0}^n v_i g_i - x_\star}^2.
	\end{align*}
	We conclude that
	\begin{equation*}
		f_n - f_\star \leq \eta\left(\frac{1}{2}\norm{x_0 - x_\star}^2 - \frac{1}{2}\norm{x_0 - \sum_{i=0}^n v_i g_i - x_\star}^2\right) \leq  \eta\frac{1}{2}\norm{x_0 - x_\star}^2.
	\end{equation*}
	
	Now, we turn to the forward direction. Suppose $h$ is \fcomposable{} with rate $\eta$.
	We handle the case $n = 0$ separately.
	If $n = 0$, then $h = [\hspace{0.5em}]$, and $\eta = 1$. We verify that the claim holds for $v = [1]\in\R^1_+$:
	\begin{equation*}
		2(f_0 - f_0) + \norm{g_0}^2 + 2\ip{g_0, x_0 - x_0} - \norm{v_0 g_0}^2 = 0 \geq 0.
	\end{equation*}
	
	Now, suppose $n\geq 1$.
	The definition of being \fcomposable{} implies that the expression
	\begin{equation*}
		\norm{x_0 - x_\star}^2 - \frac{2}{\eta}(f_n - f_\star)
	\end{equation*}
	is nonnegative for all $1$-smooth convex $f$ and achieves the value $0$ for the Huber function $H_\eta$. 
	By \cref{lem:pep_exists}, there exists $\lambda\in\R^{(n+2)\by (n+2)}$ and $S$ a PSD quadratic form so that
	\begin{equation}
		\label{eq:f_pep_cert}
		\norm{x_0 - x_\star}^2 - \frac{2}{\eta}(f_n - f_\star)= \sum_{i,j}\lambda_{i,j}Q_{i,j} + S.
	\end{equation}
	Define $v_i = \lambda_{\star,i}\geq 0$ for all $i\in[0,n]$.
	
	Now, consider \eqref{eq:f_pep_cert} for $f = H_\eta$. By assumption both the LHS and RHS evaluate to $0$.
	Note that
	\begin{equation*}
		\eta = \frac{1}{1+2\sum_{i=0}^{n-1}h_i} < \frac{1}{1+\sum_{i=0}^{n-1}h_i}
	\end{equation*}
	so that by \cref{lem:positive_Qstari} it holds that $Q_{i,\star} > 0$ for all $i\in[0,n]$. Thus, we deduce that $\lambda_{i,\star} = 0$ for all $i\in[0,n]$.
	
	We will overload notation and identify the quadratic form $S$ with a PSD matrix $S\in \S^{n+2}_+$ indexed by $\set{\star,0,1,\dots,n}$ so that
	\begin{equation*}
		S = \tr\left(\begin{pmatrix}
			x_0 - x_\star & g_0 & \dots & g_n
		\end{pmatrix} S \begin{pmatrix}
			x_0 - x_\star & g_0 & \dots & g_n
		\end{pmatrix}^\intercal\right)
	\end{equation*}
	Note that the coefficient on $\norm{x_0 - x_\star}^2$ on the LHS of \eqref{eq:f_pep_cert} is $1$. On the other hand, $\norm{x_0 - x_\star}^2$ does not appear in any of the $Q_{i,j}$ terms on the RHS of \eqref{eq:f_pep_cert}.
	We deduce that $S_{\star,\star} = 1$.
	Next, for $i \in[0,n]$, the coefficient on $\ip{g_i, x_0 - x_\star}$ in the LHS of \eqref{eq:f_pep_cert} is $0$.
	The coefficient on $\ip{g_i, x_0 - x_\star}$ in the RHS of \eqref{eq:f_pep_cert} is $2v_i + 2S_{\star,i}$. We deduce that $S_{\star,i} = -v_i$.
	
	Thus, by Schur complement lemma
	\begin{equation*}
		S = \begin{pmatrix}
			1 & -v^\intercal \\
			-v & Q
		\end{pmatrix} \succeq \begin{pmatrix}
			1 & -v^\intercal \\
			-v & vv^\intercal
		\end{pmatrix}.
	\end{equation*}
	
	Next, consider the coefficient on $f_\star$ in \eqref{eq:f_pep_cert}. On the LHS, the coefficient is $\frac{2}{\eta}$. On the RHS, it is $2\sum_{i=0}^n \left(\lambda_{\star,i}-\lambda_{i,\star}\right) = 2\sum_{i=0}^n v_i$. We deduce that $\sum_{i=0}^n v_i = \frac{1}{\eta}$.
	
	Finally, we compute
	\begin{align*}
		0&\leq \sum_{\substack{i,j\in[0,n]\\
				i\neq j}}\lambda_{i,j}Q_{i,j}\\
		&= \norm{x_0 - x_\star}^2 - \frac{2}{\eta}\left(f_n - f_\star\right) - S - \sum_{i=0}^n\lambda_{\star,i}Q_{\star,i}\\
		&\leq \norm{x_0 - x_\star}^2 - \frac{2}{\eta}\left(f_n - f_\star\right) - \norm{x_0 - x_\star - \sum_{i=0}^n v_i g_i}^2 \\
        &\qquad - \sum_{i=0}^n v_i \left(2(f_\star - f_i) - 2\ip{g_i, x_\star - x_i} - \norm{g_i}^2\right)\\
		&= \sum_{i=0}^n v_i \left(2(f_i - f_n) + 2\ip{g_i, x_0 - x_i} + \norm{g_i}^2\right) - \norm{\sum_{i=0}^n v_i g_i}^2.\qedhere
	\end{align*}
\end{proof}

\begin{remark}
	Since the analysis above bounds a term $\norm{x_0 - \sum_{i=0}^n v_i g_i - x_\star}^2$ by zero, every tight instance must minimize at this point, having $x_\star = x_0 - \sum_{i=0}^n v_i g_i$. We can interpret the quantity $x_0 - \sum_{i=0}^n v_i g_i$ as the output of a prox-step on a model of $f$ built on the first-order information seen by GD. Specifically, define the model
	\begin{equation*}
		m(x) \coloneqq \frac{\sum_{i=0}^n v_i \left(f_i +\ip{g_i, x- x_i} + \frac{1}{2}\norm{g_i}^2\right)}{\sum_{i=0}^n v_i}.
	\end{equation*}
	Then, $x_0 - \sum_{i=0}^n v_i g_i = \argmin_{x} \left(m(x) + \frac{1}{2}\norm{x - x_0}^2\right).$
	Noting that $m(x_\star) \leq f(x_\star)$, this aggregate dual model $m$ can be viewed as a dual certificate implicitly built by any \fcomposable{} schedule.
\end{remark}

A similar proof strategy allow us to give the following equivalent definitions of \gcomposable{} and \scomposable{} schedules. See Appendix \ref{ap:deferred} for proofs of \cref{prop:g_composable_eq,prop:self_comp_ineq}.
\begin{proposition}
	\label{prop:g_composable_eq}
	Let $h\in\R^n_{++}$ and let $\eta>0$. Suppose
	\begin{equation*}
		\eta = \frac{1}{1+2\sum_{i=0}^{n-1}h_i} = \prod_{i=0}^{n-1}(h_i-1)^2.
	\end{equation*}
    Then, $h$ is \gcomposable{} with rate $\eta$ if and only if for any $1$-smooth convex $f$ and $x_0$, gradient descent with stepsizes $h$ satsifies the inequality
	\begin{equation*}
		\eta(f_0-f_n) - \frac{1-\eta}{2}\norm{g_n}^2 \geq 0.
	\end{equation*}
\end{proposition}

\begin{proposition}
	\label{prop:self_comp_ineq}
	Let $h\in\R^{n}_{++}$ and let $\eta >0$. Suppose
	\begin{equation*}
		\eta = \frac{1}{1+\sum_{i=0}^{n-1}h_i} = \prod_{i=0}^{n-1}(h_i-1).
	\end{equation*}
	Then, $h$ is \scomposable{} with rate $\eta$ if and only if for any $1$-smooth convex $f$ and $x_0$, it holds that
	\begin{equation*}
		\sum_{i=0}^{n-1}h_i \left(2(f_i - f_n) + \norm{g_i}^2 + 2\ip{g_i, x_0 - x_i}\right) - \norm{x_n - x_0}^2 - \frac{1-\eta}{\eta^2} \norm{g_n}^2 \geq 0.
	\end{equation*}
\end{proposition}

\subsection{Proof of \cref{thm:f_recurrence}}
This subsection contains a proof of \cref{thm:f_recurrence}. Fix the notation of \cref{thm:f_recurrence} and for convenience let $h = a \triangleright b$, $\eta = \alpha\triangleright \beta$, and $n = n_a +n_b$. 
Note that $a\in\R^{n_a -1}$, $b\in\R^{n_b -1}$ and $h\in\R^{n-1}$.

Our goal is to show that $h$ is \fcomposable{} with rate $\eta$ using the conditions in \cref{prop:composable_eq}. let

The following claim collects useful algebraic identities relating $\alpha, \beta, \mu, \text{ and } \eta$.
\begin{claim}
	\label{lem:identities}
    Suppose that $a\in\R^{n_a -1}$, $b\in\R^{n_b -1}$, so that $a$ is \scomposable{} with rate $\alpha$, $b$ is \fcomposable{} with rate $\beta$. Let $n = n_a + n_b$ and let $h = a \triangleright b = [a, \mu, b] \in \R^{n-1}$ have \fcomposable{} rate $\eta$. Then the following identities hold:
	\begin{equation*}
		\sqrt{\frac{\beta}{\eta}} = \frac{\beta}{\eta} - \frac{2\beta}{\alpha} = \frac{\alpha\beta(\mu-1)}{\eta} 
	\end{equation*}
\end{claim}
\begin{proof}
    We first note that
    \begin{align}
		\frac{\beta}{\eta} - \frac{2\beta}{\alpha}&= \frac{\alpha + 4\beta + \sqrt{\alpha^2+8\alpha\beta}}{2\alpha} - \frac{2\beta}{\alpha}\nonumber\\
		&= \frac{\alpha + \sqrt{\alpha^2+8\alpha\beta}}{2\alpha}.\label{eq:identity_inner_1}
    \end{align}
    In particular, it is clear from this that $\frac{\beta}{\eta} - \frac{2\beta}{\alpha} \ge 0$.

	Therefore, the first identity follows from 
	\begin{align*}
		\left(\frac{\beta}{\eta} - \frac{2\beta}{\alpha}\right)^2&= \left(\frac{\alpha + \sqrt{\alpha^2+8\alpha\beta}}{2\alpha} \right)^2\\
		&= \frac{2\alpha^2 + 8\alpha\beta + 2\alpha \sqrt{\alpha^2 + 8 \alpha\beta}}{4\alpha^2}= \frac{\beta}{\eta}.\nonumber
	\end{align*}

	Next, we compute
	\begin{align*}
		\frac{\alpha\beta}{\eta}(\mu-1) &= \frac{\alpha + 4\beta + \sqrt{\alpha^2+8\alpha\beta}}{2} \cdot \frac{\sqrt{\alpha^2 + 8 \alpha\beta}-\alpha}{4\alpha\beta}\\
		&= \frac{\alpha\beta + \beta\sqrt{\alpha^2 + 8\alpha\beta}}{2\alpha\beta}= \frac{\alpha + \sqrt{\alpha^2 + 8\alpha\beta}}{2\alpha}.
	\end{align*}
	Comparing the last line here with \eqref{eq:identity_inner_1} proves the second identity.\qedhere
\end{proof}

The following two claims verify that $h$ satisfies the two conditions of being \fcomposable{}.

\begin{claim}
	\label{lem:three_expressions}
    Suppose that $a\in\R^{n_a -1}$ is \scomposable{} and $b\in \R^{n_b -1}$ is \fcomposable{} with rates $\alpha$ and $\beta$ respectively. Let $n = n_a+n_b$ and let $h = a \triangleright b \in \R^{n-1}$ and $\eta = \alpha \triangleright \beta$, then it holds that
		$\eta = \frac{1}{(1+2\sum_{i=0}^{n-2}h_i)} = \prod_{i=0}^{n-2}(h_i-1)^2$.
\end{claim}
\begin{proof}
	We can rewrite the reciprocal of the second expression as
	\begin{align*}
		1+2\sum_{i=0}^{n-2} h_i &= 1+2\left(\sum_{i=0}^{n_a-2}a_i + \mu + \sum_{i=0}^{n_b - 2} b_i\right) \\
		&= 2\left(1 + \sum_{i=0}^{n_a-2}a_i\right) + 2(\mu - 1) + \left(1 + 2\sum_{i=0}^{n_b - 2} b_i\right)\\
		&= \frac{2}{\alpha} + \frac{\sqrt{\alpha^2 + 8 \alpha\beta}-\alpha}{2\alpha\beta} + \frac{1}{\beta}= \frac{\alpha + 4\beta+\sqrt{\alpha^2 + 8 \alpha\beta}}{2\alpha\beta}= \frac{1}{\eta}.
	\end{align*}
	Here, we have used the definition of \scomposable{} schedules to simplify $1 + \sum_ia_i$ and the definition of \fcomposable{} schedules to simplify $1 + 2\sum_ib_i$.
	
	Next, we will show that $\prod_{i=0}^{n-2}(h_i-1)^2 = \eta$.
	Again, we use the definition of \scomposable{} and \fcomposable{} schedules to simplify
	\begin{align*}
		\prod_{i=0}^{n-2}(h_i-1)^2 &=
		\alpha^2 (\mu - 1)^2\beta\\
		&= \alpha^2\left(\frac{\sqrt{\alpha^2 + 8 \alpha\beta}-\alpha}{4\alpha\beta}\right)^2\beta\\
		&=\frac{\alpha^2 + 4\alpha\beta -\alpha\sqrt{\alpha^2 + 8\alpha\beta}}{8\beta}\\
		&= \frac{2\alpha\beta}{\alpha+4\beta + \sqrt{\alpha^2 + 8\alpha\beta}}\\
        &= \eta.
	\end{align*}
    Here, the second to last line follows by completing the square.\qedhere
\end{proof}

As $b$ is \fcomposable{} with rate $\beta$, \cref{prop:composable_eq} guarantees the existence of a vector $w\in\R^{n_b}_{+}$ such that GD with stepsize $b$, for any $1$-smooth convex function, satisfies
\begin{equation*}
	\sum_{i=0}^{n_b-1} w_i(2(f_i - f_{n_b-1}) + \norm{g_i}^2 + 2\ip{g_i, x_0 - x_i}) - \norm{\sum_{i=0}^{n_b-1}w_i g_i}^2 \ge 0
\end{equation*}
and $\sum_{i=0}^{n_b-1}w_i = \frac{1}{\beta}$.

\begin{claim}
	\label{lem:f_recurrence_v}
    Suppose that $a \in \R^{n_a-1}$ is \scomposable{} with rate $\alpha$ and $b \in \R^{n_b-1}$ is \fcomposable{} with rate $\beta$. Let $n = n_a + n_b$ and suppose that $h = a \triangleright b \in \R^{n-1}$ has \fcomposable{} rate $\eta$.
    Let $v\in\R^{n}_+$ be defined as
    \begin{equation*}
        v = \left[a,\hspace{0.25em} 1 + \frac{1}{\alpha} ,\hspace{0.25em} \sqrt{\frac{\beta}{\eta}} \cdot w\right].
    \end{equation*}
	Then, $v$ satisfies $\sum_{i=0}^{n-1}v_i = \frac{1}{\eta}$.
	for any $1$-smooth convex function $f$, GD with stepsize $h$ satisfies
	\begin{equation*}
		\sum_{i=0}^{n-1} v_i(2(f_i - f_{n-1}) + \norm{g_i}^2 + 2\ip{g_i, x_0 - x_i}) - \norm{\sum_{i=0}^{n-1}v_i g_i}^2 \ge 0.
	\end{equation*}
\end{claim}
\begin{proof}
	The first claim follows as
	\begin{align*}
		\sum_{i=0}^{n-1} v_i &= \sum_{i=0}^{n_a-2}a_i + 1 + \frac{1}{\alpha} + \sqrt{\frac{\beta}{\eta}}\sum_{i=0}^{n_b-1}w_i\\
		&= \frac{2}{\alpha} + \sqrt{\frac{\beta}{\eta}}\frac{1}{\beta}= \frac{1}{\eta}.
	\end{align*}
	Here, we have used the fact that $\frac{\beta}{\eta}- \frac{2\beta}{\alpha} = \sqrt{\frac{\beta}{\eta}}$ by \cref{lem:identities}.
	
    As the first $n_a-1$ stepsizes in $h$ coincide with $a$, which is \scomposable{} with rate $\alpha$, \cref{prop:self_comp_ineq} implies
	\begin{equation}
		\label{eq:f_composability_cert_1}
		\sum_{i=0}^{n_a - 2} a_{i} \left(2(f_i -f_{n_a-1}) + \norm{g_i}^2 + 2\ip{g_i, x_0 - x_i}\right)
			- \norm{x_{n_a-1} - x_0}^2 - \frac{1-\alpha}{\alpha^2}\norm{g_{n_a-1}}^2\geq 0.
	\end{equation}
	As the last $n_b-1$ steps of $h$ coincide with $b$, we have by \cref{prop:composable_eq} that
	\begin{equation}
		\label{eq:f_composability_cert_2}
		\frac{\beta}{\eta}\left(\sum_{i=0}^{n_b-1}w_i \left(2(f_{n_a +i} - f_{n-1}) + \norm{g_{n_a+i}}^2 + 2\ip{g_{n_a +i},x_{n_a} - x_{n_a +i}}\right) - \norm{\sum_{i=0}^{n_b-1}w_i g_{n_a+i}}^2\right)\geq 0.
	\end{equation}
	
	By \cref{fact}, it holds that
	\begin{align}
            &\frac{2\beta}{\alpha}\sum_{i=0}^{n_b-1} w_i Q_{n_a -1, n_a+i}\nonumber\\
            &= 
		\frac{2\beta}{\alpha}\sum_{i=0}^{n_b-1} w_i \left(
		2(f_{n_a-1} - f_{n_a +i}) - 2\ip{g_{n_a+i}, x_{n_a-1} - x_{n_a+i}} - \norm{g_{n_a-1} - g_{n_a+i}}^2
		\right) \geq 0.
		\label{eq:f_composability_cert_3}
	\end{align}

	Let $\Sigma$ denote the sum of the LHS expressions in \cref{eq:f_composability_cert_1,eq:f_composability_cert_2,eq:f_composability_cert_3} above. 
	The remainder of the proof verifies that
	\begin{equation*}
		\Sigma = \sum_{i=0}^{n - 1} v_i(2(f_i - f_{n-1}) + \norm{g_i}^2 + 2\ip{g_i, x_0 - x_i}) - \norm{\sum_{i=0}^{n - 1}v_i g_i}^2.
	\end{equation*}
	This will conclude the proof.
	
	Let $\Sigma_f$ denote the terms in $\Sigma$ that are linear in $f$ and let $\Sigma_g$ denote the terms in $\Sigma$ that are quadratic in $x_0 - x_\star$ and $g_i$. It holds that $\Sigma = \Sigma_f + \Sigma_g$.
	
	We compute
	\begin{align*}
		\frac{\Sigma_f}{2} &= \sum_{i=0}^{n_a - 2} a_{i} \left( f_i -f_{n_a - 1}\right)
		+ \frac{\beta}{\eta}\sum_{i=0}^{n_b - 1}w_i (f_{n_a+i} - f_{n-1}) + 2\frac{\beta}{\alpha}\sum_{i=0}^{n_b - 1} w_i \left(
		f_{n_a-1} - f_{n_a +i}
		\right)\\
		&= \sum_{i=0}^{n_a-2}v_i f_i + \left(1 + \frac{1}{\alpha}\right)f_{n_a-1} + \left(\frac{\beta}{\eta}- \frac{2\beta}{\alpha}\right)\sum_{i=0}^{n_b - 1}w_i f_{n_a+i} - \frac{1}{\eta}f_{n-1}\\
		&= \sum_{i=0}^{n-1}v_i (f_i -f_{n-1}).
	\end{align*}
	Here, we have used that $\frac{\beta}{\eta}- \frac{2\beta}{\alpha} = \sqrt{\frac{\beta}{\eta}}$ by \cref{lem:identities}.
	
	We next turn to $\Sigma_g$.
	We use the shorthand
	$\Delta_0 \coloneqq x_0 - x_{n_a-1} = \sum_{i=0}^{n_a-2}v_i g_i$,
	$\Delta_1 \coloneqq g_{n_a-1}$
	and
	$\Delta_2 \coloneqq \sum_{i=0}^{n_b - 1}w_i g_{n_a+i}$. Note that
	\begin{equation*}
		\sum_{i=0}^n v_i g_i = \Delta_0 + v_{n_a-1}g_{n_a-1} + \sqrt{\frac{\beta}{\eta}}\Delta_2.
	\end{equation*}
	The terms that are quadratic in $x_0-x_\star$ and $g_i$ in \eqref{eq:f_composability_cert_1} simplify as:
	\begin{align*}
		&\sum_{i=0}^{n_a - 2} a_i \left( \norm{g_i}^2 + 2\ip{g_i, x_0 - x_i}\right)
		- \norm{x_{n_a - 1} - x_0}^2 - \frac{1-\alpha}{\alpha^2}\norm{g_{n_a - 1}}^2\\
		&\qquad = \sum_{i=0}^{n_a - 1}v_i \left( \norm{g_i}^2 + 2\ip{g_i, x_0 - x_i}\right)
		-v_{n_a-1} \left( \norm{\Delta_1}^2 + 2\ip{\Delta_0, \Delta_1}\right)
		- \norm{\Delta_0}^2 - \frac{1-\alpha}{\alpha^2}\norm{\Delta_1}^2\\
		&\qquad =\sum_{i=0}^{n_a - 1} v_i\left(\norm{g_i}^2 + 2\ip{g_i, x_0 - x_i}\right) + \frac{2}{\alpha}\norm{\Delta_1}^2 - \norm{\Delta_0 + \frac{1+\alpha}{\alpha}\Delta_1}^2.
	\end{align*}
	The terms that are quadratic in $x_0-x_\star$ and $g_i$ in \eqref{eq:f_composability_cert_2} simplify as:
	\begin{align*}
		&\frac{\beta}{\eta}\sum_{i=0}^{n_b - 1}w_i \left(\norm{g_{n_a+i}}^2 + 2\ip{g_{n_a+i},x_{n_a} - x_{n_a+i}}\right) - \frac{\beta}{\eta} \norm{\Delta_2}^2\\
  		&\qquad = \frac{\beta}{\eta}\sum_{i=0}^{n_b - 1}w_i \left(\norm{g_{n_a+i}}^2 + 2\ip{g_{n_a+i},x_0 - x_{n_a+i}}\right)
    +\frac{2\beta}{\eta}\ip{\sum_{i=0}^{n_b - 1}w_ig_{n_a+i},x_{n_a} - x_0}
    - \frac{\beta}{\eta} \norm{\Delta_2}^2\\
		&\qquad = \frac{\beta}{\eta}\sum_{i=0}^{n_b - 1}w_i \left(\norm{g_{n_a+i}}^2 + 2\ip{g_{n_a+i},x_0 - x_{n_a+i}}\right)
		-\frac{2\beta}{\eta} \ip{\Delta_2,\Delta_0  + \mu \Delta_1}
		- \frac{\beta}{\eta} \norm{\Delta_2}^2\\
		&\qquad = \frac{\beta}{\eta}\sum_{i=0}^{n_b - 1}w_i \left(\norm{g_{n_a+i}}^2 + 2\ip{g_{n_a+i},x_0 - x_{n_a+i}}\right)
		-\frac{2\beta}{\eta} \ip{\Delta_0,\Delta_2}
		-\frac{2\mu \beta}{\eta} \ip{\Delta_1,\Delta_2}
		- \frac{\beta}{\eta} \norm{\Delta_2}^2.
	\end{align*}
	The terms that are quadratic in $x_0-x_\star$ and $g_i$ in \eqref{eq:f_composability_cert_3} simplify as:
	\begin{align*}
		&\frac{2\beta}{\alpha}\sum_{i=0}^{n_b - 1} w_i \left(
		- 2\ip{g_{n_a+i}, x_{n_a-1} - x_{n_a+i}} - \norm{g_{n_a-1} - g_{n_a+i}}^2
		\right)\\
  		&= -\frac{2\beta}{\alpha}\sum_{i=0}^{n_b - 1} w_i\bigg(
		2\ip{g_{n_a+i}, x_0 - x_{n_a+i}} + \norm{g_{n_a+i}}^2 + 
		2\ip{g_{n_a+i}, x_{n_a-1} - x_{0}} \\
        &\qquad - 2\ip{g_{n_a+i},g_{n_a-1}} +\norm{g_{n_a-1}}^2
		\bigg)\\
		& = -\frac{2\beta}{\alpha}\sum_{i=0}^{n_b - 1} w_i \left(
		2\ip{g_{n_a+i}, x_0 - x_{n_a+i}} + \norm{g_{n_a+i}}^2
		\right) -\frac{2}{\alpha}\norm{\Delta_1}^2  + \frac{4\beta}{\alpha}\ip{\Delta_0, \Delta_2}
		+\frac{4\beta}{\alpha}\ip{\Delta_1,\Delta_2}.
	\end{align*}

	We deduce that
	\begin{align*}
		\Sigma_g &= \sum_{i=0}^{n_a - 1} v_i\left(\norm{g_i}^2 + 2\ip{g_i, x_0 - x_i}\right) 
		+ \left(\frac{\beta}{\eta}-\frac{2\beta}{\alpha}\right)\sum_{i=0}^{n_b - 1}w_i \left(\norm{g_{n_a+i}}^2 + 2\ip{g_{n_a+i},x_0 - x_{n_a+i}}\right)\\
		&\qquad 
		- \norm{\Delta_0 + \frac{1+\alpha}{\alpha}\Delta_1}^2 
		-2\left(\frac{\beta}{\eta} - \frac{2\beta}{\alpha}\right)\ip{\Delta_0,\Delta_2}
		-2\left(\frac{\mu \beta}{\eta} - \frac{2\beta}{\alpha}\right)\ip{\Delta_1,\Delta_2}
		- \frac{\beta}{\eta} \norm{\Delta_2}^2.
	\end{align*}
	Using the fact that $\frac{\beta}{\eta}- \frac{2\beta}{\alpha} = \sqrt{\frac{\beta}{\eta}}$ by \cref{lem:identities}, we deduce that the first line on the right-hand side is equal to $\sum_{i=0}^{n-1}v_i\left(\norm{g_i}^2 + 2\ip{g_i, x_0 - x_i}\right)$.
	The remaining entries in the right-hand side are a quadratic form in $\Delta_0,\Delta_1,\Delta_2$ corresponding to the matrix
	\begin{equation*}
		\begin{pmatrix}
			-1 &-\frac{1+\alpha}{\alpha} & - \sqrt{\frac{\beta}{\eta}}\\
			\cdot & - \left(\frac{1+\alpha}{\alpha}\right)^2 & -\frac{\mu\beta}{\eta} + \frac{2\beta}{\alpha}\\
			\cdot & \cdot & -\frac{\beta}{\eta}
		\end{pmatrix} = \begin{pmatrix}
			-1 &-\frac{1+\alpha}{\alpha} & - \sqrt{\frac{\beta}{\eta}}\\
			\cdot & - \left(\frac{1+\alpha}{\alpha^2}\right)^2 & -\sqrt{\frac{\beta}{\eta}}\frac{1+\alpha}{\alpha}\\
			\cdot & \cdot & -\frac{\beta}{\eta}
		\end{pmatrix},
	\end{equation*} 
	where the bottom-left entries are defined by symmetry.
	Here, we have used the identities in \cref{lem:identities} to simplify the entries of the third column.
	We conclude that
	\begin{align*}
		\Sigma_g &= \sum_{i=0}^{n-1}v_i\left(\norm{g_i}^2 + 2\ip{g_i, x_0 - x_i}\right) - \norm{\Delta_0 - \frac{1+\alpha}{\alpha}\Delta_1 + \sqrt{\frac{\beta}{\eta}}\Delta_2}^2\\
		&= \sum_{i=0}^{n-1}v_i\left(\norm{g_i}^2 + 2\ip{g_i, x_0 - x_i}\right) - \norm{\sum_{i=0}^{n-1} v_i g_i}^2.
		\qedhere
	\end{align*}
\end{proof}

This completes the proof of \cref{thm:f_recurrence} as \cref{lem:three_expressions,lem:f_recurrence_v} verify the equivalent conditions for being \fcomposable{} stated in \cref{prop:composable_eq}.

\subsection{Proof of \cref{thm:g_recurrence}}
This subsection contains a proof of \cref{thm:g_recurrence}. Fix the notation of \cref{thm:g_recurrence} and for convenience let $h = b \triangleleft a$, $\eta = \beta\triangleleft \alpha$, and $n = n_b +n_a$. Our goal is to show that $h$ is \gcomposable{} with rate $\eta$ using the conditions in \cref{prop:g_composable_eq}.

Note that the values of $\mu$ in \cref{thm:f_recurrence} and \cref{thm:g_recurrence} are the same expressions in $\alpha,\beta$. Similarly, note that $\eta = \beta\triangleleft \alpha = \alpha\triangleright\beta$. Thus, the identities that we proved in \cref{lem:identities} relating $\alpha,\beta,\mu,\eta$ continue to hold in this setting.
Furthermore, as $\frac{1}{1 + 2\sum_{i=0}^{n-2}h_i}$ and $\prod_{i=0}^{n-2}(h_i-1)^2$ are independent of the ordering of the stepsizes in $h$, we have immediately by \cref{lem:three_expressions} that
\begin{equation*}
	\eta = \frac{1}{1 + 2\sum_{i=0}^{n-2}h_i} = \prod_{i=0}^{n-2}(h_i-1)^2.
\end{equation*}

It remains to check that the inequality in \cref{prop:g_composable_eq} holds for $h$ and $\eta$.
Similar to the proof of \cref{lem:f_recurrence_v}, the proof of \cref{lem:g_certificate_composition} will weight the guarantees of \gcomposable{} and \fcomposable{} schedules given in \cref{prop:g_composable_eq,prop:self_comp_ineq} and a few additional $Q_{i,j}$ terms with careful combination weights. Its proof is deferred to Appendix \ref{ap:deferred}.
\begin{lemma}
\label{lem:g_certificate_composition}
	For any $1$-smooth convex function $f$, GD with stepsize $h$ satisfies
	\begin{equation*}
		2f_0 - 2f_{n-1}  - \frac{1-\eta}{\eta}\norm{g_{n-1}}^2 \geq 0.
	\end{equation*}
\end{lemma}

\subsection{Proof of \cref{thm:selfcomp_recurrence}}
Fix the notation of \cref{thm:selfcomp_recurrence} and for convenience let $h = a\Join b$, $\eta = \alpha\Join\beta$, and $n = n_a +n_b$. Our goal is to show that $h$ is \scomposable{} with rate $\eta$ using the conditions in \cref{prop:self_comp_ineq}.

The following two lemmas state the necessary conditions to check. The proofs are similar to the proofs of \cref{lem:three_expressions,lem:f_recurrence_v} and are deferred to Appendix \ref{ap:deferred}.

\begin{lemma}
\label{lem:checking_scomp_identity}
Let $a \in \R^{n_a-1}$ and $b\in \R^{n_a-1}$ be \scomposable{} with rates $\alpha$ and $\beta$ respectively. Let $n = n_a + n_b$, and let $h = a \Join b\in \R^{n-1}$. Let $\eta = \alpha \Join \beta$.
	It holds that
		$\eta = \frac{1}{1+\sum_{i=0}^{n-2}h_i} = \prod_{i=0}^{n-2}(h_i-1)$.
\end{lemma}
The previous lemma asserts that $\eta$ (defined as $\alpha \Join \beta\coloneqq \frac{2\alpha\beta}{\alpha + \beta + \sqrt{\alpha^2 + 6\alpha\beta + \beta^2}}$) satisfies a given algebraic relationship.
The next lemma checks that $\eta$ is in fact the \scomposable{} rate of $h$.

\begin{lemma}
\label{lem:checking_scomp_inequality}
Let $a \in \R^{n_a-1}$ and $b \in \R^{n_b-1}$ be \scomposable{} with rates $\alpha$ and $\beta$ respectively. Let $n = n_a+n_b$ and $h = a \Join b \in \R^{n-1}$.
	For any $1$-smooth convex function $f$, GD with stepsize $h$ satisfies
	\begin{equation*}
		\sum_{i=0}^{n-2}h_i \left(2(f_i - f_{n-1}) + \norm{g_i}^2 + 2\ip{g_i, x_0 - x_i}\right) - \norm{x_{n-1} - x_0}^2 - \frac{1-\eta}{\eta^2} \norm{g_{n-1}}^2 \geq 0.
	\end{equation*}
\end{lemma}

\bibliographystyle{plainnat}

\begin{thebibliography}{21}
\providecommand{\natexlab}[1]{#1}
\providecommand{\url}[1]{\texttt{#1}}
\expandafter\ifx\csname urlstyle\endcsname\relax
  \providecommand{\doi}[1]{doi: #1}\else
  \providecommand{\doi}{doi: \begingroup \urlstyle{rm}\Url}\fi

\bibitem[Altschuler and Parrilo(2024)]{altschuler2023accelerationPartII}
Jason~M Altschuler and Pablo~A Parrilo.
\newblock Acceleration by stepsize hedging: Silver stepsize schedule for smooth
  convex optimization.
\newblock \emph{Mathematical Programming}, pages 1--14, 2024.

\bibitem[Altschuler and Parrilo(2025)]{altschuler2023accelerationPartI}
Jason~M Altschuler and Pablo~A Parrilo.
\newblock Acceleration by stepsize hedging: Multi-step descent and the silver
  stepsize schedule.
\newblock \emph{Journal of the ACM}, 72\penalty0 (2):\penalty0 1--38, 2025.

\bibitem[Drori and Teboulle(2012)]{drori2012PerformanceOF}
Yoel Drori and Marc Teboulle.
\newblock Performance of first-order methods for smooth convex minimization: a
  novel approach.
\newblock \emph{Mathematical Programming}, 145:\penalty0 451--482, 2012.

\bibitem[Grimmer et~al.(2023)Grimmer, Shu, and Wang]{grimmer2023accelerated}
Benjamin Grimmer, Kevin Shu, and Alex~L. Wang.
\newblock Accelerated gradient descent via long steps, 2023.
\newblock URL \url{https://arxiv.org/abs/2309.09961}.

\bibitem[Grimmer et~al.(2024)Grimmer, Shu, and
  Wang]{grimmer2024strengthenedconjectureminimaxoptimal}
Benjamin Grimmer, Kevin Shu, and Alex~L. Wang.
\newblock A strengthened conjecture on the minimax optimal constant stepsize
  for gradient descent, 2024.
\newblock URL \url{https://arxiv.org/abs/2407.11739}.

\bibitem[Grimmer et~al.(2025)Grimmer, Shu, and Wang]{grimmer2024accelerated}
Benjamin Grimmer, Kevin Shu, and Alex~L Wang.
\newblock Accelerated objective gap and gradient norm convergence for gradient
  descent via long steps.
\newblock \emph{INFORMS Journal on Optimization}, 7\penalty0 (2):\penalty0
  156--169, 2025.

\bibitem[Gupta et~al.(2023)Gupta, Parys, and Ryu]{gupta2023branch}
Shuvomoy~Das Gupta, Bart P.G.~Van Parys, and Ernest Ryu.
\newblock Branch-and-bound performance estimation programming: A unified
  methodology for constructing optimal optimization methods.
\newblock \emph{Mathematical Programming}, 2023.

\bibitem[Kim and Fessler(2016)]{Kim2016optimal}
Donghwan Kim and Jeffrey~A. Fessler.
\newblock Optimized first-order methods for smooth convex minimization.
\newblock \emph{Mathematical Programming}, 159\penalty0 (1–2):\penalty0
  81–107, sep 2016.
\newblock \doi{10.1007/s10107-015-0949-3}.
\newblock URL \url{https://doi.org/10.1007/s10107-015-0949-3}.

\bibitem[Kim and Fessler(2021)]{Kim2021gradient}
Donghwan Kim and Jeffrey~A. Fessler.
\newblock Optimizing the efficiency of first-order methods for decreasing the
  gradient of smooth convex functions.
\newblock \emph{J. Optim. Theory Appl.}, 188\penalty0 (1):\penalty0 192–219,
  jan 2021.
\newblock \doi{10.1007/s10957-020-01770-2}.
\newblock URL \url{https://doi.org/10.1007/s10957-020-01770-2}.

\bibitem[Kim et~al.(2023)Kim, Ozdaglar, Park, and Ryu]{Kim2024-Hduality}
Jaeyeon Kim, Asuman Ozdaglar, Chanwoo Park, and Ernest Ryu.
\newblock Time-reversed dissipation induces duality between minimizing gradient
  norm and function value.
\newblock \emph{Advances in Neural Information Processing Systems},
  36:\penalty0 23389--23440, 2023.

\bibitem[Kim(2024)]{kim2024proofexactconvergencerate}
Jungbin Kim.
\newblock A proof of exact convergence rate of gradient descent. part ii.
  performance criterion $(f(x_n)-f_*)/\|x_0-x_*\|^2$, 2024.
\newblock URL \url{https://arxiv.org/abs/2412.04427}.

\bibitem[Luner and
  Grimmer(2024)]{luner2024averagingextrapolationgradientdescent}
Alan Luner and Benjamin Grimmer.
\newblock On averaging and extrapolation for gradient descent, 2024.
\newblock URL \url{https://arxiv.org/abs/2402.12493}.

\bibitem[Malitsky and Mishchenko(2020)]{malitsky2020adaptive}
Yura Malitsky and Konstantin Mishchenko.
\newblock Adaptive gradient descent without descent.
\newblock In \emph{Proceedings of the 37th International Conference on Machine
  Learning (ICML)(2020)}, volume 119, 2020.

\bibitem[Nesterov(1983)]{Nesterov1983}
Yurii Nesterov.
\newblock A method for solving the convex programming problem with convergence
  rate $o(1/k^2)$.
\newblock \emph{Proceedings of the USSR Academy of Sciences}, 269:\penalty0
  543--547, 1983.
\newblock URL \url{https://api.semanticscholar.org/CorpusID:145918791}.

\bibitem[Nesterov(2014)]{nesterov-textbook}
Yurii Nesterov.
\newblock \emph{Introductory Lectures on Convex Optimization: A Basic Course}.
\newblock Springer Publishing Company, Incorporated, 1 edition, 2014.

\bibitem[Rotaru et~al.(2024)Rotaru, Glineur, and Patrinos]{rotaru2024exact}
Teodor Rotaru, Fran{\c{c}}ois Glineur, and Panagiotis Patrinos.
\newblock Exact worst-case convergence rates of gradient descent: a complete
  analysis for all constant stepsizes over nonconvex and convex functions,
  2024.
\newblock URL \url{https://arxiv.org/abs/2406.17506}.

\bibitem[Taylor et~al.(2017{\natexlab{a}})Taylor, Hendrickx, and
  Glineur]{taylor2017interpolation}
Adrien Taylor, Julien Hendrickx, and Fran\c{c}ois Glineur.
\newblock Smooth strongly convex interpolation and exact worst-case performance
  of first-order methods.
\newblock \emph{Mathematical Programming}, 161:\penalty0 307–345,
  2017{\natexlab{a}}.

\bibitem[Taylor et~al.(2017{\natexlab{b}})Taylor, Hendrickx, and
  Glineur]{taylor2017smooth}
Adrien Taylor, Julien Hendrickx, and Fran\c{c}ois Glineur.
\newblock Exact worst-case performance of first-order methods for composite
  convex optimization.
\newblock \emph{SIAM Journal on Optimization}, 27\penalty0 (3):\penalty0
  1283--1313, 2017{\natexlab{b}}.
\newblock \doi{10.1137/16M108104X}.
\newblock URL \url{https://doi.org/10.1137/16M108104X}.

\bibitem[Teboulle and Vaisbourd(2022)]{Teboulle2022}
Marc Teboulle and Yakov Vaisbourd.
\newblock An elementary approach to tight worst case complexity analysis of
  gradient based methods.
\newblock \emph{Mathematical Programming}, 201\penalty0 (1–2):\penalty0
  63–96, oct 2022.
\newblock \doi{10.1007/s10107-022-01899-0}.
\newblock URL \url{https://doi.org/10.1007/s10107-022-01899-0}.

\bibitem[Wang et~al.(2024)Wang, Ma, Yang, and
  Zhou]{wang2024relaxedproximalpointalgorithm}
Bofan Wang, Shiqian Ma, Junfeng Yang, and Danqing Zhou.
\newblock Relaxed proximal point algorithm: Tight complexity bounds and
  acceleration without momentum, 2024.
\newblock URL \url{https://arxiv.org/abs/2410.08890}.

\bibitem[Zhang and
  Jiang(2024)]{zhang2024acceleratedgradientdescentconcatenation}
Zehao Zhang and Rujun Jiang.
\newblock Accelerated gradient descent by concatenation of stepsize schedules,
  2024.
\newblock URL \url{https://arxiv.org/abs/2410.12395}.

\end{thebibliography}

\appendix

\section{Deferred proofs}
\label{ap:deferred}

\begin{proof}[Proof of \cref{lem:g_huber_tight}]
    	First, suppose $f(x)= \frac{1}{2}x^2$ and $x_0=1$. Then,
    	\begin{equation*}
    		\frac{1}{2}\norm{\nabla f(x_n)}^2 = \frac{1}{2}x_n^2 = \left(\prod_{i=0}^{n-1}(h_i-1)^2\right) \frac{1}{2}x_0^2= \eta (f(x_0)-f(x_\star)).
    	\end{equation*}
        
    	Next, suppose $f(x)$ is the Huber function described in the lemma statement. Note that
    	\begin{equation*}
    		1 - \delta\sum_{i=0}^{n-1}h_i = 1- \left(\frac{\eta}{1+\eta}\right)\left(\frac{1-\eta}{\eta}\right) = \frac{2\eta}{1+\eta}=\delta.
    	\end{equation*}
    	We deduce that $x_n = \delta$ and that
    	\begin{equation*}
    		\frac{1}{2}\norm{\nabla f(x_n)}^2 = \frac{\delta^2}{2} = \frac{2\eta^2}{(1+\eta)^2} = \eta \left(f_0-f_\star\right).\qedhere
    	\end{equation*}
    \end{proof}

    \begin{proof}[Proof of \cref{lem:s_composable_huber_tight}]
        	First, suppose $f(x) = \frac{1}{2}x^2$ and $x_0 = 1$.
        	Note that $x_n^2 = \eta^2$.
        	We will show that the difference in the LHS and RHS of \eqref{eq:self_composable_ineq} is zero:
        	\begin{align*}
        		&\frac{\eta^2}{2}\norm{x_0 - x_\star}^2 - \frac{1-\eta}{2}\norm{\nabla f(x_n)}^2 - \frac{\eta^2}{2}\norm{x_n-x_\star}^2 - (\eta-\eta^2)(f(x_n)-f(x_\star))\\
        		&\qquad = \frac{\eta^2}{2} - \frac{1-\eta}{2}\eta^2 - \frac{\eta^2}{2}\eta^2 - \frac{\eta-\eta^2}{2}\eta^2 = 0.
        	\end{align*}
            
        	Next, suppose $f(x)$ is the Huber function described in the lemma statement. We check that
        	\begin{equation*}
        		1-\delta \sum_{i=0}^{n-1}h_i = 1 - \frac{\delta(1-\eta)}{\eta} \geq \delta,
        	\end{equation*}
        	where the last inequality holds for all $\delta \leq \eta$. We deduce that $x_n = 1 - \frac{\delta(1-\eta)}{\eta}$. Again, we compute the difference in the LHS and RHS in \eqref{eq:self_composable_ineq}:
        	\begin{align*}
        		&\frac{\eta^2}{2}\norm{x_0 - x_\star}^2 - \frac{1-\eta}{2}\norm{\nabla f(x_n)}^2 - \frac{\eta^2}{2}\norm{x_n-x_\star}^2 - (\eta-\eta^2)(f(x_n)-f(x_\star))\\
        		&\qquad = \frac{\eta^2}{2} - \frac{1-\eta}{2}\delta^2 - \frac{\eta^2}{2}\left(1 - \frac{\delta(1-\eta)}{\eta}\right)^2 - (\eta-\eta^2)\left(\delta - \frac{\delta^2(1-\eta)}{\eta} - \frac{\delta^2}{2}\right)= 0.\qedhere
        	\end{align*}
        \end{proof}

\begin{proof}[Proof of \cref{prop:g_composable_eq}]
    We handle the reverse direction first. Let $f$ be a $1$-smooth convex function with minimizer $x_\star$. Let $x_0 = 1$. Then,
	\begin{align*}
		0&\leq 2\eta(f_0-f_n) - (1-\eta)\norm{g_n}^2 + \eta Q_{n,\star}\\
		&= 2\eta(f_0 - f_\star) - \norm{g_n}^2.
	\end{align*}
	
	Now, suppose 
	$h$ is \gcomposable{} with rate $\eta$ and $n\geq 0$.
    The definition of a \gcomposable{} schedule implies that the expression
	\begin{equation*}
		2\eta(f_0 - f_\star) - \norm{g_n}^2
	\end{equation*}
    is nonnegative for all $1$ smooth convex $f$ and achieves the value $0$ for the Huber function $H_{2\eta/(1+\eta)}$ by \cref{lem:g_huber_tight}.
	By \cref{lem:pep_exists}, there exists $\lambda\in\R^{(n+2)\by (n+2)}$ and $S$ a PSD quadratic form so that
	\begin{equation}
		\label{eq:g_pep_cert}
		2\eta(f_0 - f_\star) - \norm{g_n}^2 = \sum_{i,j}\lambda_{i,j}Q_{i,j} + S.
	\end{equation}
	
	Now, consider \eqref{eq:g_pep_cert} for $f = H_{2\eta/(1+\eta)}$. By assumption both the LHS and RHS evaluate to $0$. By \cref{lem:positive_Qstari} it holds that $Q_{i,\star} > 0$ for all $i\in[0,n-1]$. Thus, we deduce that $\lambda_{i,\star} = 0$ for all $i\in[0,n-1]$.
	
	Comparing the coefficient on $f_\star$ in the LHS and RHS of \eqref{eq:g_pep_cert} gives
	\begin{align*}
		-2\eta &= 2\sum_{i=0}^n \lambda_{\star,i} - 2\sum_{i=0}^n \lambda_{i,\star}\\
		&= 2\sum_{i=0}^n \lambda_{\star,i} - 2\lambda_{n,\star}.
	\end{align*}
	We deduce that $\lambda_{n,\star}\geq \eta$.
	Thus,
	\begin{align*}
		0 &\leq \sum_{i,j}\lambda_{i,j}Q_{i,j} + S - \eta Q_{n,\star}\\
		&= 2\eta(f_0 - f_n) - (1-\eta)\norm{g_n}^2.
	\end{align*}
	Rearranging this inequality completes the proof.\qedhere
\end{proof}

\begin{proof}[Proof of \cref{prop:self_comp_ineq}]
	We begin with the reverse direction. Let $f$ be a $1$-smooth function with minimizer $x_\star$. Then,
	\begin{align*}
		0&\leq \sum_{i=0}^{n-1}h_i \left(2(f_i - f_n) + \norm{g_i}^2 + 2\ip{g_i, x_0 - x_i}\right) - \norm{x_n - x_0}^2 - \frac{1-\eta}{\eta^2} \norm{g_n}^2\\
		&\qquad + \sum_{i=0}^{n-1}h_i Q_{\star,i}\\
		&= \sum_{i=0}^{n-1}h_i \left(2(f_\star - f_n) + 2\ip{g_i, x_0 - x_\star}\right) - \norm{x_n - x_0}^2 - \frac{1-\eta}{\eta^2} \norm{g_n}^2\\
		&= \frac{2(1-\eta)}{\eta}(f_\star - f_n) + 2\ip{x_0 - x_n, x_0 - x_\star} - \norm{x_n - x_0}^2 - \frac{1-\eta}{\eta^2}\norm{g_n}^2\\
		&= \frac{2(1-\eta)}{\eta}(f_\star - f_n) + \norm{x_0 - x_\star}^2 -\norm{x_n-x_\star}^2 - \frac{1-\eta}{\eta^2}\norm{g_n}^2.
	\end{align*}
	Rearranging this inequality shows that $h$ is \scomposable{} with rate $\eta$.

	Next, suppose $h$ is \scomposable{} with rate $\eta$. By definition, the expression
	\begin{equation*}
		\frac{2(1-\eta)}{\eta}(f_\star - f_n) + \norm{x_0 - x_\star}^2 -\norm{x_n-x_\star}^2 - \frac{1-\eta}{\eta^2}\norm{g_n}^2
	\end{equation*}
	is nonnegative for any $1$-smooth convex $f$ and achieves the value $0$ for any Huber function $f_\delta$ with $\delta \leq \eta$.
	By \cref{lem:pep_exists}, there exists $\lambda\in\R^{(n+2)\by(n+2)}$ and $S$ a PSD quadratic form so that
	\begin{equation}
		\label{eq:s-composability_pep}
		\frac{2(1-\eta)}{\eta}(f_\star - f_n) + \norm{x_0 - x_\star}^2 -\norm{x_n-x_\star}^2 - \frac{1-\eta}{\eta^2}\norm{g_n}^2 = \sum_{i,j}\lambda_{i,j}Q_{i,j} + S.
	\end{equation}
	
	We will overload notation and identify the quadratic form $S$ with a PSD matrix $S\in \S^{n+2}_+$ indexed by $\set{\star,0,1,\dots,n}$ so that
	\begin{equation*}
		S = \tr\left(\begin{pmatrix}
			x_0 - x_\star & g_0 & \dots & g_n
		\end{pmatrix} S \begin{pmatrix}
			x_0 - x_\star & g_0 & \dots & g_n
		\end{pmatrix}^\intercal\right)
	\end{equation*}
	Note that the coefficient on $\norm{x_0 - x_\star}^2$ on the LHS of \eqref{eq:s-composability_pep} is $0$. We deduce that $S_{\star,\star} = 0$.
	In turn, as $S$ is PSD, we deduce that $S_{\star,i} = 0$ for all $i\in[0,n]$.
	Now, consider the coefficient on $\ip{g_i, x_0 - x_\star}$ in \eqref{eq:s-composability_pep}. By equating the coefficient in the LHS with the coefficient in the RHS, we deduce that
	\begin{equation*}
		2h_i = 2\lambda_{\star,i} + 2S_{\star,i} = 2\lambda_{\star,i}
	\end{equation*}
    for all $i=0,1,\dots,n-1$ and $\lambda_{\star,n}=0$.
	
	Finally, we compute
	\begin{align*}
		0&\leq \sum_{\substack{i\in[0,n]\\
				j\in[\star,n]}}\lambda_{i,j}Q_{i,j}\\
		&= \frac{2(1-\eta)}{\eta}(f_\star - f_n) + \norm{x_0 - x_\star}^2 -\norm{x_n-x_\star}^2 - \frac{1-\eta}{\eta^2}\norm{g_n}^2 - S - \sum_{i=0}^{n-1}h_iQ_{\star,i}\\
		&\leq \frac{2(1-\eta)}{\eta}(f_\star - f_n) + \norm{x_0 - x_\star}^2 -\norm{x_n-x_\star}^2 - \frac{1-\eta}{\eta^2}\norm{g_n}^2\\
		&\qquad + \sum_{i=0}^{n-1}h_i\left(2(f_i - f_\star) + \ip{g_i, x_\star - x_i} + \norm{g_i}^2\right)\\
		&= \sum_{i=0}^{n-1}h_i \left(2(f_i - f_n) + \norm{g_i}^2 + 2\ip{g_i, x_0 - x_i}\right) - \norm{x_n - x_0}^2 - \frac{1-\eta}{\eta^2} \norm{g_n}^2.\qedhere
	\end{align*}
\end{proof}

\begin{proof}
    [Proof of \cref{lem:g_certificate_composition}]

	The first $n_b - 1$ steps of $h$ coincide with $b$. As $b$ is \gcomposable{} with rate $\beta$, \cref{prop:g_composable_eq} implies that the following expression is nonnegative
	\begin{equation}
		\label{eq:g_cert_part1}
		\eta\left(2f_0  - 2f_{n_b - 1} - \frac{1-\beta}{\beta}\norm{g_{n_b - 1}}^2\right).
	\end{equation}
	
	The last $n_a - 1$ steps of $h$ coincide with $a$. As $a$ is \scomposable{} with rate $\alpha$, \cref{prop:self_comp_ineq} implies that the following expression is nonnegative
	\begin{align}
		\label{eq:g_cert_part2}
		&\alpha^2\bigg[\sum_{i=0}^{n_a - 2} a_i \left(2(f_{n_b + i} -f_{n-1}) + \norm{g_{n_b + i}}^2 + 2\ip{g_{n_b + i}, x_{n_b} - x_{n_b + i}}\right)\nonumber
		\\
        &\qquad\qquad - \norm{x_{n-1} - x_{n_b}}^2 - \frac{1-\alpha}{\alpha^2}\norm{g_{n-1}}^2\bigg].
	\end{align}

	By \cref{fact}, the following expressions are nonnegative:
	\begin{gather}
		\frac{\alpha^2}{2}\sum_{i=0}^{n_a-2} a_i\left(Q_{n_b - 1,n_b + i} + Q_{n-1,n_b + i}\right) \label{eq:g_cert_part3},\qquad \text{and}\\
		\frac{\alpha^2}{2} Q_{n_b - 1, n-1} + 
		\left(\frac{\alpha}{2}-\eta\right) Q_{n-1,n_b - 1}.     \label{eq:g_cert_part4}
	\end{gather}
	Here, we have used the fact that
	\begin{equation*}
		\eta = \frac{2\alpha\beta}{\alpha+4\beta+\sqrt{\alpha^2+8\alpha\beta}} \leq \frac{2\alpha\beta}{4\beta} = \frac{\alpha}{2}.
	\end{equation*}

	Let $\Sigma$ denote the sum of \cref{eq:g_cert_part1,eq:g_cert_part2,eq:g_cert_part3,eq:g_cert_part4}. Let $\Sigma_f$ denote the terms depending on $f$ and let $\Sigma_g$ denote the terms depending on $x_0-x_\star$ and $g_i$. Again, $\Sigma = \Sigma_f + \Sigma_g$.
	
	We compute
	\begin{align*}
		\frac{\Sigma_f}{2} &= \eta(f_0 - f_{n_b -1}) + \alpha^2\sum_{i=0}^{n_a - 2} a_i (f_{n_b +i} - f_{n-1})\\
		&\qquad + \frac{\alpha^2}{2}\sum_{i=0}^{n_a -2} a_i\left(f_{n_b-1} + f_{n-1} - 2f_{n_b+i}\right) + \left(\frac{\alpha-\alpha^2}{2}-\eta\right)(f_{n - 1} - f_{n_b-1})\\
		&= \eta f_0 - \eta f_{n-1}.
	\end{align*}
	
	Let $\Delta_0 = g_{n_b-1}$, $\Delta_1 = x_{n-1} - x_{n_b}$, and $\Delta_2 = g_{n-1}$,

	The terms in \eqref{eq:g_cert_part1} that depend on $x_0 - x_\star$ and $g_i$ simplify to
	\begin{equation*}
		-\eta\frac{1-\beta}{\beta}\norm{\Delta_0}^2.
	\end{equation*}
	
	The terms in \eqref{eq:g_cert_part2} that depend on $x_0 - x_\star$ and $g_i$ simplify to
	\begin{align*}
		&\alpha^2\sum_{i=0}^{n_a - 2} a_i \left(\norm{g_{n_b + i}}^2 + 2\ip{g_{n_b + i}, x_{n_b} - x_{n_b + i}}\right)
		- \alpha^2\norm{x_{n-1} - x_{n_b}}^2 - (1-\alpha)\norm{g_{n-1}}^2\\
		&\qquad = \alpha^2\sum_{i=0}^{n_a - 2} a_i \left(\norm{g_{n_b + i}}^2 + 2\ip{g_{n_b + i}, x_{n_b} - x_{n_b + i}}\right)
		- \alpha^2\norm{\Delta_1}^2 - (1-\alpha)\norm{\Delta_2}^2.
	\end{align*}
	
	The terms in \eqref{eq:g_cert_part3} that depend on $x_0 - x_\star$ and $g_i$ simplify to
	\begin{align*}
		&\frac{\alpha^2}{2}\sum_{i=0}^{n_a-2} a_i\bigg(
		-2\ip{g_{n_b +i}, x_{n_b-1} - x_{n_b +i}} - \norm{g_{n_b+i}-g_{n_b-1}}^2
		-2\ip{g_{n_b +i}, x_{n-1} - x_{n_b +i}}\\
        &\qquad - \norm{g_{n_b+i}-g_{n-1}}^2\bigg)\\
		& = \frac{\alpha^2}{2}\sum_{i=0}^{n_a-2} a_i\bigg(
		-4\ip{g_{n_b +i}, x_{n_b} - x_{n_b +i}}
		-2\ip{g_{n_b +i}, x_{n_b-1} - x_{n_b}}
		-2\ip{g_{n_b +i}, x_{n-1} - x_{n_b}}\\
		&\qquad
		- 2\norm{g_{n_b+i}}^2
		- \norm{g_{n_b-1}}^2
		- \norm{g_{n-1}}^2
		+ 2\ip{g_{n_b+i},g_{n_b-1}}
		+2\ip{g_{n_b+i},g_{n-1}}       
		\bigg)\\
		& = 
		-\alpha^2\sum_{i=0}^{n_a-2} a_i\left( \norm{g_{n_b+i}}^2 + 2\ip{g_{n_b +i}, x_{n_b} - x_{n_b +i}}\right)\\
		&\qquad
		- \frac{\alpha-\alpha^2}{2}\norm{\Delta_0}^2
		+\alpha^2\norm{\Delta_1}^2
		- \frac{\alpha-\alpha^2}{2}\norm{\Delta_2}^2
		+ \alpha^2
		\ip{\Delta_1, (\mu-1)\Delta_0 - \Delta_2}.
	\end{align*}
	Here, we have used the fact that $\sum_{i=0}^{n_a-2}a_i g_{n_b+i} = x_{n_b} - x_{n-1} = -\Delta_1$.
	
	The terms in \eqref{eq:g_cert_part4} that depend on $x_0 - x_\star$ and $g_i$ simplify to
	\begin{align*}
		&\frac{\alpha^2}{2} \left(-2\ip{g_{n-1}, x_{n_b - 1} - x_{n-1}} - \norm{g_{n_b-1} - g_{n-1}}^2\right)
		\\
        &\qquad\qquad +
		\left(\frac{\alpha}{2}-\eta\right) \left(-2\ip{g_{n_b-1}, x_{n - 1} - x_{n_b-1}} - \norm{g_{n_b-1} - g_{n-1}}^2\right)\\
		&= \ip{(\alpha -2\eta)\Delta_0-\alpha^2\Delta_2, \mu\Delta_0 -\Delta_1} - \left(\frac{\alpha+\alpha^2}{2}-\eta\right) \norm{\Delta_2-\Delta_0}^2.
	\end{align*}
	
	Summing up the above quantities gives
	\begin{align*}
		\Sigma_g &= -\eta\frac{1-\beta}{\beta}\norm{\Delta_0}^2 - \alpha^2\norm{\Delta_1}^2 - (1-\alpha)\norm{\Delta_2}^2\\
		&\qquad - \frac{\alpha(1-\alpha)}{2}\norm{\Delta_0}^2
		+\alpha^2\norm{\Delta_1}^2
		- \frac{\alpha(1-\alpha)}{2}\norm{\Delta_2}^2
		+ \alpha^2
		\ip{\Delta_1, (\mu-1)\Delta_0 - \Delta_2}\\
		&\qquad + \ip{(\alpha -2\eta)\Delta_0-\alpha^2\Delta_2, \mu\Delta_0 -\Delta_1} - \left(\frac{\alpha(1+\alpha)}{2}-\eta\right) \norm{\Delta_2-\Delta_0}^2
	\end{align*}
	We see that $\Sigma_g$ is a quadratic form in the $\Delta_0,\Delta_1,\Delta_2$ corresponding to the symmetric matrix
	\begin{equation*}
		\begin{pmatrix}
			-\eta\left(\frac{1}{\beta} + 2(\mu-1)\right) +\alpha(\mu-1)& \eta + \frac{\alpha(\alpha(\mu-1)-1)}{2} & -\eta + \frac{\alpha(1-\alpha(\mu-1))}{2}\\
			\cdot & 0 & 0\\
			\cdot & \cdot & -(1-\eta)
		\end{pmatrix},
	\end{equation*}
	where the entries below the diagonal are defined by symmetry.
	
	We claim that the entries in the first row are zero. Indeed, the second and third entries are zero by
	\begin{align*}
		\frac{\alpha(1-\alpha(\mu-1))}{2}&= 
		\frac{\alpha}{2}\left(\frac{\alpha + 4\beta - \sqrt{\alpha^2 + 8 \alpha\beta}}{4\beta}\right)\\
		&= \frac{8\alpha\beta^2}{(4\beta)\left(\alpha+4\beta+\sqrt{\alpha^2+8\alpha\beta}\right)}= \eta.
	\end{align*}
The first entry in the first row is zero by
	\begin{align*}
		&-\eta\left(\frac{1}{\beta} + 2(\mu-1)\right) +\alpha(\mu-1)\\
        &=  (\alpha-2\eta)(\mu-1) - \frac{\eta}{\beta}\\
		&=
        \alpha\left(\frac{\alpha + \sqrt{\alpha^2 + 8\alpha\beta}}{\alpha + 4\beta + \sqrt{\alpha^2 + 8\alpha\beta}}\right)\left(\frac{\sqrt{\alpha^2 + 8\alpha \beta}-\alpha}{4\alpha\beta}\right) - \left(\frac{2\alpha}{\alpha+4\beta+\sqrt{\alpha^2 + 8\alpha\beta}}\right)= 0.
	\end{align*}

	In summary, we have shown that
	\begin{equation*}
		\eta f_0 - \eta f_{n-1} - (1-\eta)\norm{g_{n-1}}^2 \geq 0
	\end{equation*}
	and that $h$ is \gcomposable{} with rate $\eta$.\qedhere
\end{proof}

\begin{proof}[Proof of \cref{lem:checking_scomp_identity}]
	First, we compute
	\begin{align*}
		1 + \sum_{i=0}^{n-2}h_i &= \left(1 + \sum_{i=0}^{n_a -2}a_i\right) + (\mu-1) + \left(1 + \sum_{i=0}^{n_b -2}b_i\right)\\
		&= \frac{1}{\alpha} + \left(\frac{\sqrt{\alpha^2+6\alpha\beta+\beta^2}-(\alpha+\beta)}{2\alpha\beta}\right) + \frac{1}{\beta}\\
		&= \frac{\alpha+\beta + \sqrt{\alpha^2+6\alpha\beta+\beta^2}}{2\alpha\beta}.
	\end{align*}
	We recognize the final line as $\frac{1}{\eta}$.
	
	Next, we compute
	\begin{align*}
		\prod_{i=0}^{n-2}(h_i - 1) &= \prod_{i=0}^{n_a-2}(a_i - 1) \cdot \prod_{i=0}^{n_b-2}(b_i - 1) \cdot (\mu-1)\\
		&= \frac{\sqrt{\alpha^2+6\alpha\beta+\beta^2}-(\alpha+\beta)}{2}\\
		&= \frac{4\alpha\beta}{2\alpha+2\beta + 2\sqrt{\alpha^2+6\alpha\beta+\beta^2}}.
	\end{align*}
	We recognize the final line as $\eta$.\qedhere
\end{proof}

\begin{proof}[Proof of \cref{lem:checking_scomp_inequality}]
	Let
	\begin{equation*}
		\sigma = \frac{\beta + \alpha\beta (\mu-1)}{\alpha - \alpha \beta(\mu-1)}.
	\end{equation*}
	This expression is nonnegative as
	\begin{align*}
		\beta(\mu-1) &= \beta \left(\frac{\sqrt{\alpha^2+6\alpha\beta+\beta^2}-(\alpha+\beta)}{2\alpha\beta}\right)\\
		&    
		< \frac{\sqrt{9\alpha^2+6\alpha\beta+\beta^2}-(\alpha+\beta)}{2\alpha}= 1.
	\end{align*}

	As $a$ is \scomposable{} with rate $\alpha$, \cref{prop:self_comp_ineq} implies that the following expression is nonnegative:
	\begin{equation}
		\label{eq:sc_cert_part1}
		\sum_{i=0}^{n_a - 2} a_i \left( 2(f_i -f_{n_a-1}) + \norm{g_i}^2 + 2\ip{g_i, x_0 - x_i}\right)- \norm{x_{n_a-1} - x_0}^2 - \frac{1-\alpha}{\alpha^2}\norm{g_{n_a-1}}^2.
	\end{equation}
	Similarly, the following expression is nonnegative:
	\begin{align}
		\label{eq:sc_cert_part2}
		&\sum_{i=0}^{n_b - 2} b_i \left( 2(f_{n_a+i} -f_{n-1}) + \norm{g_{n_a+i}}^2 + 2\ip{g_{n_a+i}, x_{n_a} - x_{n_a+i}}\right)\nonumber\\
        &\qquad\qquad\qquad- \norm{x_{n-1} - x_{n_a}}^2 - \frac{1-\beta}{\beta^2}\norm{g_{n-1}}^2.
	\end{align}
	
	By \cref{fact}, the following expressions are also nonnegative
	\begin{gather}
		\sum_{i=0}^{n_b-2}b_i\left(Q_{n_a-1,n_a+i} + Q_{n-1,n_a+i}\right)\label{eq:sc_cert_part3}\\
		Q_{n_a-1,n-1} + (\mu-1)Q_{n-1,n_a-1}.\label{eq:sc_cert_part4}
	\end{gather}
	
	Let $\Sigma$ denote the weighted sum of \cref{eq:sc_cert_part1,eq:sc_cert_part2,eq:sc_cert_part3,eq:sc_cert_part4}
    with weights $1$, $1+2\sigma$, $\sigma$, $\sigma$ respectively. Let $\Sigma_f$ denote the terms depending on $f$ and let $\Sigma_g$ denote the terms depending on $x_0-x_\star,g_0,\dots,g_n$. Again, $\Sigma = \Sigma_f + \Sigma_g$.
	
	We compute
	\begin{align*}
		\frac{\Sigma_f}{2}&= \sum_{i=0}^{n_a - 2} a_i(f_i -f_{n_a-1})
		+ (1+2\sigma)\sum_{i=0}^{n_b - 2} b_i (f_{n_a+i} -f_{n-1}) \\
		&\qquad +
		\sigma\sum_{i=0}^{n_b-2}b_i\left(f_{n_a-1} + f_{n-1} - 2f_{n_a+i}\right) + 
		(\mu-2)\sigma \left(f_{n-1} - f_{n_a-1}\right)\\
		&=  \sum_{i=0}^{n_a - 2} h_i f_i 
		+ \left(\frac{\sigma}{\beta}
		- \frac{1-\alpha}{\alpha} - (\mu-1)\sigma\right)f_{n_a-1}
		+ \sum_{i=n_a}^{n-2}h_if_i\\
        &\qquad\qquad  - \left(\frac{1+\sigma}{\beta}-1-(\mu-1)\sigma\right)f_{n-1}.
	\end{align*}
	A straightforward calculation shows that the coefficient on $f_{n_a-1}$ is
	\begin{equation*}
		\left(\frac{1}{\beta} - (\mu-1)\right)\sigma
		- \frac{1}{\alpha}+1 = \mu = h_{n_a-1}.
	\end{equation*}
	We deduce that
	\begin{equation*}
		\Sigma_f = 2\sum_{i=0}^{n-2}h_i(f_i - f_{n-1}).
	\end{equation*}
	
	Let $\Delta_0 = x_{n_a-1}-x_0$,
	$\Delta_1 = g_{n_a-1}$,
	$\Delta_2 = x_{n-1}-x_0$, and
	$\Delta_3 = g_{n-1}$.
	Note that $x_{n_a} - x_0 = \Delta_0 - \mu\Delta_1$ and $x_{n-1} - x_{n_a} = -\Delta_0 + \mu\Delta_1 + \Delta_2$.
	
	The terms in \eqref{eq:sc_cert_part1} that depend on $x_0 - x_\star$ and $g_i$ simplify to
	\begin{align*}
		&\sum_{i=0}^{n_a - 2} a_i \left(\norm{g_i}^2 + 2\ip{g_i, x_0 - x_i}\right)- \norm{x_{n_a-1} - x_0}^2 - \frac{1-\alpha}{\alpha^2}\norm{g_{n_a-1}}^2\\
		&\qquad = \sum_{i=0}^{n_a - 2} h_i \left(\norm{g_i}^2 + 2\ip{g_i, x_0 - x_i}\right)- \norm{\Delta_0}^2 - \frac{1-\alpha}{\alpha^2}\norm{\Delta_1}^2.
	\end{align*}
	
	The terms in \eqref{eq:sc_cert_part2} that depend on $x_0 - x_\star$ and $g_i$ simplify to
	\begin{align*}
		&\sum_{i=0}^{n_b - 2} b_i \left(\norm{g_{n_a+i}}^2 + 2\ip{g_{n_a+i}, x_{n_a} - x_{n_a+i}}\right)- \norm{x_{n-1} - x_{n_a}}^2 - \frac{1-\beta}{\beta^2}\norm{g_{n-1}}^2\\
		&\qquad = \sum_{i=n_a}^{n - 2} h_i \left(\norm{g_i}^2 + 2\ip{g_i, x_0 - x_i}\right)
		+2\ip{x_{n_a} - x_{n-1},x_{n_a} - x_0}\\
		&\qquad\qquad
		- \norm{x_{n-1} - x_{n_a}}^2 - \frac{1-\beta}{\beta^2}\norm{g_{n-1}}^2\\
		&\qquad = \sum_{i=n_a}^{n - 2} h_i \left(\norm{g_i}^2 + 2\ip{g_i, x_0 - x_i}\right)
		+\norm{\Delta_0 - \mu\Delta_1}^2
		- \norm{\Delta_2}^2 - \frac{1-\beta}{\beta^2}\norm{\Delta_3}^2.
	\end{align*}
	
	The terms in \eqref{eq:sc_cert_part3} that depend on $x_0 - x_\star$ and $g_i$ simplify to
	\begin{align*}
		&\sum_{i=0}^{n_b-2}b_i\left(
		-2\ip{g_{n_a+i}, x_{n_a-1} + x_{n-1} - 2x_{n_a+i}} - \norm{g_{n_a-1} - g_{n_a +i}}^2
		- \norm{g_{n-1} - g_{n_a +i}}^2
		\right)\\
		&=\sum_{i=n_a}^{n-2}h_i\bigg(
		-4\ip{g_{i}, x_0 - x_{i}}
		- 2\norm{g_i}^2
		-2 \ip{g_i, x_{n_a-1}-x_0}
		-2 \ip{g_i, x_{n-1} -x_0}\\
		&\qquad
		- \norm{g_{n_a-1}}^2 + 2\ip{g_{n_a-1},g_i}
		- \norm{g_{n-1}}^2 + 2\ip{g_{n-1},g_i}
		\bigg)\\
		& = 
		-2\sum_{i=n_a}^{n-2}h_i\left(\norm{g_i}^2+2\ip{g_{i}, x_0 - x_{i}}\right)
		-2 \ip{x_{n_a} -x_{n-1}, x_{n_a-1}-x_0}\\
		&\qquad -2 \ip{x_{n_a} -x_{n-1}, x_{n-1} -x_0}		- \frac{1-\beta}{\beta}\norm{g_{n_a-1}}^2 + 2\ip{g_{n_a-1},x_{n_a}-x_{n-1}}\\
        &\qquad 
		- \frac{1-\beta}{\beta}\norm{g_{n-1}}^2 + 2\ip{g_{n-1},x_{n_a}-x_{n-1}}\\
		& = 
		-2\sum_{i=n_a}^{n-2}h_i\left(\norm{g_i}^2+2\ip{g_{i}, x_0 - x_{i}}\right)
		-2\norm{\Delta_0}^2 
		-\left(2\mu + \frac{1-\beta}{\beta}\right)\norm{\Delta_1}^2 + 2\norm{\Delta_2}^2\\
		&\qquad  
		- \frac{1-\beta}{\beta^2}\norm{\Delta_3}^2
		+ 2(\mu+1)\ip{\Delta_0,\Delta_1} +2(\mu-1)\ip{\Delta_1,\Delta_2} + 2\ip{\Delta_3, \Delta_0 - \mu\Delta_1-\Delta_2}.
	\end{align*}
	
	The terms in \eqref{eq:sc_cert_part4} that depend on $x_0 - x_\star$ and $g_i$ simplify to
	\begin{align*}
		&-2\ip{g_{n-1}, x_{n_a-1}-x_{n-1}} - \norm{g_{n_a-1}-g_{n-1}}^2 \\
        &\qquad+ (\mu-1) \left(-2\ip{g_{n_a-1}, x_{n-1}-x_{n_a-1}} - \norm{g_{n-1}-g_{n_a-1}}^2\right)\\
		&= 2\ip{(\mu-1) g_{n_a-1} - g_{n-1}, x_{n_a-1}-x_{n-1}} - \mu \norm{g_{n_a-1}-g_{n-1}}^2\\
		&= 2\ip{(\mu-1)\Delta_1 - \Delta_3, \Delta_0 - \Delta_2} - \mu \norm{\Delta_1}^2
		- \mu \norm{\Delta_3}^2
		+ 2\mu \ip{\Delta_1,\Delta_3}.
	\end{align*}
	
	Thus, the weighted sum of these terms, i.e., $\Sigma_g$ is equal to
	\begin{align*}
		\Sigma_g &= \sum_{i=0}^{n_a - 2} h_i \left(\norm{g_i}^2 + 2\ip{g_i, x_0 - x_i}\right) + \sum_{i=n_a}^{n-2}h_i\left(\norm{g_i}^2+2\ip{g_{i}, x_0 - x_{i}}\right)\\
		&\qquad + \left(\text{a quadratic form in }\Delta_0,\Delta_1,\Delta_2,\Delta_3\right).
	\end{align*}
	The quadratic form corresponds to the matrix
	\begin{equation*}
		\begin{pmatrix}
			0 & -\mu & 0 & 0\\
			\cdot & \frac{-1+\alpha}{\alpha^2}+ \frac{(-1+\beta)\sigma}{\beta} + \mu(\mu-3\sigma+2\mu\sigma) & 0 & 0\\
			\cdot & \cdot & -1 & 0\\
			\cdot & \cdot & \cdot & \frac{-1 + \beta - 2\sigma +\beta(1+\beta -\beta\mu)\sigma}{\beta^2}
		\end{pmatrix}
	\end{equation*}
	where the entries below the diagonal are defined by symmetry.
	
	Applying the definitions of $\mu,\sigma,\eta$ we can check
	that
	the coefficient on $\norm{\Delta_1}^2$ is $\mu$ and the coefficient on $\norm{\Delta_3}^2$ is $-\frac{1-\eta}{\eta^2}$.
	
	We conclude that
	\begin{equation*}
		\Sigma_g = \sum_{i=0}^{n -2} h_i \left(\norm{g_i}^2 + 2\ip{g_i, x_0 - x_i}\right) - \norm{x_{n-1}-x_0}^2 - \frac{1-\eta}{\eta^2}\norm{g_{n-1}}^2.\qedhere
	\end{equation*}

\end{proof}

\subsection{Proof of \cref{prop:fg_rates_of_scomp}}
\label{subap:fgratesofscomp}
We break the proof of \cref{prop:fg_rates_of_scomp} into the two following lemmas.

\begin{lemma}
Suppose $h\in\R^{n}_{++}$ is \scomposable{} with rate $\eta$. Then, for any $1$-smooth convex function $f$ with minimizer $x_\star$, gradient descent with stepsizes $h$ satisfies
        \begin{equation*}
            f_n - f_\star \leq \frac{\eta}{2-\eta}\left(\frac{1}{2}\norm{x_0 - x_\star}^2 - \frac{1}{2}\norm{x_n - x_\star - \frac{g_n}{\eta}}^2\right).
        \end{equation*}
This bound is tight and is attained when $f(x)$ is either $q(x)$ or $H_{\eta/(2-\eta)}(x)$. In the latter case, the bracketed term simplifies to $\frac{1}{2}\norm{x_0-x_\star}^2$.
\end{lemma}
\begin{proof}
Let $f$ be a $1$-smooth function with minimizer $x_\star$. 
    For the first claim, it suffices to show that 
	\begin{align*}
        \frac{2(2-\eta)}{\eta}(f_\star - f_n)
        - \norm{x_n - x_\star - \frac{g_n}{\eta}}^2
        +\norm{x_0 - x_\star}^2 \ge 0.
	\end{align*}
    We do this by showing that
    \begin{align*}
        &\frac{2(2-\eta)}{\eta}(f_\star - f_n)
        - \norm{x_n - x_\star - \frac{g_n}{\eta}}^2
        +\norm{x_0 - x_\star}^2 = \\
        &\left(\sum_{i=0}^{n-1}h_i \left(2(f_i - f_n) + \norm{g_i}^2 + 2\ip{g_i, x_0 - x_i}\right) - \norm{x_n - x_0}^2 - \frac{1-\eta}{\eta^2} \norm{g_n}^2\right)\\
        &\qquad+ \left(\sum_{i=0}^{n-1}h_i Q_{\star,i} + \frac{1}{\eta} Q_{\star,n}\right).
    \end{align*}
    Note that this last expression is the sum of two terms: the first term is nonnegative by \cref{prop:self_comp_ineq} and the second is nonnegative because it is a nonnegative combination of the $Q_{i,j}$.
    It remains to show the previous equation, which we do in the following equation block.

	\begin{align*}
        &\sum_{i=0}^{n-1}h_i \left(2(f_i - f_n) + \norm{g_i}^2 + 2\ip{g_i, x_0 - x_i}\right) - \norm{x_n - x_0}^2 - \frac{1-\eta}{\eta^2} \norm{g_n}^2\\
         &\qquad + \sum_{i=0}^{n-1}h_i Q_{\star,i} + \frac{1}{\eta} Q_{\star,n}\\
		&= \sum_{i=0}^{n-1}h_i \left(2(f_\star - f_n) + 2\ip{g_i, x_0 - x_\star}\right) - \norm{x_n - x_0}^2 - \frac{1-\eta}{\eta^2} \norm{g_n}^2 + \frac{1}{\eta} Q_{\star,n}\\
		&= \left(2\sum_{i=0}^{n-1}h_i\right) (f_\star - f_n) + 2\ip{x_n - x_0, x_0 - x_\star} - \norm{x_n - x_0}^2 - \frac{1-\eta}{\eta^2} \norm{g_n}^2 + \frac{1}{\eta} Q_{\star,n}\\
		&= 
        \frac{2(1-\eta)}{\eta}(f_\star - f_n)
        - \norm{x_n - x_\star}^2
        +\norm{x_0 - x_\star}^2        
        - \frac{1-\eta}{\eta^2} \norm{g_n}^2 + \frac{1}{\eta} Q_{\star,n}\\
		&= 
        \frac{2(2-\eta)}{\eta}(f_\star - f_n)
        - \norm{x_n - x_\star}^2
        +\norm{x_0 - x_\star}^2        
        - \frac{1}{\eta^2} \norm{g_n}^2 - \frac{2}{\eta}\ip{g_n, x_\star - x_n}\\
        &=         \frac{2(2-\eta)}{\eta}(f_\star - f_n)
        - \norm{x_n - x_\star - \frac{g_n}{\eta}}^2
        +\norm{x_0 - x_\star}^2.
	\end{align*}
    Note that for the second equation above, we make use of the identity
        \[
         \sum_{i=0}^{n-1}h_i \ip{g_i, x_0 - x_\star} = 
         \sum_{i=0}^{n-1}\ip{h_ig_i, x_0 - x_\star} = 
         \ip{x_{n} - x_0, x_0 - x_\star}.
        \]
        In the third equation, we also make use of the identity that $\sum_{i=0}^{n-1} h_i = \frac{1-\eta}{\eta}$, which follows from the fact that $h$ is \scomposable{}. We also use that $2\ip{x_n - x_0, x_0 - x_\star} - \norm{x_n - x_0}^2 = \norm{x_0 - x_\star}^2 - \norm{x_n - x_\star}^2$, which is a direct algebraic manipulation.

    We now verify that this bound is tight when $f(x)$ is either $q(x)$ or $H_{\eta/(2-\eta)}(x)$.
    First, suppose $f(x)= q(x)$ and $x_0 = 1$. Then,
        $$f_n - f_\star = \frac{1}{2}x_n^2 = \frac{1}{2}\prod_{i=0}^{n-1}(1-h_i)^2x_0^2 = \frac{\eta^2}{2},$$
    where the last identity follows from the definition of \scomposable.
    On the other hand, recognizing that $g_n = x_n$ when $f(x)=q(x)$, we have that
    \begin{align*}
        \frac{\eta}{2-\eta}\left(\frac{1}{2}\norm{x_0-x_\star}^2 - \frac{1}{2}\norm{x_n - \frac{g_n}{\eta}- x_\star}^2
        \right) &= \frac{\eta}{2(2-\eta)}\left(1-\left(1 - \frac{1}{\eta}\right)^2x_n^2
        \right)\\
        &=\frac{\eta}{2(2-\eta)}\left(2\eta - \eta^2
        \right)\\
        &= \frac{\eta^2}{2}.
    \end{align*}
    Next, suppose $f(x)= H_{\eta/(2-\eta)}(x)$ and set $x_0 = 1$. Let $\Sigma = \sum_{i=0}^{n-1}h_i$.
    Then,
    $\eta = \frac{1}{1+\Sigma}$ and $\frac{\eta}{2-\eta}= \frac{1}{1+2\Sigma}$.
    Thus,
        $$\frac{1+\Sigma}{1+2\Sigma}=1 - \frac{\eta}{2-\eta} \Sigma \geq \frac{1}{1+\Sigma}=\eta.$$
    We deduce that $x_n$ is given by the quantity on the LHS of this inequality. We may now evaluate:
    \begin{align*}
        f_n - f_\star &= \frac{1}{1+2\Sigma}\frac{1+\Sigma}{1+2\Sigma} - \frac{1}{2(1+2\Sigma)^2}\\
        &= \frac{1}{2(1+2\Sigma)},
    \end{align*}
    and 
    \begin{align*}
        \frac{\eta}{2-\eta}\left(\frac{1}{2}\norm{x_0-x_\star}^2 - \frac{1}{2}\norm{x_n - \frac{g_n}{\eta}- x_\star}^2
        \right) &= \frac{1}{2(1+2\Sigma)}\left(1 - \left(\frac{1+\Sigma}{1+2\Sigma} - \frac{1+\Sigma}{1+2\Sigma}\right)^2
        \right)\\
        &= \frac{1}{2(1+2\Sigma)}.
    \end{align*}
    Finally, note that the second term within the bracketed quantity (corresponding to $\norm{x_n - g_n/\eta - x_\star}^2$) evaluates to zero as claimed.\qedhere
\end{proof}

\begin{lemma}
    Suppose $h\in\R^n_{++}$ is \scomposable{} with rate $\eta$. Then, for any $1$-smooth convex function $f$ with minimizer $x_\star$, gradient descent with stepsize $h$ satisfies
\begin{equation*}
    \frac{1}{2} \norm{g_n}^2\leq 
    \frac{\eta}{2-\eta}\left(f_0 - f_\star - \frac{1}{2}\norm{g_0 - \eta\sum_{i=0}^{n-1}h_ig_i - \eta g_n}^2\right).
\end{equation*}
This bound is tight and is attained when $f(x)$ is either $q(x)$ or $H_{\eta}(x)$. In the latter case, the bracketed term simplifies to $f_0-f_\star$.
\end{lemma}
\begin{proof}
    The three following expressions are each guaranteed to be nonnegative:
\begin{gather}
    2\sum_{i=0}^{n-1}h_i(f_i - f_n) - \sum_{i=0}^{n-1} (h_i^2-h_i)\norm{g_i}^2 - \frac{1-\eta}{\eta^2} \norm{g_n}^2,\label{eq:scomp_grate_1}\\
    \sum_{i=0}^{n-1}h_iQ_{0,i} + Q_{0,n},\qquad\text{and}\label{eq:scomp_grate_2}\\
    \sum_{i=0}^{n-1}h_iQ_{n,i} + Q_{n,n}.\label{eq:scomp_grate_3}
\end{gather}
Indeed, \eqref{eq:scomp_grate_1} is nonnegative by \cref{prop:self_comp_ineq} and \eqref{eq:scomp_grate_2} and \eqref{eq:scomp_grate_3} are nonnegative combinations of the $Q_{i,j}$.

Let $\Delta_0=g_0$, $\Delta_1=x_n-x_0$, $\Delta_2=g_n$.
The terms depending on $x_0 - x_\star,g_0,\dots,g_n$ in \eqref{eq:scomp_grate_2} are
\begin{align*}
    & \sum_{i=0}^{n-1}h_i\left(- 2\ip{g_i, x_0 - x_i} - \norm{g_0 - g_i}^2\right) + \left(- 2\ip{g_n, x_0 - x_n} - \norm{g_0 - g_n}^2\right)\\
    &\qquad =\sum_{i=0}^{n-1}h_i\left(- 2\ip{g_i, x_n - x_i} - \norm{g_i}^2
    \right)
    -\frac{1-\eta}{\eta}\norm{\Delta_0}^2 - 2\ip{\Delta_0,\Delta_1}
    -2\norm{\Delta_1}^2\\
    &\qquad\qquad + 2\ip{\Delta_1, \Delta_2} - \norm{\Delta_0 - \Delta_2}^2\\
    &\qquad = \sum_{i=0}^{n-1}h_i\left(- 2\ip{g_i, x_n - x_i} - \norm{g_i}^2
    \right)
    + \begin{pmatrix}
    -\frac{1}{\eta}&-1&1\\
    \cdot&-2&1\\
    \cdot&\cdot&-1
    \end{pmatrix}.
\end{align*}
Here, the $3\by 3$ matrix in the final line is shorthand for the quadratic form in $\Delta_0, \Delta_1,\Delta_2$ corresponding to this $3\by 3$ matrix with entries below the diagonal defined by symmetry.

The terms depending on $x_0 - x_\star,g_0,\dots,g_n$ in \eqref{eq:scomp_grate_3} are
\begin{align*}
    &\sum_{i=0}^{n-1}h_i\left(- 2\ip{g_i, x_n - x_i} - \norm{g_n - g_i}^2\right)\\
    &\qquad = \sum_{i=0}^{n-1}h_i\left(- 2\ip{g_i, x_n - x_i} - \norm{g_i}^2\right) - \frac{1-\eta}{\eta}\norm{\Delta_2}^2-2\ip{\Delta_1,\Delta_2}\\
    &\qquad = \sum_{i=0}^{n-1}h_i\left(- 2\ip{g_i, x_n - x_i} - \norm{g_i}^2\right) + \begin{pmatrix}
    0&0&0\\
    \cdot &0& -1\\
    \cdot &\cdot&\frac{-1}{\eta}+1
    \end{pmatrix}
\end{align*}

We now sum up the three expressions with weights $a=2\eta^2-\eta^3, b=\eta^2, c=\eta^2-\eta^3$.
Let $\Sigma$ denote this sum. 
Let $\Sigma_f$ and $\Sigma_g$ denote the terms that are linear in $f_i$ and that are quadratic in $x_0-x_\star,g_0,\dots,g_n$ respectively. Then,
\begin{align*}
    \frac{\Sigma_f}{2}&=a\sum_{i=0}^{n-1}h_i(f_i - f_n) + b\sum_{i=0}^{n-1}h_i(f_0 - f_i)+b(f_0 - f_n) + c\sum_{i=0}^{n-1}h_i(f_n - f_i)\\
    &\qquad = a\sum_{i=0}^{n-1}h_if_i - a\left(\sum_{i=0}^{n-1}h_i\right)f_n + b\left(\sum_{i=0}^{n-1}h_i\right)f_0 - b\sum_{i=0}^{n-1}h_if_i+b(f_0 - f_n)\\
    &\qquad\qquad + c\left(\sum_{i=0}^{n-1}h_i\right)f_n - c\sum_{i=0}^{n-1}h_if_i\\
    &\qquad = \sum_{i=0}^{n-1}(a-b-c)h_if_i + b\left(1+\sum_{i=0}^{n-1}h_i\right)f_0 + \left(- b + (c-a)\sum_{i=0}^{n-1}h_i\right)f_n\\
    &\qquad = \eta(f_0-f_n).
\end{align*}
Before we calculate $\Sigma_g$, we observe that
\begin{align*}
    &\sum_{i=0}^{n-1} (h_i^2-h_i)\norm{g_i}^2 + \sum_{i=0}^{n-1}h_i\left(2\ip{g_i, x_n - x_i} + \norm{g_i}^2
    \right)\\
    &\qquad = -\sum_{i=0}^{n-1}2\ip{h_ig_i, \sum_{j=i}^{n-1}h_jg_j} + \norm{h_ig_i}^2
= -\norm{\Delta_1}^2.
\end{align*}
This allows us to simplify $\Sigma_g$ as
\begin{align*}
\Sigma_g &= 
 -a \sum_{i=0}^{n-1} (h_i^2-h_i)\norm{g_i}^2 
 -(b+c)\sum_{i=0}^{n-1}h_i\left(2\ip{g_i, x_n - x_i}  + \norm{g_i}^2
    \right)\\
    &\qquad 
    - a\frac{1-\eta}{\eta^2} \norm{g_n}^2
    + b\begin{pmatrix}
    -\frac{1}{\eta}&-1&1\\
    \cdot&-2&1\\
    \cdot&\cdot&-1
    \end{pmatrix} + c\begin{pmatrix}
    0&0&0\\
    \cdot &0& -1\\
    \cdot &\cdot&\frac{-1}{\eta}+1
    \end{pmatrix}\\
    &= a\begin{pmatrix}
    0&0&0\\
    \cdot&1&0\\
    \cdot&\cdot&-\frac{1-\eta}{\eta^2}
    \end{pmatrix} + b\begin{pmatrix}
        -\frac{1}{\eta}&-1&1\\
        \cdot&-2&1\\
        \cdot&\cdot&-1
        \end{pmatrix} + c\begin{pmatrix}
            0&0&0\\
            \cdot&0& -1\\
            \cdot&\cdot&\frac{-1}{\eta}+1
            \end{pmatrix}\\
&= -\eta \norm{g_0 - \eta\sum_{i=0}^{n-1}h_ig_i - \eta g_n}^2 - 2(1-\eta)\norm{g_n}^2.
\end{align*}
We deduce that
\begin{equation*}
    2\eta(f_0 - f_n) - \eta\norm{g_0 - \eta\sum_{i=0}^{n-1}h_ig_i - \eta g_n}^2 -  2(1-\eta)\norm{g_n}^2\geq 0.
\end{equation*}
The desired inequality follows by adding $\eta Q_{n,\star}$ and rescaling.

It remains to show that this bound is tight when $f(x)$ is either $q(x)$ or $H_\eta(x)$. First suppose $f(x)=q(x)$ and $x_0=1$. Then,
\begin{align*}
    \abs{x_n} = \prod_{i=0}^{n-1}(h_i-1) = \eta.
\end{align*}
Additionally, $g_i = x_i$ for all $i=0,\dots,n$ due to $f(x)=q(x)$. Thus,
\begin{align*}
    \frac{1}{2}\norm{g_n}^2 = \frac{\eta^2}{2},
\end{align*}
and
\begin{align*}
   \frac{\eta}{2-\eta}\left(f_0-f_\star - \frac{1}{2}\norm{g_0 - \eta \sum_{i=0}^{n-1}h_i g_i - \eta g_n}^2\right)
   &=\frac{\eta}{2(2-\eta)}\left(1 - \left(x_0 - \eta (x_0 - x_n) - \eta x_n\right)^2\right)\\
   &= \frac{\eta}{2(2-\eta)}\left(1 - \left(1-\eta \right)^2\right)\\
   &= \frac{\eta^2}{2}.
\end{align*}
Here, we have used the identity $x_n = x_0 - \sum_{i=0}^{n-1}h_i g_i$.

Next, suppose $f(x)=H_\eta(x)$ and set $x_0 = 1$.
Note that
\begin{equation*}
    1 - \eta \sum_{i=0}^{n-1}h_i =\eta.
\end{equation*}
Thus, we have that $x_n = \eta$ and that $g_i = \eta$ for all $i=0,\dots,n$. We deduce that
\begin{equation*}
    \frac{1}{2}\norm{g_n}^2 = \frac{\eta^2}{2}.
\end{equation*}
On the other hand,
\begin{align*}
    &\frac{\eta}{2-\eta}\left(f_0-f_\star - \frac{1}{2}\norm{g_0 - \eta \sum_{i=0}^{n-1}h_i g_i - \eta g_n}^2\right)\\
    &=\frac{\eta^2}{2(2-\eta)}\left(2 - \eta - \left(1 - \eta \left(1+\sum_{i=0}^{n-1}h_i\right)\right)^2\right)\\
    &= \frac{\eta^2}{2(2-\eta)}(2 - \eta)\\
    &= \frac{\eta^2}{2}.
\end{align*}
Here, the third line uses the identity $\eta = \frac{1}{1+\sum_i h_i}$. Thus, the second term on the first line (corresponding to $\norm{g_0 - \eta\sum_ih_ig_i - \eta g_n}^2$) evaluates to zero as claimed.\qedhere
\end{proof}

\end{document}